\documentclass[bj,authoryear]{imsart}

\usepackage{apalike}
\usepackage{xcolor}
\usepackage{booktabs}
\usepackage{multirow}
\usepackage{float}

\startlocaldefs
\numberwithin{equation}{section}
\theoremstyle{plain}
\newtheorem{theo}{Theorem}[section]
\newtheorem{prop}{Proposition}[section]
\newtheorem{lem}{Lemma}[section]
\theoremstyle{definition}

\newtheorem{remark}{Remark}

\newcommand{\p}{\mathbb{P}}

\endlocaldefs

\begin{document}

\begin{frontmatter}
\title{Cram\'{e}r-type moderate deviation for double index permutation statistics}
\runtitle{Cram\'{e}r-type moderate deviation for double index permutation statistics}

\begin{aug}
\author[A]{\inits{SH Liu}\fnms{Song-hao}~\snm{Liu}\ead[label=e1]{liusonghao@dlut.edu.cn}}
\author[B,C]{\inits{QM Shao}\fnms{Qi-man}~\snm{Shao}\ead[label=e2]{shaoqm@sustech.edu.cn}}
\author[C]{\inits{JY Xu}\fnms{Jing-yu}~\snm{Xu}\ead[label=e3]{12131253@mail.sustech.edu.cn}}
\address[A]{School of Mathematical Sciences, Dalian University of Technology, Dalian, Liaoning, China\printead[presep={,\ }]{e1}}
\address[B]{Department of Statistics and Data Science, Shenzhen International Center for Mathematics, Southern University of Science and Technology, Shenzhen, Guangdong, 518055, China\printead[presep={,\ }]{e2}}
\address[C]{Department of Statistics and Data Science, Southern University of Science and Technology, Shenzhen, Guangdong, 518055, China\printead[presep={,\ }]{e3}}
\end{aug}

\begin{abstract}

We establish a Cram\'{e}r-type moderate deviation theorem for double-index permutation statistics (DIPS). To the best of our knowledge, previous results only provided Berry-Esseen type bounds for DIPS, which cannot yield moderate deviation results and are insufficient to capture the optimal convergence rates for some relatively sparse DIPS. Our result overcome these limitations: it not only recover the optimal convergence rates for classical DIPS, such as the Mann-Whitney-Wilcoxon statistic, but also extend to sparse statistics, including the number of descents in permutations and Chatterjee's rank correlation coefficient, for which previous approaches do not apply. To prove this result, we establish a Cram\'{e}r-type moderate deviation of normal approximation for bounded exchangeable pairs. Compared with existing results, our theorem requires more easily verifiable conditions.

\end{abstract}


\end{frontmatter}
\section{Introduction}
Let $\{\xi(i,j,k,l)\}_{i,j,k,l\in [n]}$ be a 4-index real number array. We are interested in the double-indexed permutation statistics (DIPS) of the general form 
\begin{align}
    \sum_{i,j}\xi(i,j,\pi(i),\pi(j)),\notag
\end{align}
where $\pi$ is a random permutation chosen uniformly from $S_n$ (the symmetric group of degree $n$). The DIPS of the restricted form $\sum_{i,j}a_{ij}b_{\pi(i)\pi(j)}$ was first investigated by \cite{daniels1944relation} in the study of the generalized correlation coefficient, with Kendall's $\tau$ and Spearman's $\rho$ being special cases. Daniels gave a set of sufficient conditions for their asymptotic normality as $n\to \infty$. Later, various results under weakened conditions were introduced in \cite{Bloemena1976SamplingFA}, \cite{jogdeo1968asymptotic}, \cite{abe1969central}, \cite{shapiro1979asymptotic}, \cite{barbour1986random} and \cite{pham1989asymptotic}. The use of DIPS has diversely been suggested by \cite{friedman1979multivariate} (this paper is the start of graph-based tests), \cite{friedman1983graph}, \cite{schilling1986multivariate} and \cite{vigna2015weighted} in graph-based tests, by \cite{hubert1976quadratic} in clustering studies, by \cite{cliff1981spatial} in geography, by \cite{chatterjee2021new} and \cite{shi2022distribution} in testing dependence.

The asymptotic properties of DIPS have also been widely studied by people. Using Stein's method, \cite{zhao1997error} proved a Berry-Esseen type theorem for general form DIPS. However their results do not apply to statistics such as the number of descents, the number of inversions of permutation, or Chatterjee's rank correlation coefficient, which appear to be "too sparse". In contrast, by constructing a special exchangeable pair, \cite{fulman2004stein} obtained convergence rate of order $n^{-1/2}$ in the Kolmogorov metric for both the number of descents and the number of inversions of permutation. Unfortunately, their method cannot be applied to general form DIPS. Later, \cite{fang2015rates} extended \cite{fulman2004stein}'s result to the multivariate setting and remove a certain condition arising from the requirement of exchangeability. Nevertheless, both \cite{fulman2004stein} and \cite{fang2015rates} results address only certain special cases of DIPS and are not applicable to the general setting. Moreover, none of the existing works provide Cram\'{e}r-type moderate deviation results for DIPS.

Berry-Esseen bound describe the absolute error for distributional approximation while the Cram\'{e}r-type moderate deviation describes the relative error. More precisely, let $\{Y_i\}_{i = 1}^{n}$ be a sequence of random variables that converge to $Y$ in distribution, the Cram\'{e}r-type moderate deviation is 
\begin{align}
    \frac{\mathbb{P}(Y_n>x)}{\mathbb{P}(Y>x)} = 1+\text{error term} \to 1,\notag
\end{align}
for $0\leq x\leq a_n$, where $a_n\to \infty$. Specially, for the normalized sum of i.i.d. random variables with finite moment generating functions, the range $0\leq x\leq n^{1/6}$ and the order of the error term $n^{-1/2}(1+x^3)$ are optimal, refer to \cite{petrov2012sums} for details. 

The purpose of this paper is to establish a Cram\'{e}r-type moderate deviation theorem for double-indexed permutation statistics (DIPS) of the general form $\sum_{i,j}\xi(i,j,\pi(i),\pi(j))$ with an optimal convergence rate. We hope our results can not only be used to yield the optimal convergence rate for some well-known statistics, such as Kendall's $\tau$, Spearman's $\rho$ and the Mann-Whitney-Wilcoxon statistic, but also for some 'sparse' statistics such as the number of descents and the number of inversions of permutations and Chatterjee's rank correlation coefficient. To achieve this goal, we apply Stein's method and the exchangeable pair approach to the above statistics.

Stein's method was first introduced by \cite{stein1972bound}, an introduction and a survey of Stein's method can be found in \cite{chen2010normal}. The exchangeable pair approach of Stein's method is a powerful tool for estimating the convergence rates for distributional approximation. \cite{chen2013stein} developed the method to prove Cram\'{e}r-type moderate deviation results in normal approximation without skewness correction for dependent random variables under a boundedness condition. \cite{chen2013moderate} and \cite{shao2021cramer} considered Poisson approximation and nonnormal approximations, respectively. \cite{zhang2023cramer} refined the results in \cite{chen2013stein} by relaxing the boundedness condition. 

For exchangeable pair approach, let $W$ be the random variable of interest, and we say $(W,W^{\prime})$ an exchangeable pair if $(W,W^{\prime})\stackrel{d}{=} (W^{\prime},W)$. Let $\Delta = W - W^{\prime}$. It is often to assume that (see, e.g., \cite{rinott1997coupling}) there exits a constant $\lambda >0$ and a random variable $R$ such that
\begin{align}
    \mathbb{E}\{\Delta|W\} = \lambda(W+R).\notag
\end{align}
but unfortunately, this assumption is not satisfied for general DIPS. \cite{zhang2023cramer} refined this assumption, assuming that there exits $D:= \Psi(W,W^{\prime})$ which is an antisymmetric function satisfying $\mathbb{E}(D \mid W)=\lambda(W+R)$. So we can construct a suitable $D$ for general DIPS and use the exchangeable pair approach to derive a Cram\'{e}r-type moderate deviation.


By employing the powerful tools mentioned above, we establish a Cram\'{e}r-type moderate deviation result for bounded exchangeable pairs in Theorem \ref{thm-general-MD}, which can be viewed as a special case of the result obtained in \cite{zhang2023cramer}. The key distinction is that we optimize the proof so that our theorem removes a technical condition in \cite{zhang2023cramer} which is not easy to verify in practice. Building on this result, we further derive a Cram\'{e}r-type moderate deviation result for doubly-indexed permutation statistics in Theorem \ref{thm-general-double-index-permutation-statistics-a-b-form}. This theorem overcomes the limits of former results, it applies not only to classical DIPS but also to relatively sparse ones, and in both cases we are able to achieve the optimal convergence rate.

The rest of this paper is organized as follows. In Section 2, we present our main result Theorem \ref{thm-general-double-index-permutation-statistics-a-b-form}, a Cram\'{e}r-type moderate deviation theorem for double index permutation statistics. In Section 3, we provide applications of our results to some well-known statistics such as Mann-Whitney-Wilcoxon statistic, and some other "sparse" statistics such as the number of descents and inversions of permutations and Chatterjee's rank correlation coefficient. In Section 4, we prove the general bound and the application.

\section{Main Results}
Let $\{\xi(i,j,k,l)\}_{i,j,k,l\in[N]}$ be real numbers, and the double indexed permutation statistic (DIPS) is defined as
\begin{align}\label{general_DIPS}
    DIPS = \sum_{i,j = 1}^{n}\xi(i,j,\pi(i),\pi(j)),
\end{align} 
where $\pi$ is a random permutation chosen uniformly from $S_{n}$ (symmetric group of degree $n$). Inspired by the proof of the Combinatorial Central Limit Theorem, \cite{zhao1997error} converted the general form of DIPS (\ref{general_DIPS}) to the form of $\sum_{i = 1}^{n}a(i,\pi(i))+\sum_{i,j}^{\prime}b(i,j,\pi(i),\pi(j))$, \footnote{Throughout this paper, $\sum_{i,j}^{\prime}$ denotes $\sum_{i,j,i\neq j}$.} where $\{a(i,k)\}_{i,k\in[N]}$ is a real matrix and $\{b(i,j,k,l)\}_{i,j,k,l\in[N]}$ is a 4-index real array. We also use this normalized form of DIPS and consider some boundedness conditions that hold true for most cases. For the sake of simplicity in writing, some notations we use in this article are as follows
\begin{align}
    a(i,\cdot) = \frac{1}{n}\sum_{k}a(i,j),\quad &a(\cdot,\cdot) = \frac{1}{n^2}\sum_{i,k}a(i,k),\notag \\
    b(i,j,k,\cdot) = \frac{1}{n}\sum_{l}b(i,j,k,l), \quad &b(i,j,\cdot,\cdot) = \frac{1}{n^2}\sum_{k,l}b(i,j,k,l),\notag\\
    b(i,\cdot,\cdot,\cdot) = \frac{1}{n^3}\sum_{j,k,l}b(i,j,k,l),\quad &b(\cdot,\cdot,\cdot,\cdot) = \frac{1}{n^4}\sum_{i,j,k,l}b(i,j,k,l),\notag
\end{align}

\begin{prop}\label{general-DIPS-convert}
By a suitable normalization, the general DIPS (\ref{general_DIPS}) can be converted into one of the following forms
\begin{align}
    W_n =& \sum_{i}a(i,\pi(i))+{\sum_{i,j}}^{\prime}b(i,j,\pi(i),\pi(j)), \label{normalized-DIPS-a-b-form} 
\end{align}
where $\{a(i,k)\}_{i,k\in[N]}$ is a real number matrix and $\{b(i,j,k,l)\}_{i,j,k,l\in[N]}$ is a 4-index real number array, and $\{a(i,k)\}_{i,k\in[N]}$ satisfies
\begin{equation}\label{a=0 or not}
    a(i,k)\equiv0 \quad \text{ or } \quad \sum_{i,j}a^2(i,k) = n-1,
\end{equation}
and
\begin{align}
    a(i,\cdot) = a(\cdot,k) = 0, \quad \label{equ-a-condition-origin}
    \end{align}
    and $\{b(i,j,k,l)\}_{i,j,k,l\in[N]}$ satisfies
    \begin{align}
        b(i,j,k\cdot) = b(i,j,\cdot,l) = b(i,\cdot,k,l) = b(\cdot,j,k,l) = 0, \label{equ-b-condition-origin}
\end{align}
no matter what $\{a(i,k)\}_{i,k\in[N]}$ is.
\end{prop}

In earlier work, \cite{zhao1997error} derived the Berry-Esseen bound for DIPS, but their result has notable limitations. First, it only applies to the case where $\{a(i,k)\}_{i,k\in[n]} \neq 0$ in definition (\ref{normalized-DIPS-a-b-form}); when $\{a(i,k)\}_{i,k\in[n]} \equiv 0$, the result does not yield a convergence rate. However, in practice, there exits important double-index permutation statistics satisfying $\{a(i,k)\}_{i,k\in[n]} \equiv 0$, such as Chatterjee's rank correlation coefficient, which is quite popular recently for detecting the independence between random variables. Moreover, even in the case $\{a(i,k)\}_{i,k\in[n]} \neq 0$, the result in \cite{zhao1997error} is still limited: for certain statistics, including the number of descents (Des) and the number of inversions (Inv) of permutation, it fails to provide the optimal convergence rate. Although \cite{fulman2004stein} provided a Berry-Esseen bound result for both Des and Inv, there result is only suitable for specific cases $\sum_{i,j}{\bf 1}\{i<j\}M_{\pi(i),\pi(j)}$, where $M = (M_{i,j})$ be a real, antisymmetric $n*n$ matrix, and cannot be applied to the general DIPS.

The purpose of our result is to present the convergence rate of the normal approximation of a general double index permutation statistics and be able to overcome the limitations of previous results. In Theorem \ref{thm-general-double-index-permutation-statistics-a-b-form}, we give a Cram\'{e}r-type moderate deviation for the DIPS under some boundedness conditions.

\begin{theo}\label{thm-general-double-index-permutation-statistics-a-b-form}
For double index permutation statistics $W_n$ defined as (\ref{normalized-DIPS-a-b-form}) in Proposition \ref{general-DIPS-convert}, assume that for some constant $0<\delta<1$, the following boundedness conditions hold
\begin{equation}
\begin{aligned}
    &\max_{i,k}|a(i,k)|\leq \delta, \quad \max_{i,\pi}\sum_{j}\left|b(i,j,\pi(i),\pi(j))\right| \leq \delta,\\
    &\max_{i,j,k,l}|b(i,j,k,l)| \leq \delta,\quad \max_{\substack{s\in\{1,2\}\\ t\in\{1,2\}}}\max_{\substack{i_{s}\in[n]\\k_{t}\in[n]}}\sum_{\substack{i_{3-t}\in[n]\\j_{3-t}\in[n]}}|b(i_{1},i_{2},k_{1},k_{2})|\leq \delta.
\end{aligned}
\label{equ-boundedness-condition-DIPS}
\end{equation}

Without loss of generality, assume that $\mathbb{E}\{W_n^2\} = 1 + \frac{C}{\sqrt{n}}$, where $C$ is a constant. For any $\theta >0$, let $\tau(\theta) := \max\{0\leq t \leq 1/\delta: t^3\delta + \sqrt{n}\delta^3t^2 + n\delta^3t^3 + t\delta + t^2/n \leq \theta \}$. Then for any $0\leq z\leq \tau(\theta)$,
\begin{align}
    \left|\frac{\mathbb{P}(W_n>z)}{1-\Phi(z)}-1\right|\leq C_1e^{\theta}(1+z^2)(\sqrt{n}\delta^2 + n\delta^3 +n\delta^3z+\delta)
\end{align}
\end{theo}

\begin{remark}
    The boundedness conditions in (\ref{equ-boundedness-condition-DIPS}) are crucial for establishing the optimal convergence rate in Theorem \ref{thm-general-double-index-permutation-statistics-a-b-form}. They ensure that the contributions from the various components of the double index permutation statistics are controlled, allowing for precise asymptotic analysis. Although the boundedness condition may appear complicated, it is in fact satisfied by the vast majority of doubly-indexed permutation statistics, including Kendall's $\tau$, Spearman's $\rho$, the Mann-Whitney-Wilcoxon statistic, and Chatterjee's rank correlation coefficient. If $\delta = O(1/\sqrt{n})$, by Theorem \ref{thm-general-double-index-permutation-statistics-a-b-form}, we are able to obtain the optimal convergence rate.
    \begin{align}
        \left|\frac{\mathbb{P}(W_n>z)}{1-\Phi(z)}-1\right| = O(1)(1+z^3)/\sqrt{n}, \quad 0\leq z\leq n^{1/6}.
    \end{align}
\end{remark}

\begin{remark}
    In contrast, without such restrictions, the convergence rate obtained in our theorem inevitably includes an extra term of the form $z^3\sum_{i,j,k,l}|b(i,j,k,l)|^3$. A similar term $\sum_{i,j,k,l}|b(i,j,k,l)|^3$ also appears in the Berry-Esseen bound of \cite{zhao1997error}, which slows down the rate of convergence. For many classical double index permutation statistics, such as Kendall's $\tau$, Spearman's $\rho$, etc., we have $\max_{i, j, k, l}|b(i, j, k, l)|=O(1/n^{3/2})$, so that $\sum_{i,j,k,l}|b(i,j,k,l)|^3$ still yields the optimal convergence rate. However, for some relatively "sparse" statistics, such as the number of descents and inversions in permutation, Chatterjee's rank correlation coefficient, etc., we only have $\max_{i,j,k,l}|b(i,j,k,l)| = O(1/n^{1/2})$. This arises because indicator function always appears in the definition of such special statistics. For example, $b(i,j,k,l) = {\bf 1}\{i = j+1\}O(1/n^{1/2})$. In these cases, the term $\sum_{i,j,k,l}|b(i,j,k,l)|^3$ does not even tend towards 0, and thus the results of previous studies are not applicable. By contrast, our theorem ensures that the proposed boundedness condition is still satisfied in this setting, allowing us to achieve the optimal convergence rate even for such sparse statistics.
\end{remark}

\section{Applications} 
In this section, the application of Theorem \ref{thm-general-double-index-permutation-statistics-a-b-form} is demonstrated by three examples. In addition to those well-known test statistics, Theorem \ref{thm-general-double-index-permutation-statistics-a-b-form} also applies to some relatively "sparse" statistics such as the number of descents and inversions of a permutation, and the recently very popular statistic Chatterjee's rank correlation coefficient for independence testing.

\subsection{The Chatterjee's rank correlation coefficient}
\cite{chatterjee2021new} introduced a novel and concise rank-based statistic that has recently gained considerable attention. Unlike traditional measures, this statistic corresponds to a population quantity proposed by \cite{dette2013copula} that characterizes independence between random variables, and moreover, its distribution under independence can be described by an asymptotic normal law.  These characteristics make Chatterjee's coefficient a particularly appealing tool for both theoretical investigation and practical applications in dependence modeling. Moreover, the Chatterjee's rank correlation is a statistic based on rank and can be expressed as a double index permutation statistic (\ref{general_DIPS}). Therefore, we can use Theorem \ref{thm-general-double-index-permutation-statistics-a-b-form} to obtain its optimal convergence rate.

Let $(X,Y)$ be a pair of random variables, $Y$ is not a constant. Let $(X_1,Y_1),\dots,(X_n,Y_n)$ be i.i.d. pairs with the same law as $(X,Y)$, where $n\geq 2$. Suppose that the $X_{i}$'s and $Y_{i}$'s have no ties. Rearrange the data as $(X_{(1)},Y_{(1)}),\dots,(X_{(n)},Y_{(n)})$ such that $X_{(1)}\leq \dots \leq X_{(n)}$. Let $r_i$ be the rank of $Y_{(i)}$, that is, the number of $j$ such that $Y_{(j)}\leq Y_{(i)}$. The Chatterjee's rank correlation coefficient is defined as 
\begin{align}
    W = \sqrt{\frac{5n}{2}}\left(1-\frac{3\sum_{i = 1}^{n-1}|r_{i+1}-r_{i}|}{n^2-1}\right).\label{Chatterjee-rank-correlation-coefficient}
\end{align}
The asymptotic property of this statistic is a corolla of the theorem in \cite{chao1993asymptotic}, but its convergence rate has never reached an optimal result. The following theorem will provide the optimal convergence rate.
\begin{theo}\label{theo-chatterjee-rank-correlation-coefficient}
Let $W$ defined as (\ref{Chatterjee-rank-correlation-coefficient}), we have
\begin{align}
    \left|\frac{\mathbb{P}(W>z)}{1-\Phi(z)}-1\right| = O(1)(1+z^3)/\sqrt{n},\label{Chatterjee-moderate-deviation}
\end{align}
for $0\leq z \leq n^{1/6}$.
\end{theo}
\begin{remark}
\cite{chatterjee2021new} mentioned that the asymptotic property of the statistic (\ref{Chatterjee-rank-correlation-coefficient}) is essentially a restatement of the main theorem of \cite{chao1993asymptotic}. They consider an estimator called "oscillation of permutation" defined as
\begin{align}
    \sum_{i = 1}^{n-1}|\pi(i) - \pi(i+1)|.\notag
\end{align}
The Chatterjee's rank correlation coefficient is normalized "oscillation of permutation". However, \cite{chao1993asymptotic} merely presents the asymptotic property and does not provide the convergence rate. In the subsequent article \cite{chao1996estimating} give a Berry-Esseen bound of $\sum_{i = 1}^{n}\alpha_{\pi(i)\pi(i+1)}$. Let $\alpha_{ij} = |i-j|$, it is the "oscillation of permutation" that is namely the Chatterjee's rank correlation coefficient. However, when the result in \cite{chao1996estimating} is applied to the Chatterjee rank correlation, the optimal convergence rate cannot be achieved. Therefore, our result should be the first to present the optimal convergence rate of oscillation of permutation, namely the Chatterjee rank correlation coefficient.
\end{remark}

\subsection{The Number of descents and inversions of permutation}
Let $M$ be a real $n\times n $ matrix and assume that $M$ is anti-symmetric, that is for each $u,v\in \{1,\dots,n\}$, we have $M_{uv} = -M_{vu}$. Note that $M_{uu} = 0$. Let $\pi$ be a permutation of size $n$, chosen uniformly from $S_n$, and consider the statistic
\begin{align}
    W = \sum_{\substack{i,j\\i<j}}M_{\pi(i)\pi(j)}.\notag
\end{align} 
This permutation statistic was considered by many early works such as \cite{fulman2004stein}, \cite{fang2015rates}, and it is a special case of doubly-indexed permutation statistics
\begin{align}
    W = \sum_{i,j}{\bf 1}\{i<j\}M_{\pi(i)\pi(j)}.\label{eq-general-des-inver}
\end{align}
The reason to study (\ref{eq-general-des-inver}) is that two important properties of permutations, the number of descents and inversions, can be represented in this form. Choosing $M_{u,u+1} = -1$ and $M_{uv} = 0$ for all other $v>u$ (for $v<u$, $M_{uv}$ is defined via anti-symmetry), (\ref{eq-general-des-inver}) becomes  $2\text{Des}(\pi^{-1})-(n-1)$, where $\text{Des}(\pi)$ is the number of descents of $\pi$; with $M_{un} = -1$ for all $u<v$, (\ref{eq-general-des-inver}) becomes $2\text{Inv}(\pi^{-1})-C_{n}^{2}$, where $\text{Inv}(\pi)$ is the number of inversions of $\pi$. By using Stein's method, \cite{zhao1997error} prove a general Berry -Esseen type theorem for double indexed permutation statistics, but their results do not apply to the number of descents $\text{Des}(\pi)$, which seems to be "too sparse". In contrast, using a special exchangeable pair, \cite{fulman2004stein} was able to obtain a rate of convergence of $n^{-1/2}$ for the Kolmogorov metric for both, the number of descents and inversions. In contrast, our results are superior, and we can not only obtain the optimal convergence rate of those well-known classical statistics, but also achieve the optimal results for this "sparse" statistic. In the following theorem we consider the normalized number of descents and inversions of a random permutation
\begin{align}
    W = \frac{\text{Des} - (n-1)/2}{\sqrt{(n+1)/6}}, \quad T = \frac{\text{Inv} - C_{n}^{2} /2}{\sqrt{n(n-1)(2n+5)/72}}.\label{normalized-descent-inversion}
\end{align}
\begin{theo}\label{theo-descents-inversions}
Let $\text{Des}(\pi)$ and $\text{Inv}(\pi)$ be the number of descents and inversions of a random permutation $\pi\in S_{n}$, $W$ and $T$ are normalized statistics defined as (\ref{normalized-descent-inversion}). Then we have 
\begin{align}
    \left|\frac{\mathbb{P}\left(W>z\right)}{1-\Phi(z)}-1\right| = O(1)(1+z^3)/\sqrt{n},\label{descent-moderate-deviation}
\end{align}
\begin{align}
    \left|\frac{\mathbb{P}\left(T>z\right)}{1-\Phi(z)}-1\right| = O(1)(1+z^3)/\sqrt{n},\label{inversion-moderate-deviation}
\end{align}
for $0\leq z \leq n^{1/6}$.
\end{theo}

\subsection{The Mann-Whitney-Wilcoxon statistic}
The Mann-Whitney-Wilcoxon statistic is one of the members of U-statistics of degree two. Let $x_{1},\dots,x_{n_1}$, and $y_{1},\dots,y_{n_2}$, $n_{1}+n_{2} = n$,be independent univariate random samples from unknown continuous distributions $F_{X}$ and $F_{Y}$, respectively. The Mann-Whitney-Wilcoxon statistic for testing the hypothesis $H_0: F_{X} = F_{Y}$ is defined to be the total number of pairs $(x_{i},y_{i})$ for which $x_{i}<y_{i}$. Let $\pi(i),i = 1,\dots,n_{1}$, denote the rank of $x_{i}$ and $\pi(n_1+j), j = 1,\dots,n_{2}$, denote that of $y_{j}$ in the combined sample. Then the Mann-Wilcoxon-Wilcoxon statistic can be expressed as 
\begin{align}
    \sum_{i,j}\xi(i,j,\pi(i),\pi(j)),\notag
\end{align}
where
\begin{align}
    \xi(i,j,\pi(i),\pi(j)) = {\bf 1}\{1\leq i\leq n_{1},n_{1}+1\leq j\leq n,1\leq \pi(i)<\pi(j)\leq n\},\notag  
\end{align}
and $\pi$ is chosen uniformly from $S_{n}$ under $H_{0}$. Consider the statistic
\begin{align}
    W = \frac{\sum_{i,j}\xi(i,j,\pi(i),\pi(j))-n_1n_2/2}{\left(n_1n_2(n+1)/12\right)^{1/2}},\label{statistic-mann-whitney-wilcoxon}
\end{align}
we have the following theorem.
\begin{theo}\label{theo-mann-whitney-wilcoxon}
Assume $W$ is defined as (\ref{statistic-mann-whitney-wilcoxon}), we have
\begin{align}
     \left|\frac{\mathbb{P}\left(W>z\right)}{1-\Phi(z)}-1\right| = O(1)(1+z^3)/(1/n_{1}+1/n_{2})^{1/2},\label{mann-whitney-wilcoxon-moderate-deviation}
\end{align}
for $0\leq z \leq \min(n_{1},n_{2})^{1/6}$.
\end{theo}

\section{Proof of main results}
In this section, we give the proof of Theorem \ref{thm-general-double-index-permutation-statistics-a-b-form}. In subsection \ref{decomposition-of-general-DIPS}, we prove the Proposition \ref{general-DIPS-convert}. In subsection \ref{Cramer type moderate deviation result for bounded exchangeable pairs}, we give Theoren \ref{thm-general-MD}, a Cram\'{e}r-type moderate deviation result for bounded exchangeable pairs, which is the main tool to prove Theorem \ref{thm-general-double-index-permutation-statistics-a-b-form}. In subsection \ref{proof-of-theorem-general-double-index-permutation-statistics-a-b-form}, we give Lemma \ref{lem-same-distribution-of-permutation} and Lemma \ref{lem-E(xetW)}, the proof of Theorem \ref{thm-general-double-index-permutation-statistics-a-b-form} follows a combination of Theorem \ref{thm-general-MD} and two Lemmas,the details of the proof are put in the subsection \ref{proof-of-theorem-general-double-index-permutation-statistics-a-b-form}. We put the proof of Theorem \ref{thm-general-MD} in subsection \ref{proof-of-theorem-general-MD-section}. The proof of Theorem \ref{thm-general-MD} is based on Stein's method and the exchangeable pair approach. We give two propositions \ref{prop-moment-generating-function}, \ref{prop-general-moderate-deviation} in subsection \ref{proof-of-theorem-general-MD-section}. the proof of Theorem \ref{thm-general-MD} follows from a combination of Proposition \ref{prop-moment-generating-function} and Proposition \ref{prop-general-moderate-deviation}. 

\subsection{A decomposition of general DIPS}\label{decomposition-of-general-DIPS}
Before giving the proof we briefly recall the notation used above. The double–indexed permutation statistic (DIPS) is defined in (\ref{general_DIPS}), where $\pi$ is a random permutation chosen uniformly from $S_n$. As explained in the introduction (and following \cite{zhao1997error}), we aim to reduce the general kernel $\xi(i,j,k,l)$ to the normalized form $W_n=\sum_{i}a(i,\pi(i))+\sum_{i,j}' b(i,j,\pi(i),\pi(j))$, with $\sum_{i,j}'$ denoting summation over $i\neq j$. For convenience we use the marginal–averaging notation introduced earlier, e.g. $a(i,\cdot)=\frac{1}{n}\sum_k a(i,k)$, $b(i,j,k,\cdot)=\frac{1}{n}\sum_l b(i,j,k,l)$, and similarly for higher–order averages. The proposition asserts that, after an appropriate centering and scaling, the matrix $a(i,k)$ can be chosen to satisfy either $a(i,k)\equiv0$ or $\sum_{i,k}a^2(i,k)=n-1$, together with the marginal conditions $a(i,\cdot)=a(\cdot,k)=0$; likewise the four–way array $b(i,j,k,l)$ may be taken to have all one–dimensional and two–dimensional marginals equal to zero (cf. (\ref{equ-a-condition-origin}) and (\ref{equ-b-condition-origin})). The proof proceeds by successive marginal–centering of the kernel $\xi$; for notational convenience we denote the fully centered kernel by $\xi^*$ as in (\ref{decomposition-of-general-DIPs-1}). We now turn to the verification of Proposition \ref{general-DIPS-convert}.
\begin{proof}[Proof of Proposition \ref{general-DIPS-convert}]
For general DIPS (\ref{general_DIPS}), let 
\begin{align}\label{decomposition-of-general-DIPs-1}
    \xi^{*}(i,j,k,l) =& \xi(i,j,k,l)-[\xi(i,j,k,\cdot)+\xi(i,j,\cdot,l)+\xi(i,\cdot,k,l)+\xi(\cdot,j,k,l)]\notag\\
    &+[\xi(i,j,\cdot,\cdot)+\xi(i,\cdot,k,\cdot)+\xi(i,\cdot,\cdot,l)+\xi(\cdot,j,k,\cdot)+\xi(\cdot,j,
    \cdot,l)+\xi(\cdot,\cdot,k,l)]\notag \\
    &- [\xi(i,\cdot,\cdot,\cdot)+\xi(\cdot,j,\cdot,\cdot)+\xi(\cdot,\cdot,k,\cdot)+\xi(\cdot,\cdot,\cdot,l)]\notag\\
    &+\xi(\cdot,\cdot,\cdot,\cdot),
\end{align}
then we have 
\begin{align}
    \xi^{*}(i,j,k,\cdot) = \xi^{*}(i,j,\cdot,l) = \xi^{*}(i,\cdot,k,l) = \xi^{*}(\cdot,j,k,l) =0, \notag
\end{align}
and 
\begin{align}\label{decomposition-of-general-DIPs-2}
    \sum_{i,j}\xi(i,j,\pi(i),\pi(j))=&\sum_{i,j}\xi^{*}(i,j,\pi(i),\pi(j))+n\sum_{i}\xi(i\cdot,\pi(i),\cdot)+ n\sum_{j}\xi(\cdot,j,\cdot,\pi(j))-n^2\xi(\cdot,\cdot,\cdot,\cdot)\notag \\
    =& \sum_{i,j}^{\prime}\xi^{*}(i,j,\pi(i),\pi(j))+\sum_{i}\eta^{*}(i,\pi(i))+n\eta(\cdot,\cdot),
\end{align}
where
\begin{align}\label{decomposition-of-general-DIPs-3}
    \eta(i,k) &= \xi^{*}(i,i,k,k)+n\xi(i,\cdot,k,\cdot)+n\xi(\cdot,i,\cdot,k)-n\xi(\cdot,\cdot,\cdot,\cdot),\notag \\
    \eta^{*}(i,k) &= \eta(i,k) - \eta(i,\cdot) - \eta(\cdot,k) + \eta(\cdot,\cdot), \notag\\
    \eta^{*}(i,\cdot) &= \eta^{*}(\cdot,k) = 0. 
\end{align}
If $\eta^{*}(i,k)\neq 0$, we define
$\sigma^2 = \sum_{i,k}\eta^{*2}(i,k)/(n-1)$, and the normalized DIPS is defined as
\begin{align}
    W_n = & \frac{D-n\eta(\cdot,\cdot)}{\sigma} =  \sum_{i}\frac{1}{\sigma}\eta^{*}(i,\pi(i))+{\sum_{i,j}}^{\prime}\frac{1}{\sigma}\xi^{*}(i,j,\pi(i),\pi(j))\notag\\
    := & \sum_{i}a(i,\pi(i))+{\sum_{i,j}}^{\prime}b(i,j,\pi(i),\pi(j)).\notag
\end{align}
If $\eta^{*}(i,k) = 0$ for all $i,k\in [N]$, we define the normalized DIPS as
\begin{align}
    W_n = & D- n\eta(\cdot,\cdot) = {\sum_{i,j }}^{\prime}\xi^{*}(i,j,\pi(i),\pi(j))\notag\\
    := & {\sum_{i,j}}^{\prime}b(i,j,\pi(i),\pi(j)).\notag
\end{align}
This completes the proof of Proposition \ref{general-DIPS-convert}.
\end{proof}

\subsection{Cram\'{e}r-type moderate deviation for bounded exchangeable pairs}\label{Cramer type moderate deviation result for bounded exchangeable pairs}
To derive the Cram\'{e}r-type moderate deviation for double index permutation statistics, we firstly provide a Cram\'{e}r-type moderate deviation result for bounded exchangeable pairs. Let $(X,X^{\prime})$ be an exchangeable pair, $X$ is $\mathcal{F}$-measurable and valued on a measurable space $\mathcal{X}$, let $W$ be an $\mathbb{R}$-valued random variable of interest. We consider the following condition:
\begin{itemize}
    \item [$(D_{1}):$] Let $(X,X^{\prime})$ be an exchangeable pair. Assume that there exits $D := \Psi(X,X^{\prime})$, where $\Psi:\mathcal{X}\times\mathcal{X}\to \mathbb{R}$ is an antisymmetric function, satisfying that $\mathbb{E}(D \mid \mathcal{F})=\lambda(W+R)$ for some constant $\lambda>0$ and some random variable $R$ which is measurable with respect to $\mathcal{F}$.
\end{itemize}
\begin{theo}
\label{thm-general-MD}
Let $(X,X^{\prime})$ be an exchangeable pair satisfying the condition (D1), and $(W,W^{\prime})$ is also an exchangeable pair, $\Delta = W-W^{\prime}$. Let $\max\{|\Delta|,|D|\}\leq \delta$ for some constant $\delta>0$. Assume that there exits a constant $\tau>0$ such that 
\begin{itemize}
    \item [$(A_{1}):$] $\mathbb{E}\{e^{tW}\}<\infty$,
    \item [$(A_{2}):$] $\mathbb{E}\{\big|1-\frac{1}{2\lambda}\mathbb{E}\{D\Delta|W\}\big|e^{tW}\}\leq \delta_1(t)\mathbb{E}\{e^{tW}\}$,
    \item [$(A_{3}):$] $\mathbb{E}\{|R|e^{tW}\}\leq \delta_2(t)\mathbb{E}\{e^{tW}\}$,
\end{itemize}
where for each $j =1,2$, the fucntion $\delta_{j}(\cdot)$ is increasing and satisfies that $0\leq \delta_{j}(\tau)<\infty$. For $\theta>0$, let $\tau_{0}(\theta) := \max\{0\leq t\leq \min\{\tau,1/\delta\}:t^2(t\delta+2\delta_{1}(t))/2+3t\delta_{2}(t)\leq \theta\}$. Then for any $\theta>0$, $0\leq z\leq \tau_{0}(\theta)$,
\begin{align}
    \left|\frac{\mathbb{P}(W>z)}{1-\Phi(z)}-1\right|\leq 31e^{\theta}(1+9\delta)\left\{(1+z^2)[\delta_1(z)+\delta+\delta\cdot\delta_{2}(z)]+(1+z)\delta_2(z)\right\}
\end{align}
\end{theo}
The proof of Theorem \ref{thm-general-MD} is put in subsection \ref{proof-of-theorem-general-MD-section}.
\begin{remark}
Theorem \ref{thm-general-MD} establishes a Cram\'{e}r-type moderate deviation for bounded exchangeable pairs under condition $(D_{1})$, together with the boundedness assumptions on $|D|$ and $|\Delta|$. These boundedness conditions are natural, as all the examples of double-index permutation statistics we consider satisfy them. By comparison, Theorem 2.1 in \cite{zhang2023cramer}addresses the unbounded case and requires verifying four main conditions (A1)-(A4). Among them, condition (A3) is particularly difficult to verify in practice under the unbounded setting, and even when restricted to the bounded case, it still needs to be verified separately. By refining the proof, we are able to eliminate this condition entirely in the bounded setting.
\end{remark}


\subsection{Proof of Theorem \ref{thm-general-double-index-permutation-statistics-a-b-form}}\label{proof-of-theorem-general-double-index-permutation-statistics-a-b-form}
To simplify notation, we denote $a_{ik}: = a(i,k)$ and $b_{ijkl}:=b(i,j,k,l)$, denote $h(t) := \mathbb{E}\{e^{tW_n}\}$ and $\Psi_{t}(W_n) := e^{tWn}$.

\begin{proof}[Proof of Theorem \ref{thm-general-double-index-permutation-statistics-a-b-form}]
Firstly, we define the exchangeable pair $(\pi,\pi^{\prime})$, where $\pi$ is a random permutation chosen uniformly from $S_n$ (symmetric group of degree $n$), $\pi^{\prime}$ is a random permutation by interchanging $\pi(I)$ and $\pi(J)$ and leaving the rest of the indices of $\pi$ unchanged. Let $(I,J)$ be a random pair of indices chosen uniformly from $\{(i,j): i\neq j\in [N]\}$. Then we can easily have $(\pi,\pi^{\prime})$ is an exchangeable pair.

By Proposition \ref{general-DIPS-convert}, the general double index permutation statistics is defined as
\begin{align}
    W_n = \sum_{i = 1}a_{i\pi(i)}+{\sum_{i,j}}^{\prime}b_{ij\pi(i)\pi(j)},\quad W^{\prime}_{n} = \sum_{i = 1}a_{i\pi^{\prime}(i)}+{\sum_{i,j}}^{\prime}b_{ij\pi^{\prime}(i)\pi(j)},
\end{align}
since $(\pi,\pi^{\prime})$ is an exchangeable pair, so as $(W_n,W^{\prime}_n)$. Then we define $D = \Psi(\pi,\pi^{\prime})$ as an antisymmetric function of exchangeable pairs $(\pi,\pi^{\prime})$ as follows
\begin{align}
    D =& a_{I\pi(I)} - a_{I\pi^{\prime}(I)}+\sum_{s\notin \{I,J\}}b_{Is\pi(I)\pi(s)} - \sum_{s\notin \{I,J\}}b_{Is\pi^{\prime}(I)\pi^{\prime}(s)}\notag\\
    =&a_{I\pi(I)} - a_{I\pi(J)}+\sum_{s\notin \{I,J\}}b_{Is\pi(I)\pi(s)} - \sum_{s\notin \{I,J\}}b_{Is\pi(J)\pi(s)}.
\end{align}
Therefore we have
\begin{align}
    \mathbb{E}\{D\mid\pi\} =& \frac{1}{n}\sum_{i=1}^{n}a_{i\pi(i)}+\frac{1}{n(n-1)}\sum_{i = 1}^{n}a_{i\pi(i)}+\frac{1}{n}\sum_{i\neq j}b_{ij\pi(i)\pi(j)}-\frac{1}{n(n-1)}\sum_{i = 1}^{n}b_{ii\pi(i)\pi(i)}\notag\\
    =&\lambda(W_n+R),
\end{align}
where
\begin{align}
    \lambda = \frac{1}{n},\quad R = \frac{1}{n-1}\sum_{i = 1}^{n}a_{i\pi(i)}-\frac{1}{n-1}\sum_{i = 1}^{n}b_{ii\pi(i)\pi(i)}.
\end{align}
Under condition (\ref{equ-boundedness-condition-DIPS}), we have $|D|\leq 4\delta$, similarly we have
\begin{align}
    |\Delta| = |W_n - W^{\prime}_n| =& \big|a_{I\pi(I)} +a_{J\pi(J)} - a_{I\pi(J)} - a_{J\pi(I)}\notag\\
    & + b_{IJ\pi(I)\pi(J)} + b_{JI\pi(J)\pi(I)} - b_{IJ\pi(J)\pi(I)} - b_{JI\pi(I)\pi(J)} \notag\\
    &+ \sum_{s\notin \{I,J\}}(b_{sJ\pi(s)\pi(J)}+b_{sI\pi(s)\pi(I)}+b_{Js\pi(J)\pi(s)}+b_{Is\pi(I)\pi(s)})\notag\\
    &-\sum_{s\notin \{I,J\}}(b_{sJ\pi(s)\pi(I)}+b_{sI\pi(s)\pi(J)}+b_{Js\pi(I)\pi(s)}+b_{Is\pi(J)\pi(s)})\big| \notag\\
    \leq & 16\delta. 
\end{align}

To prove Theorem \ref{thm-general-double-index-permutation-statistics-a-b-form}, we apply Theorem \ref{thm-general-MD} on $W_n$. By Theorem \ref{thm-general-MD}, we need to verify the conditions $(A_{1})$-$(A_{3})$.
By the definition of $W_n$,  we know that $\{a(i,k)\}_{i,k\in[N]}$ is a real number matrix and $\{b(i,j,k,l)\}_{i,j,k,l\in[N]}$ is a 4-index real number array, so $W_n$ is finite. Then we have $\mathbb{E}\{e^{tW_n}\}< \infty$, therefore the first condition (A1) holds.

Next we consider the second condition (A2). For any absolutely continuous function $f:\mathbb{R}\to\mathbb{R}$ satisfying that $\mathbb{E}\{|f(W_n)|\}<\infty$, we have
\begin{align}
    \mathbb{E}(W_nf(W_n)) = \frac{1}{2\lambda}\mathbb{E}\left\{D\int_{-\Delta}^{0}f^{\prime}(W_n+u)\,du\right\} - \mathbb{E}\{Rf(W_n)\}.\notag
\end{align}
By applying $f(w) = w$, we have 
\begin{align}
    \frac{1}{2\lambda}\mathbb{E}\{D\Delta\} = \mathbb{E}\{W_n^2\} + \mathbb{E}\{RW_n\}.\notag
\end{align}
Under proposition \ref{general-DIPS-convert}, we know that $W_n$ follows condition
\begin{align}
    a(i,\cdot) = a(\cdot,k) = 0,\quad b(i,j,k\cdot) = b(i,j,\cdot,l) = b(i,\cdot,k,l) = b(\cdot,j,k,l) = 0,\notag
\end{align}
then we have
\begin{align}
    \mathbb{E}\{RW_n\} = & \mathbb{E}\left\{\left(\frac{1}{n-1}\sum_{i = 1}^{n}a_{i\pi(i)}-\frac{1}{n-1}\sum_{i = 1}^{n}b_{ii\pi(i)\pi(i)}\right)\left(\sum_{i = 1}^{n}a_{i\pi(i)}+\sum_{i\neq j}b_{ij\pi(i)\pi(j)}\right)\right\}\notag\\
    =& \frac{1}{n-1}\mathbb{E}\left\{\sum_{i,j}a_{i\pi(i)}a_{j\pi(j)}+\sum_{i = 1}^{n}\sum_{p\neq q}a_{i\pi(i)}b_{pq\pi(p)\pi(q)}\right\},\notag\\
    & -\frac{1}{n-1}\mathbb{E}\left\{\sum_{i,j}b_{ii\pi(i)\pi(i)}a_{j\pi(j)}-\sum_{i = 1}^{n}\sum_{p\ne q}b_{ii\pi(i)\pi(i)}b_{pq\pi(p)\pi(q)}\right\},\notag\\
    =& \frac{1}{(n-1)^2}\sum_{i,j}a^{2}_{ij} + \frac{2}{(n-1)^2(n-2)}\sum_{i,j}a_{ij}b_{iijj}-\frac{1}{(n-1)^2}\sum_{i,j}a_{ij}b_{iijj}\notag\\
    &+\frac{1}{n(n-1)^2(n-2)}\sum_{i\neq j}\sum_{k = 1}^{n}b_{iikk}b_{jjkk}-\frac{1}{n(n-1)^2(n-2)}\sum_{i\neq j}\sum_{k\neq l}b_{iijj}b_{jjll}\notag\\
    &+\frac{1}{n(n-1)^2(n-2)}\sum_{i  =1}^{n}\sum_{j\neq k}b_{iijj}b_{iikk}-\frac{2n-3}{n(n-1)^2(n-2)}\sum_{i,j}b^{2}_{ijij}.
\end{align}
Applying condition (\ref{equ-boundedness-condition-DIPS}), it follows that $\mathbb{E}\{RW_n\} \leq 5\delta^2$. Together with $\mathbb{E}\{W_n^2\} = 1 + \frac{C}{\sqrt{n}}$, we deduce that
\begin{align}
    &\mathbb{E}\left\{\left|1-\frac{1}{2\lambda}\mathbb{E}\{D\Delta\mid W_n\}\right|e^{tW_n}\right\}\notag\\
    = & \mathbb{E}\left\{\left|\frac{1}{2\lambda}\mathbb{E}\{D\Delta\mid W_n\} - \frac{1}{2\lambda}\mathbb{E}\{D\Delta\}+\frac{C}{\sqrt{n}}+\mathbb{E}\{RW_{n}\}\right|e^{tW_n}\right\}\notag\\
    \leq &\sqrt{\mathbb{E}\left\{\left(\frac{1}{2\lambda}\mathbb{E}\{D\Delta\mid W_n\} - \frac{1}{2\lambda}\mathbb{E}\{D\Delta\}+\frac{C}{\sqrt{n}}+\mathbb{E}\{RW_{n}\}\right)^2e^{tW_n}\right\}}\sqrt{\mathbb{E}\{e^{tW_n}\}}\notag\\
    \leq & \sqrt{3\mathbb{E}\left\{\left[\left(\frac{1}{2\lambda}\mathbb{E}\{D\Delta\mid W_n\} - \frac{1}{2\lambda}\mathbb{E}\{D\Delta\}\right)^2+\frac{C^2}{n}+25\delta^4\right]e^{tW_n}\right\}}\sqrt{\mathbb{E}\{e^{tW_n}\}}.\label{inequ:1/2lambda-D-Delta-etW-decompose}
\end{align}
By the definition of $D$ and $\Delta$, we decompose $\frac{1}{2\lambda}\mathbb{E}\{D\Delta\mid \pi\}$ as
\begin{align}
    &\frac{1}{2\lambda}\mathbb{E}\{D\Delta\mid \pi\} = \frac{n}{2}\sum_{i = 1}^{2}\sum_{j = 1}^{4}\mathbb{E}\{H_{i}Q_{j}\mid \pi\},\notag
\end{align}
where
\begin{align}
    H_{1} = & a_{I\pi(I)} - a_{I\pi(J)},\notag\\
    H_{2} = & \sum_{s\notin\{I,J\}}(b_{Is\pi(I)\pi(s)} - b_{Is\pi(J)\pi(s)}),\notag\\
    Q_{1} = & a_{I\pi(I)} +a_{J\pi(J)} - a_{I\pi(J)} - a_{J\pi(I)}, \notag\\
    Q_{2} = & b_{IJ\pi(I)\pi(J)} + b_{JI\pi(J)\pi(I)} - b_{IJ\pi(J)\pi(I)} - b_{JI\pi(I)\pi(J)}, \notag\\
    Q_{3} = & \sum_{s\notin \{I,J\}}(b_{sJ\pi(s)\pi(J)}+b_{sI\pi(s)\pi(I)}+b_{Js\pi(J)\pi(s)}+b_{Is\pi(I)\pi(s)}), \notag\\
    Q_{4} = & -\sum_{s\notin \{I,J\}}(b_{sJ\pi(s)\pi(I)}+b_{sI\pi(s)\pi(J)}+b_{Js\pi(I)\pi(s)}+b_{Is\pi(J)\pi(s)}). \notag
\end{align}
Then by Cauchy's inequality, we have
\begin{align}
    \mathbb{E}\left\{\left(\frac{1}{2\lambda}\mathbb{E}\{D\Delta\mid W_n\} - \frac{1}{2\lambda}\mathbb{E}\{D\Delta\}\right)^2e^{tW_n}\right\} \leq C \sum_{i = 1}^{2}\sum_{j = 1}^{4}\mathbb{E}\left\{(n\mathbb{E}\{H_{i}Q_{j}\mid\pi\}-n\mathbb{E}\{H_{i}Q_{j}\})^2e^{tW_n}\right\}.\label{inequ:1/2lambda-D-Delta-etW-decompose-HQ}
\end{align}
The right hand side of (\ref{inequ:1/2lambda-D-Delta-etW-decompose-HQ}) contains eight terms. We will analyze the upper bound of each term in the following. For the first term $\mathbb{E}\{(n\mathbb{E}\{H_{1}Q_{1}\mid\pi\} - n\mathbb{E}\{H_{1}Q_{1}\})^2e^{tW_n}\}$, we calculate the expectation $n\mathbb{E}\{H_{1}Q_{1}\mid\pi\}$ first
\begin{align}
    n\mathbb{E}\{H_{1}Q_{1}\mid\pi\} = \frac{n+3}{n-1}\sum_{i = 1}^{n}a^{2}_{i\pi(i)}+ \frac{1}{n-1}\sum_{i\neq j}(a^{2}_{i\pi(j)}+a_{i\pi(i)}a_{j\pi(j)}+a_{i\pi(j)}a_{j\pi(i)}).
\end{align}
Then applying Cauchy's inequality, the first term of (\ref{inequ:1/2lambda-D-Delta-etW-decompose-HQ}) is bounded by three parts $J_{1}, J_{2}, J_{3}$,
where
\begin{align}
    J_{1} = & \mathbb{E}\left\{\left(\sum_{i = 1}^{n}a^{2}_{i\pi(i)} - \mathbb{E}\left\{\sum_{i = 1}^{n}a^{2}_{i\pi(i)}\right\}\right)^2e^{tW_n}\right\},\notag\\
    J_{2} = & \mathbb{E}\left\{\left(\frac{1}{n}\sum_{i \neq j}a^{2}_{i\pi(j)} - \mathbb{E}\left\{\frac{1}{n}\sum_{i \neq j}a^{2}_{i\pi(j)}\right\}\right)^2e^{tW_n}\right\},\notag\\
    J_{3} = & \mathbb{E}\left\{\left(\frac{1}{n}\sum_{i \neq j}a_{i\pi(i)}a_{j\pi(j)} - \mathbb{E}\left\{\frac{1}{n}\sum_{i \neq j}a_{i\pi(i)}a_{j\pi(j)}\right\}\right)^2e^{tW_n}\right\}.\notag
\end{align}
Next, we establish a useful lemma that will help us bound $J_1$, $J_2$ and $J_3$.
\begin{lem}\label{lem-E(xetW)}
Let $W$ be defined in (\ref{normalized-DIPS-a-b-form}) which satisfies
 (\ref{equ-a-condition-origin}), (\ref{equ-b-condition-origin}) and (\ref{equ-boundedness-condition-DIPS}), 
then we have for $0<t<\frac{1}{\delta}$,
\begin{align}
    &\mathbb{E}\left\{\left(\sum_{i = 1}^{n}a^{2}_{i\pi(i)} - \mathbb{E}\left\{\sum_{i = 1}^{n}a^{2}_{i\pi(i)}\right\}\right)^2e^{tW_n}\right\} \leq  C(n\delta^4+n^2\delta^6t^2)h(t),\label{inequ:final-bound-a-square-form}\\
    &\mathbb{E}\left\{\left(\frac{1}{n}\sum_{i \neq j}a^{2}_{i\pi(j)} - \mathbb{E}\left\{\frac{1}{n}\sum_{i \neq j}a^{2}_{i\pi(j)}\right\}\right)^2e^{tW_n}\right\} \leq  C(n\delta^4+n^2\delta^6t^2)h(t),\label{inequ:final-bound-aij-square-form}\\
    &\mathbb{E}\left\{\left(\sum_{i\neq j}b^{2}_{ij\pi(i)\pi(j)} - \mathbb{E}\left\{\sum_{i\neq j}b^{2}_{ij\pi(i)\pi(j)}\right\}\right)^2e^{tW_n}\right\} \leq  C(n\delta^4+n^2\delta^6t^2)h(t),\label{inequ:final-bound-b-square-form}\\
    &\mathbb{E}\left\{\left(\sum_{i\neq j}b_{ij\pi(i)\pi(j)}b_{ji\pi(j)\pi(i)} - \mathbb{E}\left\{\sum_{i\neq j}b_{ij\pi(i)\pi(j)}b_{ji\pi(j)\pi(i)}\right\}\right)^2e^{tW_n}\right\} \leq  C(n\delta^4+n^2\delta^6t^2)h(t),\label{inequ:final-bound-bjiij-square-form}
\end{align}
\end{lem}
The proof of Lemma \ref{lem-E(xetW)} can be found in Section \ref{proof-of-other-results}. Then by Lemma \ref{lem-E(xetW)}, we have
\begin{align}
    J_{1} \leq C(n\delta^4+n^2\delta^6t^2)\mathbb{E}\{e^{tW_n}\}, \quad J_{2} \leq   C(n\delta^4+n^2\delta^6t^2)\mathbb{E}\{e^{tW_n}\}.\label{term-a-J1-J2-bound}
\end{align}
Considering $J_{3}$, by condition (\ref{equ-a-condition-origin}) and (\ref{equ-boundedness-condition-DIPS}) we obtain
\begin{align}
    \mathbb{E}\left\{\frac{1}{n}\sum_{i\neq j}a_{i\pi(i)}a_{j\pi(j)}\right\} = \frac{1}{n^3(n-1)}\sum_{i,j}a^{2}_{ij} \leq \frac{\delta^2}{n},\notag
\end{align}
therefore, it follows that
\begin{align}
    J_{3} \leq &\frac{1}{n^2}\left|\mathbb{E}\left\{\sum_{i\neq j}\sum_{k\neq l}a_{i\pi(i)}a_{j\pi(j)}a_{k\pi(k)}a_{l\pi(l)}e^{tW_n}\right\}\right| + \frac{2\delta^2}{n^2} \left|\mathbb{E}\left\{\sum_{i\neq j}a_{i\pi(i)}a_{j\pi(j)}e^{tW_n}\right\}\right|+\frac{\delta^4}{n^2}\mathbb{E}\{e^{tW_n}\}\notag\\
    \leq & \frac{1}{n^2}\left|\mathbb{E}\left\{\sum_{i\neq j}\sum_{k\neq l}a_{i\pi(i)}a_{j\pi(j)}a_{k\pi(k)}a_{l\pi(l)}e^{tW_n}\right\}\right| + C\delta^4\mathbb{E}\{e^{tW_n}\}\notag\\
    \leq &\frac{1}{n^4(n-1^2)}\left|\sum_{i\neq j}\sum_{k\neq l}\sum_{p\neq q}\sum_{u\neq v}a_{ip}a_{jq}a_{ku}a_{lv}\mathbb{E}\left\{e^{tW_n}\middle| \substack{\pi(i) = p,\pi(j) = q\\\pi(k) = u,\pi(l) = v} \right\}\right| + C\delta^4\mathbb{E}\{e^{tW_n}\}\notag\\
    \leq & n^2\delta^4\max_{\substack{i\neq j\neq k\neq l\\p\neq q\neq u\neq v}}\left|\mathbb{E}\left\{e^{tW_n}\middle| \substack{\pi(i) = p,\pi(j) = q\\\pi(k) = u,\pi(l) = v} \right\} - \mathbb{E}\{e^{tW_n}\}\right| + C\delta^4\mathbb{E}\{e^{tW_n}\},\label{term-a-J3-decompose}
\end{align}
where we use (\ref{equ-boundedness-condition-DIPS}) in the last inequality and $\max_{i_{1}\neq \dots \neq i_{k}}$ means $\max_{i_{1},\dots,i_{k}}$ under the condition that $i_{1},\dots,i_{k}$ are all distinct. Applying (\ref{general-final-bound-PsiWn-i1-ik-ht}) in the proof of Lemma \ref{lem-E(xetW)}, we have
\begin{align}
    \max_{\substack{i\neq j\neq k\neq l\\p\neq q\neq u\neq v}}\left|\mathbb{E}\left\{e^{tW_n}\middle| \substack{\pi(i) = p,\pi(j) = q\\\pi(k) = u,\pi(l) = v} \right\} - \mathbb{E}\{e^{tW_n}\}\right| \leq & C(\frac{1}{n}+\delta^2t^2)\mathbb{E}\{e^{tW_n}\}.\label{term-a-J3-bound-difference}
\end{align}
Combing (\ref{term-a-J1-J2-bound})-(\ref{term-a-J3-bound-difference}), we deduce the upper bound of the first term in (\ref{inequ:1/2lambda-D-Delta-etW-decompose-HQ}),
\begin{align}
    \mathbb{E}\{(n\mathbb{E}\{H_{1}Q_{1}\mid\pi\} - n\mathbb{E}\{H_{1}Q_{1}\})^2e^{tW_n}\} \leq & C(n\delta^4+n^2\delta^6t^2)\mathbb{E}\{e^{tW_n}\}.\label{term-a-H1Q1-bound}
\end{align}
Next we consider the term $\mathbb{E}\{(n\mathbb{E}\{H_{1}Q_{3}\mid\pi\} - n\mathbb{E}\{H_{1}Q_{3}\})^2e^{tW_n}\}$ in (\ref{inequ:1/2lambda-D-Delta-etW-decompose-HQ}). Since 
\begin{align}
    n\mathbb{E}\{H_{1}Q_{3}\mid \pi\} = & \frac{n}{n-1}\sum_{i\neq j}a_{i\pi(i)}b_{ij\pi(i)\pi(j)}+ \frac{n}{n-1}\sum_{i\neq j}a_{i\pi(i)}b_{ji\pi(j)\pi(i)}\notag\\
    &+\frac{2}{n-1}\sum_{i\neq j}a_{i\pi(j)}(b_{ji\pi(i)\pi(j)}+b_{ji\pi(j)\pi(i)})+\frac{2}{n-1}\sum_{i\neq j\neq s}a_{i\pi(i)}b_{sj\pi(s)\pi(j)},
\end{align}
this term can be divided into $J_{4}, J_{5}, J_{6}, J_{7}$, where
\begin{align}
    J_{4} = &  \mathbb{E}\left\{\left(\sum_{i \neq j}a_{i\pi(i)}b_{ij\pi(i)\pi(j)} - \mathbb{E}\left\{\sum_{i \neq j}a_{i\pi(i)}b_{ij\pi(i)\pi(j)}\right\}\right)^2e^{tW_n}\right\},\notag\\
    J_{5} = & \mathbb{E}\left\{\left(\sum_{i \neq j}a_{i\pi(i)}b_{ji\pi(j)\pi(i)} - \mathbb{E}\left\{\sum_{i \neq j}a_{i\pi(i)}b_{ji\pi(j)\pi(i)}\right\}\right)^2e^{tW_n}\right\},\notag\\
    J_{6} = & \mathbb{E}\left\{\frac{1}{n^2}\left(\sum_{i \neq j}a_{i\pi(j)}(b_{ji\pi(i)\pi(j)}+b_{ji\pi(j)\pi(i)}) - \mathbb{E}\left\{\sum_{i \neq j}a_{i\pi(j)}(b_{ji\pi(i)\pi(j)}+b_{ji\pi(j)\pi(i)})\right\}\right)^2e^{tW_n}\right\},\notag\\
    J_{7} = & \mathbb{E}\left\{\left(\frac{1}{n}\sum_{i \neq j\neq s}a_{i\pi(i)}b_{sj\pi(s)\pi(j)} - \mathbb{E}\left\{\frac{1}{n}\sum_{i \neq j\neq s}a_{i\pi(i)}b_{sj\pi(s)\pi(j)}\right\}\right)^2e^{tW_n}\right\}.\notag\\
\end{align}
We use a similar approach of (\ref{term-a-J3-decompose}) to get the upper bounds of these four parts. The proofs of these four parts are very similar. Here, we present the proof of the upper bound of $J_{4}$ as a representative, and the proofs of the other three parts can be obtained in the same argument.
Applying condition (\ref{equ-a-condition-origin}), (\ref{equ-b-condition-origin}) and (\ref{equ-boundedness-condition-DIPS}), we get
\begin{align}
    \mathbb{E}\left\{\sum_{i \neq j}a_{i\pi(i)}b_{ij\pi(i)\pi(j)}\right\} = \frac{1}{n(n-1)}\sum_{i,j}a_{ij}b_{iijj} \leq 2\delta^2,\notag
\end{align}
it then follows that
\begin{align}
    J_{4} \leq & \left|\mathbb{E}\left\{\sum_{i\neq j}\sum_{k\neq l}a_{i\pi(i)}b_{ij\pi(i)\pi(j)}a_{k\pi(l)}b_{kl\pi(k)\pi(l)}e^{tW_n} \right\}\right| + C n\delta^4\mathbb{E}\{e^{tW_n}\}\notag\\
    \leq & \frac{1}{n^2(n-1)^2}\left|\sum_{i\neq j}\sum_{k\neq l}\sum_{p\neq q}\sum_{u\neq v}a_{ip}b_{ijpq}a_{ku}b_{kluv}\mathbb{E}\left\{e^{tW_n}\middle| \substack{\pi(i) = p,\pi(j) = q\\ \pi(k) = u,\pi(l) = v}\right\}\right| + Cn\delta^4\mathbb{E}\{e^{tW_n}\}\notag\\
    \leq & Cn^2\delta^4 \max_{\substack{i\neq j\neq k\neq l\\p\neq q\neq u\neq v}}\left|\mathbb{E}\left\{e^{tW_n}\middle| \substack{\pi(i) = p,\pi(j) = q\\\pi(k) = u,\pi(l) = v} \right\} - \mathbb{E}\{e^{tW_n}\}\right| + Cn\delta^4\mathbb{E}\{e^{tW_n}\},\notag
\end{align}
where we use (\ref{equ-boundedness-condition-DIPS}) in the last inequality. Applying (\ref{term-a-J3-bound-difference}), we get 
\begin{align}
    J_{4} \leq & C(n\delta^4+n^2\delta^6t^2)\mathbb{E}\{e^{tW_n}\}.\label{term-ab-J4-bound}
\end{align}
By a same argument, we have
\begin{align}
    J_{5} \leq & C(n\delta^4+n^2\delta^6t^2)\mathbb{E}\{e^{tW_n}\}, \ J_{6} \leq C(n\delta^4+n^2\delta^6t^2)\mathbb{E}\{e^{tW_n}\}, \ J_{7} \leq C(n\delta^4+n^2\delta^6t^2)\mathbb{E}\{e^{tW_n}\}.\label{term-ab-J5-J6-J7-bound}
\end{align}
Together with (\ref{term-ab-J4-bound}) and (\ref{term-ab-J5-J6-J7-bound}), we obtain
\begin{align}
    \mathbb{E}\{(n\mathbb{E}\{H_{1}Q_{3}\mid\pi\} - n\mathbb{E}\{H_{1}Q_{3}\})^2e^{tW_n}\} \leq & C(n\delta^4+n^2\delta^6t^2)\mathbb{E}\{e^{tW_n}\}.\label{term-ab-H1Q3-bound}
\end{align}
By using the same argument, we obtain the upper bounds of the other three terms in (\ref{inequ:1/2lambda-D-Delta-etW-decompose-HQ}),
\begin{align}
    \mathbb{E}\{(n\mathbb{E}\{H_{1}Q_{2}\mid\pi\} - n\mathbb{E}\{H_{1}Q_{2}\})^2e^{tW_n}\} \leq & C(n\delta^4+n^2\delta^6t^2)\mathbb{E}\{e^{tW_n}\},\notag\\
    \mathbb{E}\{(n\mathbb{E}\{H_{1}Q_{4}\mid\pi\} - n\mathbb{E}\{H_{1}Q_{4}\})^2e^{tW_n}\} \leq & C(n\delta^4+n^2\delta^6t^2)\mathbb{E}\{e^{tW_n}\},\notag\\
    \mathbb{E}\{(n\mathbb{E}\{H_{2}Q_{1}\mid\pi\} - n\mathbb{E}\{H_{2}Q_{1}\})^2e^{tW_n}\} \leq & C(n\delta^4+n^2\delta^6t^2)\mathbb{E}\{e^{tW_n}\}.
\end{align}
Next we consider the term $\mathbb{E}\{(n\mathbb{E}\{H_{2}Q_{3}\mid\pi\} - n\mathbb{E}\{H_{2}Q_{3}\})^2e^{tW_n}\}$ in (\ref{inequ:1/2lambda-D-Delta-etW-decompose-HQ}). Since
\begin{align}
    &n\mathbb{E}\{H_{2}Q_{3}\mid \pi\}\notag\\
    = &\frac{n-2}{n-1}\sum_{i\neq j}b^{2}_{ij\pi(i)\pi(j)} + \frac{n-2}{n-1}\sum_{i\neq j}b_{ij\pi(i)\pi(j)}b_{ji\pi(j)\pi(i)}\notag\\
    &+\frac{n-3}{n-1}\sum_{i\neq j\neq s}b_{ij\pi(i)\pi(j)}(b_{is\pi(i)\pi(s)}+b_{si\pi(s)\pi(i)})\notag\\
    &+\frac{1}{n-1}\sum_{i\neq j\neq s}b_{ij\pi(i)\pi(j)}(b_{sj\pi(s)\pi(j)}+b_{js\pi(j)\pi(s)})\notag\\
    &+ \frac{2}{n-1}\sum_{i\neq j\neq p\neq q}b_{ij\pi(i)\pi(j)}b_{pq\pi(p)\pi(q)}\notag\\
    &- \frac{1}{n-1}\sum_{i\neq j\neq p\neq q}b_{ip\pi(j)\pi(p)}(b_{qj\pi(q)\pi(j)}+b_{qi\pi(q)\pi(i)}+b_{jq\pi(j)\pi(q)}+b_{iq\pi(i)\pi(q)})\notag\\
    &-\frac{1}{n-1}\sum_{i\neq j\neq s}b_{is\pi(j)\pi(s)}(b_{sj\pi(s)\pi(j)}+b_{si\pi(s)\pi(i)}+b_{js\pi(j)\pi(s)}+b_{is\pi(i)\pi(s)}) \notag\\
    : = & \sum_{i = 1}^{7}A_{i},
\end{align}
by using Cauchy's inequality, the term $\mathbb{E}\{(n\mathbb{E}\{H_{2}Q_{3}\mid\pi\} - n\mathbb{E}\{H_{2}Q_{3}\})^2e^{tW_n}\}$ is divided into seven parts,
\begin{align}
    \sum_{i = 1}^{7}\mathbb{E}\{(A_{i} - \mathbb{E}(A_{i}))^2e^{tW_n}\}.\label{inequ:H2Q3-etW-decompose}
\end{align}
We divide (\ref{inequ:H2Q3-etW-decompose}) into two groups, with terms including $A_{1},A_{2}$ as one group and the remaining terms including $A_{3}-A_{7}$ as another group, mainly based on their expectations. Following condition (\ref{equ-b-condition-origin}) and (\ref{equ-boundedness-condition-DIPS}), we calculate the expectation of $A_{1}-A_{7}$ as follows
\begin{align}
    |\mathbb{E}\{A_{1}\}| \leq Cn\delta^2,\quad |\mathbb{E}\{A_{2}\}| \leq Cn\delta^2, \quad \max_{i\in\{2,\dots,15\}}|\mathbb{E}\{A_{i}\}| \leq C\delta^2.
\end{align}
Note that the expectations of $A_{1}$ and $A_{2}$ are of order $n\delta^2$, while the expectations of $A_{3}-A_{7}$ are of order $\delta^2$. This difference leads us to adopt different approaches when analyzing the upper bounds of these two groups. We first consider the first group which including terms $\mathbb{E}\{(A_{1} - \mathbb{E}(A_{1}))^2e^{tW_n}\}$, and $\mathbb{E}\{(A_{2} - \mathbb{E}(A_{2}))^2e^{tW_n}\}$. These two terms can be bounded directly by using the result of Lemma \ref{lem-E(xetW)},
\begin{align}
    \sum_{i = 1}^{2}\mathbb{E}\{(A_{i} - \mathbb{E}(A_{i}))^2e^{tW_n}\} \leq & C(n\delta^4+n^2\delta^6t^2)\mathbb{E}\{e^{tW_n}\}.\label{term-b-H2Q3-A1A2-bound}
\end{align}
For the second group including terms $\mathbb{E}\{(A_{i} - \mathbb{E}(A_{i}))^2e^{tW_n}\}, i = 3,\dots,7$, we use a similar approach of (\ref{term-a-J3-decompose}) to get the upper bounds of these five parts. The proofs of these five parts are very similar. Here, we present the proof of the upper bound of $\mathbb{E}\{(A_{3} - \mathbb{E}(A_{3}))^2e^{tW_n}\}$ as a representative, and the proofs of the other four parts can be obtained in the same way.
Applying condition (\ref{equ-b-condition-origin}) and (\ref{equ-boundedness-condition-DIPS}), we have
\begin{align}
    &\mathbb{E}\{(A_{3} - \mathbb{E}\{A_{3}\})^2e^{tW_n}\}\notag\\
    \leq & C\left|\mathbb{E}\left\{\sum_{i\neq j\neq p}\sum_{k\neq l\neq q}b_{ij\pi(i)\pi(j)}b_{ip\pi(i)\pi(p)}b_{kl\pi(k)\pi(l)}b_{kq\pi(k)\pi(q)} e^{tW_n}\right\}\right|\notag\\
    & + C\delta^2\left|\mathbb{E}\left\{\sum_{i\neq j\neq s}b_{ij\pi(i)\pi(j)}b_{is\pi(i)\pi(s)} e^{tW_n}\right\}\right| + C\delta^4h(t)\notag\\
    \leq & Cn^2\delta^4\max_{B_{6}}\left|\mathbb{E}\left\{\Psi_{t}(W_n)\middle|\substack{\pi(i) = i^{\prime}, \pi(j) = j^{\prime}\\\pi(k) = k^{\prime},\pi(l) = l^{\prime}\\\pi(p) = p^{\prime},\pi(q) = q^{\prime}}\right\} - h(t)\right| + C \delta^4 h(t) \notag\\
    &+ Cn\delta^4\max_{B_{6}}\left(\mathbb{E}\left\{\Psi_{t}(W_{n})\middle|\substack{\pi(i) = i^{\prime}\\ \pi(j) = j^{\prime}\\\pi(k) = k^{\prime}\\\pi(l) = l^{\prime}\\\pi(p) = p^{\prime} }\right\} + \mathbb{E}\left\{\Psi_{t}(W_{n})\middle|\substack{\pi(i) = i^{\prime}\\ \pi(j) = j^{\prime}\\\pi(k) = k^{\prime}\\\pi(l) = l^{\prime} }\right\} + \mathbb{E}\left\{\Psi_{t}(W_{n})\middle|\substack{\pi(i) = i^{\prime}\\ \pi(j) = j^{\prime}\\\pi(k) = k^{\prime} }\right\} \right),\label{term-b-H2Q3-A3-decompose}
\end{align}
where $B_{6}$ denotes the set of indices $i,j,k,l,p,q\in[n]$ are all distinct and $i^{\prime},j^{\prime},k^{\prime},l^{\prime},p^{\prime},q^{\prime}\in[n]$ are also all distinct. We define $\sigma^{ijklp}_{i^{\prime}j^{\prime}k^{\prime}l^{\prime}p^{\prime}}: = \mathcal{P}^{p}_{p^{\prime}}\circ\mathcal{P}^{kl}_{k^{\prime}l^{\prime}}\circ\mathcal{P}^{ij}_{i^{\prime}k^{\prime}}\circ \sigma$ and $S_{\sigma^{ijklp}_{i^{\prime}j^{\prime}k^{\prime}l^{\prime}p^{\prime}}} = \sum_{i = 1}^{n}a_{i\sigma^{ijklp}_{i^{\prime}j^{\prime}k^{\prime}l^{\prime}p^{\prime}}(i)}+\sum_{i\neq j}b_{ij\sigma^{ijklp}_{i^{\prime}j^{\prime}k^{\prime}l^{\prime}p^{\prime}}(i)\sigma^{ijklp}_{i^{\prime}j^{\prime}k^{\prime}l^{\prime}p^{\prime}}(j)}$, where $\sigma$ is a random permutation chosen uniformly from $\mathcal{S}_n$ and independent of $\pi$. Before proceeding to the next step, we introduce a key auxiliary lemma that will be critical for identifying the conditional distribution of $W_n$.
\begin{lem}\label{lem-same-distribution-of-permutation}
Let $\pi$ and $\sigma$ be two independent random permutations, chosen uniformly from $S_n$. Suppose $i\neq j$ and $k\neq l$ are elements of $[1,\dots,n]$ and denote $\tau_{i,j}$ the transposition of $i$ and $j$, then
\begin{align}
    \sigma^{i}_{k}: = \mathcal{P}^{i}_{k}\sigma = \begin{cases}
        \sigma, & \sigma(i) = k,\\
        \sigma\circ\tau_{i,\sigma^{-1}(k)}, & \sigma(i) \neq k,
    \end{cases}\label{def-transformation-single}
\end{align}
\begin{align}
    \sigma^{ij}_{kl}: = \mathcal{P}^{ij}_{kl}\sigma = \begin{cases}
        \sigma\circ\tau_{i,\sigma^{-1}(k)}\circ\tau_{j,\sigma^{-1}(k)}, & \sigma(i) = l, \sigma(j) \neq  k,\\
        \sigma\circ\tau_{j,\sigma^{-1}(l)}\circ\tau_{i,\sigma^{-1}(l)}, & \sigma(i) \neq l, \sigma(j) = k,\\
        \sigma\circ\tau_{i,\sigma^{-1}(k)}\circ\tau_{j,\sigma^{-1}(l)}\circ \tau_{i,j}, &\sigma(i) = l, \sigma(j) = k,\\
        \sigma\circ\tau_{i,\sigma^{-1}(k)}\circ\tau_{j,\sigma^{-1}(l)}, & otherwise,
    \end{cases}\label{def-transformation-sigma-prime}
\end{align}
are two permutations that satisfy
\begin{align}
    \mathcal{L}\left(\sigma^{i}_{k}\right) \overset{d}{=}\mathcal{L}\left(\pi\mid \pi(i) = k\right),\label{same-distribution-of-permutation-single-transformation}
\end{align}
and
\begin{align}
    \mathcal{L}\left(\sigma^{ij}_{kl}\right) \overset{d}{=}\mathcal{L}\left(\pi\middle|\substack{\pi(i) = k\\ \pi(j) = l}\right).\label{same-distribution-of-permutation-double-transformation}
\end{align}
And for any elements $p\neq q, u\neq v$ of $[n]$ satisfying $p,q\notin\{i,j\}$, $u,v\notin\{k,l\}$, we have
\begin{align}
    \mathcal{L}\left(\mathcal{P}^{p}_{q}\sigma^{ij}_{kl}\right) \overset{d}{=}\mathcal{L}\left(\pi\middle| \substack{\pi(i) = k\\ \pi(j) = l\\ \pi(p) = q}\right), \quad \mathcal{L}\left(\mathcal{P}^{pq}_{uv}\sigma^{ij}_{kl}\right)\overset{d}{=}\mathcal{L}\left(\pi\middle|\substack{\pi(i) = k\\\pi(j) = l\\\pi(p) = u\\\pi(q) = v}\right).\label{composite-transformation}
\end{align}
\end{lem}
The proof of Lemma \ref{lem-same-distribution-of-permutation} can be found in Section \ref{proof-of-other-results}. Then, using Lemma \ref{lem-same-distribution-of-permutation}, we obtain
\begin{align}
    \mathcal{L}(S_{\sigma^{ijklp}_{i^{\prime}j^{\prime}k^{\prime}l^{\prime}p^{\prime}}}) = \mathcal{L}\left(W_n\middle| \pi(i) = i^{\prime}, \pi(j) = j^{\prime}, \pi(k) = k^{\prime}, \pi(l) = l^{\prime}, \pi(p) = p^{\prime}\right).\notag
\end{align}
By the definition of $S_{\sigma^{ijklp}_{i^{\prime}j^{\prime}k^{\prime}l^{\prime}p^{\prime}}}$ and using condition (\ref{equ-boundedness-condition-DIPS}), we notice that $|S_{\sigma^{ijklp}_{i^{\prime}j^{\prime}k^{\prime}l^{\prime}p^{\prime}}}-W_n| \leq C\delta$. There for we deduce that
\begin{align}
    \mathbb{E}\left\{\Psi_{t}(W_{n})\middle|\substack{\pi(i) = i^{\prime}\\ \pi(j) = j^{\prime}\\\pi(k) = k^{\prime}\\\pi(l) = l^{\prime}\\\pi(p) = p^{\prime} }\right\} = \mathbb{E}\left\{\Psi_{t}(S_{\sigma^{ijklp}_{i^{\prime}j^{\prime}k^{\prime}l^{\prime}p^{\prime}}})\right\} \leq Ch(t).\notag
\end{align}
By a same argument, the other two conditional expectations in the last line of (\ref{term-b-H2Q3-A3-decompose}) can be bounded in the same way, leading to
\begin{align}
    \mathbb{E}\{(A_{3} - \mathbb{E}(A_{3}))^2e^{tW_n}\} \leq Cn^2\delta^4\max_{B_{6}}\left|\mathbb{E}\left\{\Psi_{t}(W_n)\middle|\substack{\pi(i) = i^{\prime}, \pi(j) = j^{\prime}\\\pi(k) = k^{\prime},\pi(l) = l^{\prime}\\\pi(p) = p^{\prime}, \pi(q) = q^{\prime}}\right\} - h(t)\right| + Cn\delta^4h(t).\notag
\end{align}
Following (\ref{general-final-bound-PsiWn-i1-ik-ht}), we deduce that
\begin{align}
    \mathbb{E}\{(A_{3} - \mathbb{E}(A_{3}))^2e^{tW_n}\} \leq & C(n\delta^4+n^2\delta^6t^2)h(t).\label{term-b-H2Q3-A3-bound}
\end{align}
Therefor, by a same argument it follows that
\begin{align}
    \mathbb{E}\{(A_{i} - \mathbb{E}(A_{i}))^2e^{tW_n}\} \leq & C(n\delta^4+n^2\delta^6t^2)h(t), \quad i = 4,\dots,7.\label{term-b-H2Q3-A4A5A6A7-bound}
\end{align}
Together with (\ref{term-b-H2Q3-A1A2-bound}), (\ref{term-b-H2Q3-A3-bound}) and (\ref{term-b-H2Q3-A4A5A6A7-bound}), we obtain
\begin{align}
    \mathbb{E}\{(n\mathbb{E}\{H_{2}Q_{3}\mid\pi\} - n\mathbb{E}\{H_{2}Q_{3}\})^2e^{tW_n}\} \leq & C(n\delta^4+n^2\delta^6t^2)h(t).\notag
\end{align}
By using the same method, we obtain the upper bounds of the other two terms in (\ref{inequ:1/2lambda-D-Delta-etW-decompose-HQ}),
\begin{align}
    \mathbb{E}\{(n\mathbb{E}\{H_{2}Q_{2}\mid\pi\} - n\mathbb{E}\{H_{2}Q_{2}\})^2e^{tW_n}\} \leq & C(n\delta^4+n^2\delta^6t^2)h(t),\notag\\
    \mathbb{E}\{(n\mathbb{E}\{H_{2}Q_{4}\mid\pi\} - n\mathbb{E}\{H_{2}Q_{4}\})^2e^{tW_n}\} \leq & C(n\delta^4+n^2\delta^6t^2)h(t).\notag
\end{align}
Combining the above bounds for the eight terms in (\ref{inequ:1/2lambda-D-Delta-etW-decompose-HQ}), we conclude that
\begin{align}
    \mathbb{E}\left\{\left(\frac{1}{2\lambda}\mathbb{E}\{D\Delta\mid W_n\} - \frac{1}{2\lambda}\mathbb{E}\{D\Delta\}\right)^2e^{tW_n}\right\} \leq & C(n\delta^4+n^2\delta^6t^2)\mathbb{E}\{e^{tW_n}\}.\label{term-D-Delta-bound}
\end{align}
Substituting (\ref{term-D-Delta-bound}) into (\ref{inequ:1/2lambda-D-Delta-etW-decompose}), we obtain
\begin{align}
    \mathbb{E}\left\{\left|1-\frac{1}{2\lambda}\mathbb{E}\{D\Delta\mid W_n\}\right|e^{tW_n}\right\} \leq C\left(\sqrt{n}\delta^2+n\delta^3+n\delta^3t+1/\sqrt{n}\right)\mathbb{E}\left\{e^{tW_{n}}\right\}.\label{inequ:bound-1/2lambda-D-Delta-etW}
\end{align}
Finally we consider the third condition (A3). As for $\mathbb{E}\{R^2e^{tW_n}\}$, by using cauchy's inequality, and following condition (\ref{equ-boundedness-condition-DIPS}), we deduce that
\begin{align}
    \mathbb{E}\{R^2e^{tW_n}\} = &\mathbb{E}\left\{\left(\frac{1}{n-1}\sum_{i = 1}^{n}a_{i\pi(i)}-\frac{1}{n-1}\sum_{i = 1}^{n}b_{ii\pi(i)\pi(i)}\right)^2e^{tW_n}\right\}\notag\\
    \leq & \frac{2}{(n-1)^2}\left[\mathbb{E}\left\{\left(\sum_{i = 1}^{n}a_{i\pi(i)}\right)^2e^{tW_n}\right\}+\mathbb{E}\left\{\left(\sum_{i = 1}^{n}b_{ii\pi(i)\pi(i)}\right)^2e^{tW_n}\right\}\right]\notag\\
    \leq & 6\delta^2\mathbb{E}\{e^{tW_n}\}. \label{inequ:bound-R-square-etW}
\end{align}
For condition $(A_{3})$, together with (\ref{inequ:bound-R-square-etW}) and using holder's inequality, it follows that
\begin{align}
    \mathbb{E}\{|R|e^{tW}\}\leq \sqrt{\mathbb{E}\{R^2e^{tW_n}\}}\sqrt{\mathbb{E}\{e^{tW_n}\}}\leq \sqrt{6}\delta\mathbb{E}\{e^{tW_n}\}.\label{inequ:bound-R-etW}
\end{align}
Recalling Theorem \ref{thm-general-MD} and combining (\ref{inequ:bound-1/2lambda-D-Delta-etW}), (\ref{inequ:bound-R-etW}), we complete the proof of Theorem \ref{thm-general-double-index-permutation-statistics-a-b-form}.
\end{proof}

\subsection{Proof of Theorem \ref{thm-general-MD}}\label{proof-of-theorem-general-MD-section}
In this subsection, we develop the proof of Theorem \ref{thm-general-MD}. Proposition \ref{prop-moment-generating-function} establishes a bound for $\mathbb{E}\{e^{tW}\}$ via Stein's method. And in Proposition \ref{prop-general-moderate-deviation}, we derive a more general moderate deviation theorem for bounded exchangeable pairs. Finally, Theorem \ref{thm-general-MD} is obtained by combining Proposition \ref{prop-moment-generating-function} with Proposition \ref{prop-general-moderate-deviation}.

\begin{prop}\label{prop-moment-generating-function}
Under the assumption in Theorem \ref{thm-general-MD}. For $0\leq t\leq \min(\tau,\frac{1}{\delta})$, we have
\begin{align}\label{equ-moment-generating-function}
    \mathbb{E}\{e^{tW}\}\leq (1+9\delta)\exp\left(\frac{t^2}{2}(1+t\delta+2\delta_{1}(t))+3t\delta_{2}(t)\right).
\end{align}
\end{prop}

\begin{proof}[Proof of Proposition \ref{prop-moment-generating-function}]
    For $t = 0$, it is trivial that (\ref{equ-moment-generating-function}) holds. So we only need to consider $0<t\leq \min(\tau,\frac{1}{\delta})$. Let $h(t) = \mathbb{E}\{e^{tW}\}$. In order to bound $h(t)$, we need to find an upper bound for $h^{\prime}(t)$ using Stein's method, and then obtain a bound for $(\log h(t))^{\prime}$. This technique was firstly considered by \dots

    By conditions $(A_{1})$ and $(A_{3})$, $\mathbb{E}\{e^{tW}\}<\infty$ and $\mathbb{E}\{|R|e^{tW}\}<\infty$. Since condition $(D_{1})$ $\mathbb{E}\{D|W\} = \lambda(W+R)$. We have $\mathbb{E}\{|W|e^{tW}\}<\infty$.

    Under the condition $(D_{1})$, by antisymmetry, it follows that $\mathbb{E}\{D(f(W)+f(W^{\prime}))\} = 0$ for any absolutely continuous function $f:\mathbb{R}\to\mathbb{R}$ satisfying that $\mathbb{E}\{|f(W)|\}<\infty$. We obtain
    \begin{align}
        0&=\mathbb{E}\{D(f(W)+f(W^{\prime}))\}\notag\\
        &=2\mathbb{E}\{Df(W)\} - \mathbb{E}\{D(f(W) - f(W^{\prime}))\}\notag\\
        &=2\lambda\mathbb{E}\{(W+R)f(W)\} -\mathbb{E}\left\{D\int_{-\Delta}^{0}f^{\prime}(W+u)\,du\right\}.\notag
    \end{align}
    Then
    \begin{align}\label{equ_exchage_pari_start}
        \mathbb{E}(Wf(W)) = \frac{1}{2\lambda}\mathbb{E}\left\{D\int_{-\Delta}^{0}f^{\prime}(W+u)\,du\right\} - \mathbb{E}\{Rf(W)\}.
    \end{align}
    Applying (\ref{equ_exchage_pari_start}) with $f(w) = e^{tw}$, we have
    \begin{align}
        h^{\prime}(t) = & \mathbb{E}\{We^{tW}\} = \frac{t}{2\lambda}\mathbb{E}\left\{D\int_{-\Delta}^{0}e^{t(W+u)}\,du\right\} - \mathbb{E}\{Rf(W)\}\notag\\
        \leq & t\mathbb{E}\{e^{tW}\} +t\mathbb{E}\left\{\left|\frac{1}{2\lambda}\mathbb{E}\{D\Delta|W\}-1\right|e^{tW}\right\}+\frac{t}{2\lambda}\left|\mathbb{E}\left\{D\int_{-\Delta}^{0}(e^{t(W+u)} -e^{tW})\,du\right\}\right|\notag\\
        &+\mathbb{E}\{|R|e^{tW}\}\notag\\
        \leq & t\mathbb{E}\{e^{tW}\} + t\mathbb{E}\left\{\left|\frac{1}{2\lambda}\mathbb{E}\{D\Delta|W\}-1\right|e^{tW}\right\}+t\mathbb{E}\left\{\left|\frac{1}{2\lambda}\mathbb{E}\{D^{*}\Delta|W\}\right|e^{tW}\right\}+\mathbb{E}\{|R|e^{tW}\}
    \end{align}
    where $D^{*}:=D^{*}(X,X^{\prime})$ is any random variable such that $D^{*}(X,X^{\prime}) = D^{*}(X^{\prime},X)$ and $D^{*}\geq |D|$. We have this result by using Lemma 4.2 in Zhang (2023) in the last line, and by the boundedness condition we choose
    \begin{align}\label{equ-D-star}
        D^{*} = \delta
    \end{align}
    which is a constant. Then by conditions $(A_2), (A_3)$ in Theorem \ref{thm-general-MD} and condition  $(D_{1})$ we have
    \begin{align}
        h^{\prime}(t)\leq & th(t)+t\delta_{1}(t)h(t)+\frac{t\delta}{2}\mathbb{E}\{|W|e^{tW}\}+(1+\frac{t\delta}{2})\delta_{2}(t)h(t)\notag
    \end{align}
    since $|W| = W + 2W^{-}$ and $xe^{-tx}\leq \frac{e}{t}$, for $t>0$, we have
    \begin{align}\label{inequ:|W|etW}
        \mathbb{E}\{|W|e^{tW}\} = \mathbb{E}\{We^{tW}\}+2\mathbb{E}\{W^{-}e^{tW^{-}}\}\leq \mathbb{E}\{We^{tW}\}+\frac{2e}{t}.
    \end{align}
    then we have
    \begin{align}\label{inequ:h-prime}
        h^{\prime}(t)\leq & th(t)+t\delta_{1}(t)h(t)+\frac{t\delta}{2}(h^{\prime}(t)+\frac{2e}{t})+(1+\frac{t\delta}{2})\delta_{2}(t)h(t)\notag\\
        \leq & \frac{2}{2-t\delta}\left[t(1+\delta_{1}(t))h(t)+(1+\frac{t\delta}{2})\delta_{2}(t)h(t)+e\delta\right]\notag\\
        \leq & [t(1+\frac{t\delta}{2-t\delta}+2\delta_{1}(t))+(2+t\delta)\delta_{2}(t)]h(t)+2e\delta\notag\\
        := & g(t)h(t)+2e\delta.
    \end{align}
    Let $\mu(t) = \exp\left(-\int_{0}^{t}g(s)\,ds\right)$, then we have $\mu(0)=1$ and let both side of the above inequality multiply by $\mu(t)$, we have
    \begin{align}
        \mu(t)h^{\prime}(t)-g(t)\mu(t)h(t) &\leq 2e\delta\mu(t).\notag\\
        \frac{d}{dt}(\mu(t)h(t)) &\leq 2e\delta\mu(t).\notag
    \end{align}
    Integrating both sides from $0$ to $t$, we have
    \begin{align}
        \mu(t)h(t)-\mu(0)h(0) &\leq 2e\delta\int_{0}^{t}\mu(s)\,ds.\notag
    \end{align}
    Note that $h(0) = 1$, we have
    \begin{align}
        h(t) &\leq \mu(t)^{-1}\left(1+2e\delta\int_{0}^{t}\mu(s)\,ds\right).\notag
    \end{align}
    Since $\delta_{1}(t)$ and $\delta_{2}(t)$ are nondecreasing functions, we have
    \begin{align}
        \mu(t) =& \exp\left(-\int_{0}^{t}g(s)\,ds\right) = \exp\left(-\int_{0}^{s}\left[s(1+\frac{s\delta}{2-s\delta}+2\delta_{1}(s))+(2+s\delta)\delta_{2}(s)\right]\,ds\right)\notag\\
        \geq & \exp\left[-\left(\frac{t^2}{2}+\frac{t^3\delta}{3}+t^2\delta_{1}(t)+2t\delta_{2}(t)+\frac{t^2\delta}{2}\delta_{2}(t)\right)\right].\notag
    \end{align}
    And since $\mu(t)\leq e^{-\frac{t^2}{2}}$, we have $\int_{0}^{t}\mu(s)\,ds\leq \int_{0}^{t}e^{-\frac{s^2}{2}}\,ds\leq 1+\int_{1}^{t}\frac{1}{s^3}\,ds \leq \frac{3}{2}$, then
    \begin{align}
        h(t)\leq & \mu(t)^{-1}\left(1+2e\delta\int_{0}^{t}\mu(s)\,ds\right)\notag\\
        \leq &(1+3e\delta)\exp\left(\frac{t^2}{2}(1+t\delta+2\delta_{1}(t))+3t\delta_{2}(t)\right)\notag
    \end{align}
    which implies the desired result (\ref{equ-moment-generating-function}).
\end{proof}

\begin{prop}\label{prop-general-moderate-deviation}
Let $(W,W^{\prime})$, $\Delta$, $D$ and $D^{*}$ be defined as in Theorem \ref{thm-general-MD}. Assume that there exits a constant $\tau_0>0$ such that for all $0\leq t\leq \tau_0$,
\begin{itemize}
    \item [$(B_{1}):$] $\mathbb{E}\{\big|1-\frac{1}{2\lambda}\mathbb{E}\{D\Delta|W\}\big|e^{tW}\}\leq \kappa_1(t)e^{t^2/2}$,
    \item [$(B_{2}):$] $\mathbb{E}\{|R|e^{tW}\}\leq \kappa_2(t)e^{t^2/2}$,
\end{itemize}
where $\kappa_{1}(\cdot)$ and $\kappa_{2}(\cdot)$ are nondecreasing functions satisfying that $\kappa_{1}(\tau_0)<\infty$ and $\kappa_{2}(\tau_0)<\infty$. Then, for $0\leq z \leq \tau_{0}$
\begin{align}
    \left|\frac{\mathbb{P}(W>z)}{1-\Phi(z)}-1\right|\leq 31\left[(1+z^2)(\kappa_{1}(z)+\delta(1+\delta_{3}(z)+\kappa_{2}(z)))+(1+z)\kappa_{2}(z)\right]
\end{align}
where 
\begin{align}
    \delta_{3}(z) = [z(1+z\delta+2\delta_{1}(z))+(2+z\delta)\delta_{2}(z)](1+3e\delta)e^{\frac{z^2}{2}(z\delta+2\delta_{1}(z))+3z\delta_{2}(z)}.
\end{align}
\end{prop}

\begin{proof}[Proof of Proposition \ref{prop-general-moderate-deviation}]
Let $z\geq 0$ be a fixed real number, and let $f_{z}$ be the solution to the Stein equation
\begin{align}\label{stein-equation}
    f^{\prime}(w) - wf(w) = {\bf1}_{\{w\leq z\}} - \Phi(z),
\end{align} 
where $\Phi(\cdot)$ is the standard normal distribution function. It is well known that $f_{z}$ is given by
\begin{align}\label{solution-stein-equation}
    f_{z}(w) = \begin{cases}
        \frac{\Phi(w)\{1-\Phi(z)\}}{p(w)}, & w\leq z,\\
        \frac{\Phi(z)\{1-\Phi(w)\}}{p(w)}, & w>z.
    \end{cases}
\end{align}
where $p(w) = (2\pi)^{-1/2}e^{-w^2/2}$ is the stadard normal density function.

By (\ref{stein-equation}) and (\ref{equ_exchage_pari_start}) and taking $f = f_{z}$, we have
\begin{align}\label{total-decompose-J1-J2-J3}
    \mathbb{P}(W>z) - \{1-\Phi(z)\} = \mathbb{E}\{f^{\prime}_{z}(W)-Wf_{z}(W)\} = J_{1} + J_{2} + J_{3},
\end{align}
where
\begin{align}
    J_{1} &= \mathbb{E}\left\{f^{\prime}_{z}(W)\left(1-\frac{1}{2\lambda}\mathbb{E}(D\Delta|W)\right)\right\},\notag\\
    J_{2} &= \frac{1}{2\lambda}\mathbb{E}\left\{D\int_{-\Delta}^{0}(f^{\prime}_{z}(W+u) - f^{\prime}_{z}(W))\,du\right\},\notag\\
    J_{3} &= \mathbb{E}\{Rf_{z}(W)\}.\notag
\end{align}
Without loss of generality, we only consider $J_2$,because $J_{1}$ and $J_{3}$ can be bounded in a similar way. 

For $J_{2}$, observe that $f^{\prime}_{z}(w) = wf(w)-{\bf 1}_{\{w> z\}}+\{1-\Phi(z)\}$, and both $wf_{z}(w)$ and ${\bf 1}_{\{w> z\}}$ are increasing functions (see, \cite{chen2010normal}, Lemma 2.3), by lemma 4.2 in \cite{zhang2023cramer} and (\ref{equ-D-star}), we have
\begin{align}\label{inequ:J2-decompose}
    |J_{2}|\leq & \frac{1}{2\lambda}\left|\mathbb{E}\left[D\int_{-\Delta}^{0}\{(W+u)f_{z}(W+u) -Wf_{z}(W)\}\right]\right|\notag\\
    & + \frac{1}{2\lambda}\left|\mathbb{E}\left[D\int_{-\Delta}^{0}\{{\bf 1}_{\{W+u>z\}} - {\bf 1}_{\{W>z\}}\}\right]\right|\notag\\
    \leq & \frac{1}{2\lambda}\mathbb{E}\{|\mathbb{E}\{D^{*}\Delta\mid W\}|(|Wf_{z}(W)|+{\bf 1}_{\{W>z\}})\}\notag\\
    = & \frac{\delta}{2}\mathbb{E}\left\{\left|W+R\right|(|Wf_{z}(W)|+{\bf 1}_{\{W>z\}})\right\} = J_{21} +J_{22}
\end{align}
where 
\begin{align}
    J_{21} = \frac{\delta}{2}\mathbb{E}\left\{\left|W+R\right|\cdot |Wf_{z}(W)|\right\}, \quad J_{22} = \frac{\delta}{2}\mathbb{E}\left\{|W+R|{\bf 1}_{\{W>z\}}\right\}. \notag
\end{align}
For any $w>0$, it is well known that $(1-\Phi(w))/p(w)\leq \min\{1/w,\sqrt{2\pi}/2\}$. Then for $w>z$,
\begin{align}\label{inequ:bound-fz-w>z}
    |f_{z}(w)|\leq \frac{\sqrt{2\pi}}{2}\Phi(z), \quad |wf_{z}(w)|\leq \Phi(z)
\end{align}
and by symmetry, for $w<0$,
\begin{align}\label{inequ:bound-fz-w<0}
    |f_{z}(w)|\leq \frac{\sqrt{2\pi}}{2}\{1-\Phi(z)\}, |wf_{z}(w)|\leq 1-\Phi(z)
\end{align}
For $J_{21}$, by (\ref{solution-stein-equation}), (\ref{inequ:bound-fz-w>z}) and (\ref{inequ:bound-fz-w<0}), we have
\begin{align}\label{inequ:J21-decompose}
    J_{21} \leq & \frac{\delta}{2}\{1-\Phi(z)\}\mathbb{E}\left\{\left|W+R\right|{\bf 1}_{\{W<0\}}\right\}\notag\\
    &+\frac{\sqrt{2\pi}\delta}{2}\{1-\Phi(z)\}\mathbb{E}\left\{\left|W+R\right|\cdot We^{W^2/2}{\bf 1}_{\{0\leq W\leq z\}}\right\}\notag\\
    &+\frac{\delta}{2}\mathbb{E}\left\{\left|W+R\right|{\bf 1}_{\{W>z\}}\right\}.
\end{align}
Thus, by (\ref{inequ:J2-decompose}) and (\ref{inequ:J21-decompose}),
\begin{align}\label{inequ:J2-decompose-3-parts}
    |J_{2}| \leq & \frac{\delta}{2}\{1-\Phi(z)\}\mathbb{E}\left\{\left|W+R\right|{\bf 1}_{\{W<0\}}\right\}\notag\\
    &+\frac{\sqrt{2\pi}\delta}{2}\{1-\Phi(z)\}\mathbb{E}\left\{\left|W+R\right|\cdot We^{W^2/2}{\bf 1}_{\{0\leq W\leq z\}}\right\}\notag\\
    &+\delta\mathbb{E}\left\{\left|W+R\right|{\bf 1}_{\{W>z\}}\right\}.
\end{align}
For the first term of (\ref{inequ:J21-decompose}), without loss of generality, we assume that $\mathbb{E}\{W^2\} \leq 2$ and we have
\begin{align}\label{inequ:J21-1st-term}
    \frac{\delta}{2}\{1-\Phi(z)\}\mathbb{E}\left\{\left|W+R\right|{\bf 1}_{\{W<0\}}\right\} \leq \delta\{1-\Phi(z)\}(1+\kappa_{2}(z)).  
\end{align}
For the second term of (\ref{inequ:J21-decompose}), similarly to lemma 4.3 in \cite{zhang2023cramer}, we have
\begin{align}\label{inequ:J21-2nd-term-1} 
    & \mathbb{E}\left\{\left|W+R\right|\cdot We^{W^2/2}{\bf 1}_{\{0\leq W\leq z\}}\right\}\notag\\
    = &  \sum_{j = 1}^{\lfloor z \rfloor}\mathbb{E}\{|W+R|\cdot We^{W^2/2}{\bf 1}_{\{j-1\leq W<j\}} \} + \mathbb{E}\{|W+R|\cdot We^{W^2/2}{\bf 1}_{\{\lfloor z \rfloor \leq W<z\}} \}\notag\\
    = &   \sum_{j = 1}^{\lfloor z \rfloor}\mathbb{E}\{|W+R|\cdot We^{W^2/2-jW}e^{jW}{\bf 1}_{\{j-1\leq W<j\}} \} + \mathbb{E}\{|W+R|\cdot We^{W^2/2-zW}e^{zW}{\bf 1}_{\{\lfloor z \rfloor \leq W<z\}} \}\notag\\
    \leq  &   \sum_{j = 1}^{\lfloor z \rfloor}\mathbb{E}\{|W+R|\cdot \sup_{t\in (j-1,j)}(te^{t^2/2-jt}e^{jW}){\bf 1}_{\{j-1\leq W<j\}} \} + \mathbb{E}\{|W+R|\cdot \sup_{t\in(\lfloor z \rfloor,z)}(te^{t^2/2-zt}e^{zW}){\bf 1}_{\{\lfloor z \rfloor \leq W<z\}} \}\notag\\
    = &   \sum_{j = 1}^{\lfloor z \rfloor}je^{\frac{(j-1)^2}{2}-j(j-1)}\mathbb{E}\{|W+R|\cdot e^{jW}{\bf 1}_{\{j-1\leq W<j\}} \} + ze^{\lfloor z \rfloor^2/2-z\lfloor z \rfloor}\mathbb{E}\{|W+R|\cdot e^{zW}{\bf 1}_{\{\lfloor z \rfloor \leq W<z\}} \}\notag\\
    \leq &  2\sum_{j = 1}^{\lfloor z \rfloor}je^{-j^2/2}\mathbb{E}\{|W|\cdot e^{jW}{\bf 1}_{\{j-1\leq W<j\}} \} + 2ze^{-z^2/2}\mathbb{E}\{|W|\cdot e^{zW}{\bf 1}_{\{\lfloor z \rfloor \leq W<z\}} \}\notag\\
    & +  2\sum_{j = 1}^{\lfloor z \rfloor}je^{-j^2/2}\mathbb{E}\{|R|\cdot e^{jW}{\bf 1}_{\{j-1\leq W<j\}} \} + 2ze^{-z^2/2}\mathbb{E}\{|R|\cdot e^{zW}{\bf 1}_{\{\lfloor z \rfloor \leq W<z\}} \}\notag\\
    \leq &  2\sum_{j = 1}^{\lfloor z \rfloor}je^{-j^2/2}(\mathbb{E}\{We^{jW}\}+\frac{2e}{j}) + 2ze^{-z^2/2}(\mathbb{E}\{We^{zW}\}+\frac{2e}{z}) + 2\kappa_{2}(z)\left(\sum_{j = 1}^{\lfloor z \rfloor}j + z\right)\notag\\
    \leq & 2\sum_{j = 1}^{\lfloor z\rfloor}je^{-\frac{j^2}{2}}(\mathbb{E}\{We^{jW}\}+\frac{2e}{j})+2ze^{-\frac{z^2}{2}}(\mathbb{E}\{We^{zW}\}+\frac{2e}{z})+4(1+z^2)\kappa_{2}(z).
\end{align}
Then by (\ref{equ-moment-generating-function}), (\ref{inequ:h-prime}) and (\ref{inequ:J21-2nd-term-1}), we have
\begin{align}\label{inequ:J21-2nd-term-2}
    &\frac{\sqrt{2\pi}\delta}{2}\{1-\Phi(z)\}\mathbb{E}\left\{\left|W+R\right|\cdot We^{W^2/2}{\bf 1}_{\{0\leq W\leq z\}}\right\}\notag\\
    \leq & 2\sqrt{2\pi}\delta\{1-\Phi(z)\}(1+z^2)[z(1+z\delta+2\delta_{1}(z))+(2+z\delta)\delta_{2}(z)](1+3e\delta)e^{\frac{z^2}{2}(z\delta+2\delta_{1}(z))+3z\delta_{2}(z)} \notag\\
    &+ 2\sqrt{2\pi}\delta\{1-\Phi(z)\}(1+z^2)\kappa_{2}(z)\notag\\
    = & 2\sqrt{2\pi}\delta\{1-\Phi(z)\}(1+z^2)(\kappa_{2}(z) + \delta_{3}(z))
\end{align}
For the last term of (\ref{inequ:J21-decompose}), by Markov's inequality, we have for $z>1$,
\begin{align}
    \delta\mathbb{E}\left\{\left|W+R\right|{\bf 1}_{\{W>z\}}\right\} \leq & \delta\mathbb{E}\{|W|e^{zW}\}e^{-z^2}+\delta\mathbb{E}\{|R|e^{zW}\}e^{-z^2}\notag\\
    \leq & \delta(\mathbb{E}\{We^{zW}\}+\frac{2e}{z})e^{-z^2} + \delta\kappa_{2}(z)e^{\frac{z^2}{2}}e^{-z^2} \notag\\
    \leq & 2\delta(1+\delta+\delta_{3}(z)+\kappa_{2}(z))e^{-\frac{z^2}{2}}.\notag
\end{align}
Since it is well known that for $z>0$,
\begin{align}
    e^{-z^2/2}\leq \sqrt{2\pi}(1+z)\{1-\Phi(z)\}\leq \frac{3\sqrt{2\pi}}{2}(1+z^2)\{1-\Phi(z)\},\notag
\end{align}
we deduce that 
\begin{align}\label{inequ:J21-3rd-term}
    \delta\mathbb{E}\left\{\left|W+R\right|{\bf 1}_{\{W>z\}}\right\}\leq 9\sqrt{2\pi}\delta\{1-\Phi(z)\}(1+z^2)(1+\delta+\delta_{3}(z)+\kappa_{2}(z)).
\end{align}
Therefore, combining (\ref{inequ:J21-1st-term}), (\ref{inequ:J21-2nd-term-2}) and (\ref{inequ:J21-3rd-term}), for $z>1$, we have
\begin{align}\label{inequ:J2-final-bound}
    |J_{2}|\leq 12\sqrt{2\pi}\delta\{1-\Phi(z)\}(1+z^2)(1+\delta_{3}(z)+\kappa_{2}(z)).
\end{align}
Similarly to that in (\ref{inequ:J21-decompose}), by dividing $(-\infty,\infty)$ into three parts, and analyze each part carefully, it also follows that
\begin{align}
    |J_1|\leq 20\{1-\Phi(z)\}(1+z^2)\kappa_{1}(z), \quad |J_{3}|\leq 20\{1-\Phi(z)\}(1+z)\kappa_{2}(z).\notag
\end{align}
This completes the proof for $z>1$ together with (\ref{total-decompose-J1-J2-J3}). As for $0\leq z\leq 1$, in this case $1-\Phi(z)$ has a lower bound, so we can directly use the Berry-Esseen bound result in \cite{zhang2022berry} to complete the proof.
\end{proof}

\begin{proof}[Proof of Theorem \ref{thm-general-MD}]
By Proposition \ref{prop-moment-generating-function}, we have $\mathbb{E}\{e^{tW}\}\leq (1+9\delta)e^{\theta}e^{t^2/2}$, for $0\leq t\leq \tau_{0}(\theta)$. By conditions $(A_{1})$-$(A_{3})$, we have conditions $(B_{1})$ and $(B_{2})$ are satisfied with $\tau_{0} = \tau_{0}(\theta)$, and
\begin{align}
    \kappa_{1}(t) = (1+9\delta)\delta_1(t)e^{\theta}, \quad \kappa_{2}(t) = (1+9\delta)\delta_2(t)e^{\theta}.\notag
\end{align}
This proves Theorem \ref{thm-general-MD} by Proposition \ref{prop-general-moderate-deviation}.
\end{proof}

\section{Proof of other results}\label{proof-of-other-results}

\subsection{Proof of Theorem \ref{theo-chatterjee-rank-correlation-coefficient}}


\begin{proof}[Proof of Theorem \ref{theo-chatterjee-rank-correlation-coefficient}]
By \cite{chao1996estimating}, we know that the Chatterjee's rank correlation coefficient $\xi_n = (\Gamma_n - \mathbb{E}\Gamma_n)/(\text{Var} \Gamma_n)^{1/2} = (T_n-\mathbb{E}T_n)/(\text{Var} T_n)^{1/2}$, where $\Gamma_n$ is the oscillation of a permutation which is a core part of Chatterjee's rank correlation coefficient and $T_n = \sum_{i = 1}^{n}a_{\pi(i)\pi(i+1)}$, where 
\begin{align}
    a_{ij} = &(\alpha_{ij}-\alpha_{i\cdot}-\alpha_{\cdot j}+\alpha_{\cdot\cdot})/B(n),\notag\\
    B^2(n) = &\sum_{i,j}(\alpha_{ij}-\alpha_{i\cdot}-\alpha_{\cdot j}+\alpha_{\cdot\cdot})^2/(n-1) = \frac{(n+1)(2n^2+7)}{45},\notag
\end{align} 
and 
\begin{align}
    \alpha_{i\cdot} = \sum_{j}\alpha_{ij}/n,\quad \alpha_{\cdot j} = \sum_{i}\alpha_{ij}/n, \quad \alpha_{\cdot\cdot} = \sum_{i,j}\alpha_{ij}/n^2.\notag
\end{align}
By a direct calculation, we obtain
\begin{align}
    \mathbb{E}T_n =& -\frac{1}{n-1}\sum_{i}a_{ii},\notag\\
    \text{Var}T_n =& \frac{1}{n-2}\sum_{i,j}a_{ij}^{2}-\frac{1}{(n-1)(n-2)}\sum_{i,j}a_{ij}a_{ji}\notag\\
    &+\frac{1}{(n-1)^2(n-2)}\left(\sum_{i}a_{ii}\right)^2-\frac{n}{(n-1)(n-2)}\sum_{i}a_{ii}^2\notag\\
    =& 1+O(n^{-1}).\notag
\end{align}
So the normalized statistic is defined as 
\begin{align}
    W_n = T_n-\mathbb{E}T_n =& \sum_{i}a_{\pi(i)\pi(i+1)}+\sum_{i}a_{ii}/(n-1)\notag\\
    =&\sum_{i,j}\left({\bf 1}_{\{j = i+1\}}a_{\pi(i)\pi(i+1)}+\frac{a_{ii}}{n(n+1)}\right)\notag\\
    :=&\sum_{i,j}\xi(i,j,\pi(i),\pi(j)).\notag
\end{align}
By (\ref{decomposition-of-general-DIPs-1})-(\ref{decomposition-of-general-DIPs-3}) we have
\begin{align}
    \xi(i,j,k,\cdot) = \frac{a_{ii}}{n(n-1)} & \quad \xi(i,j,\cdot,l) = \frac{a_{ii}}{n(n-1)}\notag\\
    \xi(i,\cdot,k,l) = \frac{a_{kl}}{n}+\frac{a_{ii}}{n(n-1)}& \quad \xi(\cdot,j,k,l) = \frac{a_{kl}}{n}-\frac{\alpha_{\cdot\cdot}}{n(n-1)B(n)} \notag\\
    \xi(i,j,\cdot,\cdot) = \frac{a_{ii}}{n(n-1)} & \quad \xi(i,\cdot,k,\cdot) = \frac{a_{ii}}{n(n-1)}\notag\\
    \xi(i,\cdot,\cdot,l) = \frac{a_{ii}}{n(n-1)}& \quad \xi(\cdot,j,k,\cdot) = \frac{-\alpha_{\cdot\cdot}}{n(n-1)B(n)}\notag\\
    \xi(\cdot,j,\cdot,l) = \frac{-\alpha_{\cdot\cdot}}{n(n-1)B(n)}& \quad \xi(\cdot,\cdot,k,l) = \frac{a_{kl}}{n}-\frac{\alpha_{\cdot\cdot}}{n(n-1)B(n)}\notag\\
    \xi(i,\cdot,\cdot,\cdot) = \frac{a_{ii}}{n(n-1)}& \quad \xi(\cdot,j,\cdot,\cdot) = \frac{-\alpha_{\cdot\cdot}}{n(n-1)B(n)}\notag\\
    \xi(\cdot,\cdot,k,\cdot) = \frac{-\alpha_{\cdot\cdot}}{n(n-1)B(n)}& \quad \xi(\cdot,\cdot,\cdot,l) = \frac{-\alpha_{\cdot\cdot}}{n(n-1)B(n)}\notag\\
    \xi(\cdot,\cdot,\cdot,\cdot) = \frac{-\alpha_{\cdot\cdot}}{n(n-1)B(n)}& \quad \xi^{*}(i,j,k,l) = {\bf 1}_{\{j = i+1\}}a_{kl}- \frac{a_{kl}}{n}\notag\\
    \eta(i,k) = \frac{a_{ii}}{n-1} - \frac{a_{kk}}{n}& \quad \eta(i,\cdot) = \frac{a_{ii}}{n-1}+\frac{\alpha_{\cdot\cdot}}{nB(n)}\notag\\
    \eta(\cdot,k) = \frac{-\alpha_{\cdot\cdot}}{(n-1)B(n)}-\frac{a_{kk}}{n}& \quad \eta(\cdot,\cdot) = \frac{-\alpha_{\cdot\cdot}}{n(n-1)B(n)}\notag\\
    \eta^{*}(i,k) = 0& \quad \alpha_{\cdot\cdot} = \frac{n^2-1}{3n}\notag
\end{align}
and
\begin{align}
    W_n =& \sum_{i,j}^{\prime}\xi^{*}(i,j,\pi(i),\pi(j))-n\eta(\cdot,\cdot)\notag\\
    =&\sum_{i,j}^{\prime}\left({\bf 1}_{\{j = i+1\}}a_{\pi(i)\pi(j)}-\frac{1}{n}a_{\pi(i)\pi(j)}\right)+\sqrt{\frac{5(n+1)}{n^2(2n^2+7)}}.
\end{align}
So, if we define 
\begin{align}
    Y_n = \sum_{i,j}^{\prime}\left({\bf 1}_{\{j = i+1\}}a_{\pi(i)\pi(j)}-\frac{1}{n}a_{\pi(i)\pi(j)}\right):=\sum_{i,j}^{\prime}b(i,j,\pi(i),\pi(j)).
\end{align}
we have for $z\geq 0$,
\begin{align}
    &\left|\mathbb{P}(\xi_{n} > z) - (1-\Phi(z))\right|\notag\\
    \leq & \left|\mathbb{P}\left(Y_n > z(1+C_1/n) - \sqrt{\frac{5(n+1)}{n^2(2n^2+7)}}\right) - \mathbb{P}\left(Z > z(1+C_1/n) - \sqrt{\frac{5(n+1)}{n^2(2n^2+7)}}\right)\right|\notag\\
    &+\left|\mathbb{P}\left(Z > z(1+C_1/n) - \sqrt{\frac{5(n+1)}{n^2(2n^2+7)}}\right) - \mathbb{P}\left(Z > z\right)\right|\notag\\
    :=& J_1 + J_2.\notag
\end{align}
where $Z$ is a standard normal random variable, since $\max_{i,k}|a_{ik}|  = a_{1n} = a_{n1} = \frac{\sqrt{5}(n^2-1)}{n\sqrt{(n+1)(2n^2+7)}}\leq \sqrt{\frac{5}{n}}$, and then we can easily verify that $\{b(i,j,k,l)\}_{i,j,k,l\in[n]}$ satisfy condition (\ref{equ-b-condition-origin}) and the boundedness condition (\ref{equ-boundedness-condition-DIPS}), where $\delta = \frac{C}{\sqrt{n}}$. By Theorem \ref{thm-general-double-index-permutation-statistics-a-b-form}, we have for $0\leq z \leq n^{1/6}$,
\begin{align}
    J_1 = O(1)(1-\Phi(z))\frac{(1+z^3)}{\sqrt{n}}.
\end{align}
since for $z>1$, we have $\phi(z) \leq 2z(1-\Phi(z))$ and for $0\leq z\leq 1$,we have $\phi(z) \leq 2(1-\Phi(z))$, by the mean value theorem, we have for $0\leq z \leq n^{1/6}$,
\begin{align}
    J_2 = O(1)(1-\Phi(z))\frac{(1+z^3)}{\sqrt{n}}.
\end{align}
then we combine $J_1$ and $J_2$ to complete the proof.
\end{proof}


\subsection{Proof of Theorem \ref{theo-descents-inversions}}
\begin{proof}[Proof of Theorem \ref{theo-descents-inversions}]
Since
\begin{align}
    D = & \sum_{i,j}({\bf 1}\{i<j,\pi(i)-1 = \pi(j)\} - {\bf 1}\{i<j, \pi(i)+1 = \pi(j)\}) = 2 \text{Des}(\pi^{-1}) - (n-1)\notag\\
    := &\sum_{i,j}\xi(i,j,\pi(i),\pi(j)),\notag
\end{align}
where $\xi(i,j,k,l) =  {\bf 1}\{i<j,k-1 = l\} - {\bf 1}\{i<j, k + 1 = l\}$, by (\ref{decomposition-of-general-DIPs-1})-(\ref{decomposition-of-general-DIPs-3}) we have
\begin{align}
    \xi(i,j,k,\cdot) = &\frac{1}{n}\left({\bf 1}\{i<j,k=n\} - {\bf 1}\{i<j,k=1\}\right), \quad \xi(i,j,\cdot,l) = \frac{1}{n}\left({\bf 1}\{i<j,l=1\} - {\bf 1}\{i<j,l=n\}\right)\notag\\
    \xi(i,\cdot,k,l) = &\frac{n-i}{n}\left({\bf 1}\{k-1=l\} - {\bf 1}\{k+1=l\}\right),\quad \xi(\cdot,j,k,l) = \frac{j-1}{n}\left({\bf 1}\{k-1=l\} - {\bf 1}\{k+1=l\}\right)\notag\\
    \xi(i,j,\cdot,\cdot) = & 0,\quad \xi(i,\cdot,k,\cdot) = \frac{n-i}{n^2}\left({\bf 1}\{k = n\} - {\bf 1}\{k = 1\}\right)\notag\\
    \xi(i,\cdot,\cdot,l) = &\frac{n-i}{n^2}\left({\bf 1}\{l = 1\} - {\bf 1}\{l = n\}\right),\quad \xi(\cdot,j,k,\cdot) = \frac{j-1}{n^2}\left({\bf 1}\{k = n\} - {\bf 1}\{k = 1\}\right)\notag\\
    \xi(\cdot,j,\cdot,l) = & \frac{j-1}{n^2}\left({\bf 1}\{l  =1\} - {\bf 1}\{l = n\}\right),\quad \xi(\cdot,\cdot,k,l) =  \frac{n-1}{2n}\left({\bf 1}\{k-1=l\} - {\bf 1}\{k+1=l\}\right)\notag
\end{align}
and
\begin{align}
    \xi(i,\cdot,\cdot,\cdot) = & 0, \quad \xi(\cdot,j,\cdot,\cdot) = 0, \quad \xi(\cdot,\cdot,\cdot,\cdot) = 0\notag\\
    \xi(\cdot,\cdot,k,\cdot) = & \frac{n-1}{2n^2}\left({\bf 1}\{k  =n\} - {\bf 1}\{k =1\}\right),\quad \xi(\cdot,\cdot,\cdot,l) =  \frac{n-1}{2n^2}\left({\bf 1}\{l = 1\} - {\bf 1}\{l = n\}\right)\notag
\end{align}
then we have 
\begin{align}
    \xi^{*}(i,j,k,l) = & {\bf 1}\{i<j,k-1 = l\} - {\bf 1}\{i<j, k + 1 = l\} - \frac{1}{n}\left({\bf 1}\{i<j,k=n\} - {\bf 1}\{i<j,k=1\}\right)\notag\\
    & - \frac{1}{n}\left({\bf 1}\{i<j,l=1\} - {\bf 1}\{i<j,l=n\}\right)-\frac{n-1+2(j-i)}{2n}{\bf 1}\{k-1 =l\}\notag\\
    &+\frac{n-1+2(j-i)}{2n}{\bf 1}\{k+1=l\} + \frac{n-1+2(j-i)}{2n^2}{\bf 1}\{k = n\}-\frac{n-1+2(j-i)}{2n^2}{\bf 1}\{k = 1\}\notag\\
    &+\frac{n-1+2(j-i)}{2n^2}{\bf 1}\{l  =1\} - \frac{n-1+2(j-i)}{2n^2}{\bf 1}\{l = n\}.\notag
\end{align}
and
\begin{align}
    \eta(i,k)  = &\eta^{*}(i,k)=  \frac{n-2i+1}{n}{\bf 1}\{k = n\} - \frac{n-2i+1}{n}{\bf 1}\{k = 1\},\notag\\
    \eta(\cdot,k) = & \eta(i,\cdot) = \eta(\cdot,\cdot) = 0, \quad \sigma^2 = 2(n+1)/3\notag,
\end{align}
so 
\begin{align}
    W = &\frac{D - n\eta(\cdot,\cdot)}{\sigma} = \frac{\text{Des}-(n-1)/2}{\sqrt{(n+1)/6}}\notag\\
    = & \sum_{i}\sqrt{\frac{6}{n+1}}\eta^{*}(i,\pi(i))+\sum_{i,j}^{\prime}\sqrt{\frac{6}{n+1}}\xi^{*}(i,j,\pi(i),\pi(j))\notag\\
    :=&\sum_{i}a(i,\pi(i))+\sum_{i,j}^{\prime}b(i,j,\pi(i),\pi(j)).
\end{align}
and we can easily verify that $\{a(i,k)\}_{i,k\in[n]}$ and $\{b(i,j,k,l)\}_{i,j,k,l\in[n]}$ satisfy condition (\ref{equ-a-condition-origin}), (\ref{equ-b-condition-origin}) and the boundedness condition (\ref{equ-boundedness-condition-DIPS}) where $\delta = \frac{C}{\sqrt{n}}$, therefore we apply Theorem \ref{thm-general-double-index-permutation-statistics-a-b-form} to prove (\ref{descent-moderate-deviation}), by a same argument we prove (\ref{inversion-moderate-deviation}).
\end{proof}

\subsection{Proof of Theorem \ref{theo-mann-whitney-wilcoxon}}
\begin{proof}[Proof of Theorem \ref{theo-mann-whitney-wilcoxon}]
Since the Mann-Whitney-Wilcoxon statistic is defined as
\begin{align}
    D = \sum_{i,j}\xi(i,j,\pi(i),\pi(j)) ,\quad \xi(i,j,\pi(i),\pi(j)) = {\bf 1}\{1\leq i\leq n_{1},n_{1}+1\leq j\leq n,1\leq \pi(i)<\pi(j)\leq n\}\notag,
\end{align}
by (\ref{decomposition-of-general-DIPs-1})-(\ref{decomposition-of-general-DIPs-3}) we have
\begin{align}
    \xi(i,j,k,\cdot) = & \frac{n-k}{n}{\bf 1}\{1\leq i\leq n_1,n_1+1\leq j\leq n\},\quad \xi(i,j,\cdot,l) = \frac{l-1}{n}{\bf 1}\{1\leq i\leq n_{1},n_{1}+1\leq j\leq n\}\notag\\
    \xi(i,\cdot,k,l) = & \frac{n_2}{n}{\bf 1}\{1\leq i\leq n_{1},1\leq k<l\leq n\},\quad \xi(\cdot,j,k,l) = \frac{n_1}{n}{\bf 1}\{n_{1}+1\leq j\leq n,1\leq k<l\leq n\}\notag\\
    \xi(i,j,\cdot,\cdot) = & \frac{n-1}{2n}{\bf 1}\{1\leq i\leq n_{1},n_{1}+1\leq j\leq n\},\quad \xi(i,\cdot,k,\cdot) = \frac{n_{2}(n-k)}{n^2}{\bf 1}\{1\leq i\leq n_{1}\}\notag\\
    \xi(i,\cdot,\cdot,l) = &\frac{n_{2}(l-1)}{n^2}{\bf 1}\{1\leq i\leq n_{1}\},\quad \xi(\cdot,j,k,\cdot) = \frac{n_{1}(n-k)}{n^2}{\bf 1}\{n_{1}+1\leq j\leq n\}\notag\\
    \xi(\cdot,j,\cdot,l) = & \frac{n_{1}(l-1)}{n^2}{\bf 1}\{n_{1}+1\leq j\leq n\},\quad \xi(\cdot,\cdot,k,l) = \frac{n_{1}n_{2}}{n^2}{\bf 1}\{1\leq k<l\leq n\}\notag\\
    \xi(i,\cdot,\cdot,\cdot) = &\frac{n_{2}(n-1)}{2n^2}{\bf 1}\{1\leq i\leq n_{1}\},\quad \xi(\cdot,j,\cdot,\cdot) = \frac{n_{1}(n-1)}{2n^2}{\bf 1}\{n_{1}+1\leq j\leq n\}\notag\\
    \xi(\cdot,\cdot,k,\cdot) = &\frac{n_{1}n_{2}}{n^3}(n-k),\quad \xi(\cdot,\cdot,\cdot,l) = \frac{n_{1}n_{2}}{n^3}(l-1),\quad \xi(\cdot,\cdot,\cdot,\cdot) = \frac{n_{1}n_{2}(n-1)}{2n^3} \notag
\end{align}
and
\begin{align}
    \xi^{*}(i,j,k,l) = & {\bf 1}\{1\leq i\leq n_1,n_{1}+1\leq j\leq n,1\leq k<l\leq n\} - \frac{n-1+2(l-k)}{2n}{\bf 1}\{1\leq i\leq n_{1},n_{1}+1\leq j\leq n\}\notag\\
    &- \frac{n_{2}}{n}{\bf 1}\{1\leq i\leq n_{1},1\leq k<l\leq n\} - \frac{n_{1}}{n}{\bf 1}\{n_{1}+1\leq j\leq n,1\leq k<l\leq n\}\notag\\
    &+\frac{n_{2}(n-1+2(l-k))}{2n^2}{\bf 1}\{1\leq i\leq n_{1}\}+\frac{n_{1}(n-1+2(l-k))}{2n^2}{\bf 1}\{n_{1}+1\leq j\leq n\}\notag\\
    &+ \frac{n_{1}n_{2}}{n^2}{\bf 1}\{1\leq k<l\leq n\}-\frac{n_{1}n_{2}(n-1+2(l-k))}{2n^3}\notag.
\end{align}
then we have
\begin{align}
    \eta(i,k) = &{\bf 1}\{1\leq i\leq n_{1}\}\left(\frac{n_{2}(n-1)}{2n^2} + \frac{n_{2}(n-k)}{n}\right) + {\bf 1}\{n_{1}+1\leq i\leq n\}\left(\frac{n_{1}(n-1)}{2n^2}+\frac{n_{1}(k-1)}{n}\right)\notag\\
    &+\frac{n_1n_2(n-1)(n+1)}{2n^3}\notag\\
    \eta^{*}(i,k) = & {\bf 1}\{1\leq i\leq n_{1}\}\left(\frac{n_2(n-2k+1)}{2n}\right)+{\bf 1}\{n_{1}+1\leq j\leq n\}\left(\frac{n_{1}(2k-n-1)}{2n}\right)\notag\\
    \eta(\cdot,\cdot) = & \frac{n_1n_2(n-1)(n+1)}{2n^3},\quad \sigma^2 = \frac{n_1n_2(n+1)}{12}\notag
\end{align}
so
\begin{align}
    W = &\frac{D - n\eta(\cdot,\cdot)}{\sigma} = \frac{\sum_{i,j}\xi(i,j,\pi(i),\pi(j)) - n_1n_2(n-1)(n+1)/2n^2}{\sqrt{n_1n_2(n+1)/12}}\notag\\
    = & \sum_{i}\sqrt{\frac{12}{n_1n_2(n+1)}}\eta^{*}(i,\pi(i))+\sum_{i,j}^{\prime}\sqrt{\frac{12}{n_1n_2(n+1)}}\xi^{*}(i,j,\pi(i),\pi(j))\notag\\
    :=&\sum_{i}a(i,\pi(i))+\sum_{i,j}^{\prime}b(i,j,\pi(i),\pi(j)).
\end{align}
and we can easily verify that $\{a(i,k)\}_{i,k\in[n]}$ and $\{b(i,j,k,l)\}_{i,j,k,l\in[n]}$ satisfy condition (\ref{equ-a-condition-origin}), (\ref{equ-b-condition-origin}) and the boundedness condition (\ref{equ-boundedness-condition-DIPS}) where $\delta = \frac{C}{\sqrt{n}}$, therefore we apply Theorem \ref{thm-general-double-index-permutation-statistics-a-b-form} to prove (\ref{mann-whitney-wilcoxon-moderate-deviation}).
\end{proof}

\begin{proof}[Proof of Lemma \ref{lem-E(xetW)}]
We First consider (\ref{inequ:final-bound-a-square-form}), by a simple calculation we have
\begin{align}
    &\mathbb{E}\left\{\left(\sum_{i = 1}^{n}a^{2}_{i\pi(i)}-\mathbb{E}\left\{\sum_{i = 1}^{n}a^{2}_{i\pi(i)}\right\}\right)^2e^{tW_n}\right\}\notag\\
    =&\mathbb{E}\left\{\left(\sum_{i\neq j}a^{2}_{i\pi(i)}a^{2}_{j\pi(j)}\right)e^{tW_n}\right\}-\frac{2}{n}\sum_{i,j}a^{2}_{ij}\mathbb{E}\left\{\left(\sum_{k = 1}^{n}a^{2}_{k\pi(k)}\right)e^{tW_n}\right\}\notag\\
    &+\frac{1}{n^2}\left(\sum_{i,j}a^{2}_{ij}\right)^2h(t)+\mathbb{E}\left\{\left(\sum_{i = 1}^{n}a^{4}_{i\pi(i)}\right)e^{tW_n}\right\}\notag\\
    \leq & J_1 -J_2 + \frac{1}{n^2}\left(\sum_{i,j}a^{2}_{ij}\right)^2h(t) + n\delta^4h(t).\label{equ-a-square-form-decompose}
\end{align}
where
\begin{align}
    J_{1} = & \mathbb{E}\left\{\left(\sum_{i\neq j}a^{2}_{i\pi(i)}a^{2}_{j\pi(j)}\right)e^{tW_n}\right\},\notag\\
    J_{2} = & \frac{2}{n}\sum_{i,j}a^{2}_{ij}\mathbb{E}\left\{\left(\sum_{k = 1}^{n}a^{2}_{k\pi(k)}\right)e^{tW_n}\right\},\notag
\end{align}
Considering $J_1$, for any index $i,j\in[n]$ satisfy $i\neq j$, let 
\begin{align}
    &W_{n}^{(i,j)} = \sum_{\substack{p = 1\\ p\notin \{i,j\}}}^{n}a_{p\pi(p)} + \sum_{\substack{p\neq q\\ p,q\notin \{i,j\}}}b_{pq\pi(p)\pi(q)},\quad V_{ij} = W_n - W_{n}^{(i,j)}. \label{def:Wn-ij-and-V}
\end{align}
Applying condition (\ref{equ-boundedness-condition-DIPS}), we get
\begin{align}
    &|V_{ij}| =  \left|a_{i\pi(i)} + a_{j\pi(j)} +\sum_{\substack{p = 1\\ p\neq i}}^{n}(b_{ip\pi(i)\pi(p)}+b_{pi\pi(p)\pi(i)})+\sum_{\substack{p = 1\\ p\notin \{i,j\}}}^{n}(b_{jp\pi(j)\pi(p)}+b_{pj\pi(p)\pi(j)})\right|\leq 12\delta.\label{V-bound}
\end{align}
For any index $i\neq j$, we perform a Taylor expansion $W_{n}^{(i,j)}$. It then follows that
\begin{align}
    J_1 =& \frac{1}{n(n-1)}\sum_{i\neq j}\sum_{k\neq l}a^{2}_{ik}a^{2}_{jl}\mathbb{E}\left\{\Psi_{t}(W_n)\middle| \substack{\pi(i) = k\\ \pi(j) = l}\right\}\notag\\
    = & \frac{1}{n(n-1)}\sum_{i\neq j}\sum_{k\neq l}a^{2}_{ik}a^{2}_{jl}\mathbb{E}\left\{\Psi_{t}(W_{n}^{(i,j)})\middle| \substack{\pi(i) = k\\ \pi(j) = l}\right\}\notag\\
    &+\frac{t}{n(n-1)}\sum_{i\neq j}\sum_{k\neq l}a^{2}_{ik}a^{2}_{jl}\mathbb{E}\left\{V_{ij}\Psi_{t}(W_{n}^{(i,j)})\middle| \substack{\pi(i) = k\\ \pi(j) = l}\right\}\notag\\
    &+\frac{t^2}{n(n-1)}\sum_{i\neq j}\sum_{k\neq l}a^{2}_{ik}a^{2}_{jl}\mathbb{E}\left\{V_{ij}^{2}\Psi_{t}(W_{n}^{(i,j)}+UV_{ij})(1-U)\middle| \substack{\pi(i) = k\\ \pi(j) = l}\right\}\notag\\
    :=& J_{11} + J_{12} + J_{13}, \label{J1-decompose}
\end{align}
where $U$ is a uniform random variable on $[0,1]$ and is independent of $\pi$. Note that, $J_{1}$ is decomposed into three parts $J_{11}, J_{12} \text{ and } J_{13}$. We first consider $J_{11}$, by using the technique of adding and subtracting one item and the condition (\ref{equ-boundedness-condition-DIPS}), we obtain
\begin{align}
    J_{11} \leq & \frac{1}{n(n-1)}\sum_{i\neq j}\sum_{k\neq l}a^{2}_{ik}a^{2}_{jl}h(t) + \frac{1}{n(n-1)}\sum_{i\neq j}\sum_{k\neq l}a^{2}_{ik}a^{2}_{jl}\left|\mathbb{E}\left\{\Psi_{t}(W_{n}^{(i,j)})\middle| \substack{\pi(i) = k\\ \pi(j) = l}\right\}-h(t)\right| \notag\\
    \leq &\frac{1}{n(n-1)}\left(\sum_{i,j}a^{2}_{ij}\right)^2h(t) + \frac{\delta^4}{n(n-1)}\sum_{i\neq j}\sum_{k\neq l}\left|\mathbb{E}\left\{\Psi_{t}(W_{n}^{(i,j)})\middle| \substack{\pi(i) = k\\ \pi(j) = l}\right\}-h(t)\right|,\label{J11-decompose}
\end{align}
Next, to estimate $J_{1}$, we estimate the absolute difference $\left|\mathbb{E}\left\{\Psi_{t}(W_{n}^{(i,j)})\middle| \substack{\pi(i) = k\\ \pi(j) = l}\right\}-h(t)\right|$ for any fixed indexes $i\neq j,k\neq l$. Denote $\tau_{i,j}$ the transposition of $i$ and $j$, then we define $\sigma^{ij}_{kl}$ and $S^{(i,j)}_{\sigma^{ij}_{kl}}$ as follows 
\begin{align}
    \sigma^{ij}_{kl} = \begin{cases}
        \sigma\circ\tau_{i,\sigma^{-1}(k)}\circ\tau_{j,\sigma^{-1}(k)}, & \sigma(i) = l, \sigma(j) \neq  k,\\
        \sigma\circ\tau_{j,\sigma^{-1}(l)}\circ\tau_{i,\sigma^{-1}(l)}, & \sigma(i) \neq l, \sigma(j) = k,\\
        \sigma\circ\tau_{i,\sigma^{-1}(k)}\circ\tau_{j,\sigma^{-1}(l)}\circ \tau_{i,j}, &\sigma(i) = l, \sigma(j) = k,\\
        \sigma\circ\tau_{i,\sigma^{-1}(k)}\circ\tau_{j,\sigma^{-1}(l)}, & otherwise,
    \end{cases}\label{def-sigma}
\end{align}
\begin{align}
    S^{(i,j)}_{\sigma^{ij}_{kl}} =& \sum_{\substack{p = 1\\ p\notin \{i,j\}}}^{n}a_{p\sigma^{ij}_{kl}(p)} + \sum_{\substack{p\neq q\\ p,q\notin \{i,j\}}}b_{pq\sigma^{ij}_{kl}(p)\sigma^{ij}_{kl}(q)},\label{def-Sij}
\end{align}
where $i\neq j,k\neq l \in [n]$, $\sigma$ is a random permutation chosen uniformly from $S_n$ and independent of $\pi$. By lemma \ref{lem-same-distribution-of-permutation}, we have
\begin{align}
    \mathcal{L}(\sigma^{ij}_{kl}) \overset{d}{=}\mathcal{L}(\pi\mid\substack{\pi(i) = k\\ \pi(j) = l}).
\end{align}
Then it follows by the definition of $W^{(i,j)}_{n}$ in (\ref{def:Wn-ij-and-V}) and the definition of $S^{(i,j)}_{\sigma^{ij}_{kl}}$ in (\ref{def-Sij}) that
\begin{align}
    \mathcal{L}(S^{(i,j)}_{\sigma^{ij}_{kl}}) \overset{d}{=} \mathcal{L}(W_{n}^{(i,j)}\mid\substack{\pi(i) = k\\ \pi(j) = l}).
\end{align}
Hence $\mathbb{E}\left\{\Psi_{t}(W_{n}^{(i,j)})\middle|\substack{\pi(i) = k\\\pi(j) = l}\right\}$ can be replaced by $\mathbb{E}\left\{\Psi_{t}(S^{(i,j)}_{\sigma^{ij}_{kl}})\right\}$, and we perform a Taylor expansion on $\Psi_{t}(w)$ at $T_n = \sum_{i = 1}^{n}a_{i\sigma(i)}+\sum_{i\neq j}b_{ij\sigma(i)\sigma(j)}$. For any index $i\neq j,k\neq l \in [n]$, we have
\begin{align}
    &\left|\mathbb{E}\left\{\Psi_{t}(W_{n}^{(i,j)})\middle|\substack{\pi(i) = k\\\pi(j) = l}\right\}-h(t)\right| = \left|\mathbb{E}\left\{\Psi_{t}(S^{(i,j)}_{\sigma^{ij}_{kl}})\right\}-h(t)\right|\notag\\
    =&\left|t\mathbb{E}\left\{(T_n-S^{(i,j)}_{\sigma^{ij}_{kl}})\Psi_{t}(T_{n})\right\}+t^2\mathbb{E}\left\{(T_{n}-S^{(i,j)}_{\sigma^{ij}_{kl}})^2\Psi_{t}(T_{n}+U(T_{n}-S^{(i,j)}_{\sigma^{ij}_{kl}})(1-U))\right\}\right|\notag
\end{align}
Similarly to (\ref{V-bound}), applying condition (\ref{equ-boundedness-condition-DIPS}), we have $\left|T_n-S^{(i,j)}_{\sigma^{ij}_{kl}}\right|\leq C\delta$. Recalling that $0< t < 1/\delta$, we obtain
\begin{align}
    &\left|\mathbb{E}\left\{\Psi_{t}(W_{n}^{(i,j)})\middle|\substack{\pi(i) = k\\\pi(j) = l}\right\}-h(t)\right| \leq  \left|t\mathbb{E}\left\{(T_n-S^{(i,j)}_{\sigma^{ij}_{kl}})\Psi_{t}(T_{n})\right\}\right| + Ct^2\delta^2h(t).\notag\\
    \leq & \left|\frac{t}{n(n-1)}\sum_{\substack{p\neq q\\p,q \notin \{k,l\}}}\mathbb{E}\left\{(T_n - S^{(i,j)}_{\sigma^{ij}_{kl}})\Psi_{t}(T_n)\middle|\substack{\sigma(i)=p\\\sigma(j)=q}\right\}\right| + \frac{2t\delta}{n}\max_{i\neq j,p\neq q}\left|\mathbb{E}\left\{\Psi_{t}(T_n)\middle|\substack{\sigma(i)=p\\\sigma(j)=q}\right\}-h(t)\right|\notag\\
    &+ C(\frac{1}{n}+t^2\delta^2)h(t).\label{Psi(Wn-ij)-h(t)-bound-decompose}
\end{align}
So we next estimate 
\begin{align}
    \left|\frac{t}{n(n-1)}\sum_{\substack{p\neq q\\ p,q \notin \{k,l\}}}\mathbb{E}\left\{(T_n - S^{(i,j)}_{\sigma^{ij}_{kl}})\Psi_{t}(T_n)\middle|\substack{\sigma(i)=p\\\sigma(j)=q}\right\}\right|
\end{align}
Under the condition
\begin{align}
    \sigma(i) = p,\sigma(j) = q, i\neq j, k\neq l, p\neq q, p,q\notin \{k,l\}, \label{condition-Tn-Sij}
\end{align}
the values of permutation $\sigma^{ij}_{kl}$ and $\sigma$ on indexes $i,j,\sigma^{-1}(k),\sigma^{-1}(l)$ are given in the following table.
\begin{table}[htbp]
    \centering
    \begin{tabular}{ccccccccc}
        \toprule
        permutation & \multicolumn{4}{c}{$\sigma$} & \multicolumn{4}{c}{$\sigma^{ij}_{kl}$} \\
        \cmidrule(lr){2-5} \cmidrule(lr){6-9}
        index & $i$ & $j$ & $\sigma^{-1}(k)$ & $\sigma^{-1}(l)$ & $i$ & $j$ & $\sigma^{-1}(k)$ & $\sigma^{-1}(l)$ \\
        permutation(index) & $p$ & $q$ & $k$ & $l$ & $k$ & $l$ & $p$ & $q$ \\
        \bottomrule
    \end{tabular}
\end{table}

Recalling the definition of $T_n$ and $S^{(i,j)}_{\sigma^{ij}_{kl}}$, under condition (\ref{condition-Tn-Sij}) it follows that,
\begin{align}
    T_n - S^{(i,j)}_{\sigma^{ij}_{kl}}= & a_{ip} + a_{jq} + b_{ijpq} +a_{\sigma^{-1}(k)k} +a_{\sigma^{-1}(l)l} - a_{\sigma^{-1}(k)p} - a_{\sigma^{-1}(l)q} \notag\\
    & + \sum_{s\notin \{i,j\}}(b_{isp\sigma(s)}+b_{si\sigma(s)p}+b_{jsq\sigma(s)}+b_{sj\sigma(s)q})\notag\\
    & + b_{\sigma^{-1}(l)\sigma^{-1}(k)lk}+b_{\sigma^{-1}(k)\sigma^{-1}(l)kl} - b_{\sigma^{-1}(l)\sigma^{-1}(k)qp} - b_{\sigma^{-1}(k)\sigma^{-1}(l)pq}\notag\\
    & + \sum_{s \notin \{i,j,\sigma^{-1}(k),\sigma^{-1}(l)\}}(b_{s\sigma^{-1}(k)\sigma(s)k}+b_{\sigma^{-1}(k)sk\sigma(s)}- b_{s\sigma^{-1}(k)\sigma(s)p} - b_{\sigma^{-1}(k)sp\sigma(s)})\notag\\
    & + \sum_{s \notin \{i,j,\sigma^{-1}(k),\sigma^{-1}(l)\}}(b_{s\sigma^{-1}(l)\sigma(s)l}+b_{\sigma^{-1}(l)sl\sigma(s)} - b_{\sigma^{-1}(l)sq\sigma(s)} - b_{s\sigma^{-1}(l)\sigma(s)q}).\label{decompose-Tn-Sij}
\end{align}
We divide terms of (\ref{decompose-Tn-Sij}) into four groups. The first group contains $a_{ip}$, $a_{jq}$ and $b_{jipq}$ which do not have random index; the second group contains $a_{\sigma^{-1}(k)k}$, $a_{\sigma^{-1}(l)l}$, $a_{\sigma^{-1}(k)p}$, $a_{\sigma^{-1}(l)q}$ and $\sum_{s\notin \{i,j\}}(b_{isp\sigma(s)}+b_{si\sigma(s)p}+b_{jsq\sigma(s)}+b_{sj\sigma(s)q})$ which have one random index; the third group is the thrid line of (\ref{decompose-Tn-Sij}), which has two random indexes; the last group is the fourth and fifth line of (\ref{decompose-Tn-Sij}) with three random indexes. Therefore $\left|\frac{t}{n(n-1)}\sum_{\substack{p\neq q\\ p,q \notin \{k,l\}}}\mathbb{E}\left\{(T_n - S^{(i,j)}_{\sigma^{ij}_{kl}})\Psi_{t}(T_n)\middle|\substack{\sigma(i)=p\\\sigma(j)=q}\right\}\right|$ also be divided into four parts accordingly.

For the terms in $\left|\frac{t}{n(n-1)}\sum_{\substack{p\neq q\\ p,q \notin \{k,l\}}}\mathbb{E}\left\{(T_n - S^{(i,j)}_{\sigma^{ij}_{kl}})\Psi_{t}(T_n)\middle|\substack{\sigma(i)=p\\\sigma(j)=q}\right\}\right|$ corresponding to each part, we can use the same method to estimate their upper bounds, for each part, we only provide the estimation process of the upper bound of a representative term here. The rest can be obtained through the same method.

For the first part, we consider the representative term corresponding to $a_{ip}$. Applying condition (\ref{equ-a-condition-origin}) and (\ref{equ-boundedness-condition-DIPS}), we have 
\begin{align}
    &\left|\frac{t}{n(n-1)}\sum_{\substack{p\neq q\\ p,q \notin \{k,l\}}}\mathbb{E}\left\{a_{ip}\Psi_{t}(T_n)\middle|\substack{\sigma(i)=p\\\sigma(j)=q}\right\}\right|\notag\\
    \leq &\left|\frac{t}{n(n-1)}\sum_{\substack{q = 1\\q\notin \{k,l\}}}^{n}(-a_{iq}-a_{ik}-a_{il})h(t)\right| + \left|\frac{t}{n(n-1)}\sum_{\substack{p\neq q\\ p,q \notin \{k,l\}}}a_{ip}\left(\mathbb{E}\left\{\Psi_{t}(T_n)\middle|\substack{\sigma(i)=p\\\sigma(j)=q}\right\}-h(t)\right)\right|\notag\\
    \leq & \frac{3t\delta}{n}h(t) + t\delta\max_{i\neq j,p\neq q}\left|\mathbb{E}\left\{\Psi_{t}(T_n)\middle|\substack{\sigma(i)=p\\\sigma(j)=q}\right\}-h(t)\right|.\label{first-part-representative-term-bound}
\end{align}
Next, for the second part, we consider the representative term related to $\sum_{s\notin \{i,j\}}b_{isp\sigma(s)}$. Using condition (\ref{equ-b-condition-origin}) and (\ref{equ-boundedness-condition-DIPS}), it follows that
\begin{align}
    &\left|\frac{t}{n(n-1)}\sum_{\substack{p\neq q\\ p,q \notin \{k,l\}}}\mathbb{E}\left\{\sum_{s\notin \{i,j\}}(b_{isp\sigma(s)}\Psi_{t}(T_n)\middle|\substack{\sigma(i)=p\\\sigma(j)=q}\right\}\right|\notag\\
    \leq & \left|\frac{Ct}{n^3}\sum_{\substack{p\neq q\\ p,q \notin \{k,l\}}}\sum_{s\notin\{i,j\}}\sum_{r\notin\{p,q\}}b_{ispr}h(t)\right| + \left|\frac{Ct}{n^3}\sum_{\substack{p\neq q\\ p,q \notin \{k,l\}}}\sum_{s\notin\{i,j\}}\sum_{r\notin\{p,q\}}b_{ispr }\left(\mathbb{E}\left\{\Psi_{t}(T_n)\middle|\substack{\sigma(i)=p\\\sigma(j)=q\\\sigma(s) = r}\right\}-h(t)\right)\right|\notag\\
    \leq & \frac{Ct}{n^3}\sum_{\substack{p\neq q\\ p,q \notin \{k,l\}}}|b_{iipp}+b_{ijpp}+b_{iipq}+b_{ijpq}|h(t) + \frac{Ct}{n}\sum_{s,r}|b_{ispr }|\cdot\max_{\substack{i\neq j\neq s\\p\neq q\neq r}}\left|\mathbb{E}\left\{\Psi_{t}(T_n)\middle|\substack{\sigma(i)=p\\\sigma(j)=q\\\sigma(s) = r}\right\}-h(t)\right| \notag\\
    \leq & \frac{Ct\delta}{n}h(t) + Ct\delta\max_{\substack{i\neq j\neq s\\p\neq q\neq r}}\left|\mathbb{E}\left\{\Psi_{t}(T_n)\middle|\substack{\sigma(i)=p\\\sigma(j)=q\\\pi(s) = r}\right\}-h(t)\right|.\label{second-part-representative-term-bound}
\end{align}
Where $a\neq b\neq c$ means $a,b,c$ are distinct, and $a\neq b\neq c\neq d$ is a similar generalization. Then for the third part, we consider the representative term related to $b_{\sigma^{-1}(l)\sigma^{-1}(k)lk}$, by using condition (\ref{equ-b-condition-origin}) and (\ref{equ-boundedness-condition-DIPS}), we deduce that
\begin{align}
    &\left|\frac{t}{n(n-1)}\sum_{\substack{p\neq q\\ p,q \notin \{k,l\}}}\mathbb{E}\left\{b_{\sigma^{-1}(l)\sigma^{-1}(k)lk}\Psi_{t}(T_n)\middle|\substack{\sigma(i)=p\\\sigma(j)=q}\right\}\right|\notag\\
    \leq & \left|\frac{Ct}{n^4}\sum_{\substack{p\neq q\\p,q \notin\{k,l\}}}\sum_{\substack{u\neq v\\u,v \notin\{i,j\}}}b_{vulk}h(t)\right| + \left|\frac{Ct}{n^4}\sum_{\substack{p\neq q\\p,q \notin\{k,l\}}}\sum_{\substack{u\neq v\\u,v \notin\{i,j\}}}b_{vulk}\left(\mathbb{E}\left\{\Psi_{t}(T_{n})\middle|\substack{\pi(i) = p,\pi(j) = 1\\\pi(u)=p,\pi(v)=l}\right\} - h(t)\right)\right|\notag\\
    \leq & \left|\frac{Ct}{n^4}\sum_{\substack{p\neq q\\p,q \notin\{k,l\}}}\sum_{u\neq i,j}(-b_{iulk}-b_{julk}-b_{uulk})h(t)\right| + Ct\delta\max_{\substack{i\neq j\neq u\neq s\\p\neq q\neq k\neq r}}\left|\mathbb{E}\left\{\Psi_{t}(T_{n})\middle|\substack{\pi(i) = p\\\pi(j) = 1\\\pi(u)=p\\\pi(v)=l}\right\} - h(t)\right|\notag\\
    \leq & \frac{Ct\delta}{n}h(t) + Ct\delta\max_{\substack{i\neq j\neq u\neq s\\p\neq q\neq k\neq r}}\left|\mathbb{E}\left\{\Psi_{t}(T_n)\middle|\substack{\sigma(i)=p\\\sigma(j) = q\\\sigma(u)=k\\\sigma(s) = r}\right\}-h(t)\right|.\label{third-part-representative-term-bound}
\end{align}
For the last part, we consider the representative term related to $\sum_{s \notin \{i,j,\sigma^{-1}(k),\sigma^{-1}(l)\}}b_{s\sigma^{-1}(l)\sigma(s)l}$, by using condition (\ref{equ-b-condition-origin}) and (\ref{equ-boundedness-condition-DIPS}), we have
\begin{align}
    &\left|\frac{t}{n(n-1)}\sum_{\substack{p\neq q\\ p,q \notin \{k,l\}}}\mathbb{E}\left\{\sum_{s \notin \{i,j,\sigma^{-1}(k),\sigma^{-1}(l)\}}b_{s\sigma^{-1}(l)\sigma(s)l}\Psi_{t}(T_n)\middle|\substack{\sigma(i)=p\\\sigma(j)=q}\right\}\right|\notag\\
    \leq & \left|\frac{Ct}{n^5}\sum_{\substack{p\neq q\\ p,q \notin \{k,l\}}}\sum_{\substack{u\neq v \neq s\\ u,v,s \notin \{i,j\}}}\sum_{r\notin \{p,q,k,l\}}b_{svrl}h(t)\right|\notag\\
    &+ \left|\frac{Ct}{n^5}\sum_{\substack{p\neq q\\ p,q \notin \{k,l\}}}\sum_{\substack{u\neq v \neq s\\ u,v,s \notin \{i,j\}}}\sum_{r\notin \{p,q,k,l\}}b_{svrl}\left(\mathbb{E}\left\{\Psi_{t}(T_n)\middle|\substack{\sigma(i)=p\\\sigma(j) = q\\\sigma(u)=k\\\sigma(v)=l\\\sigma(s)=r}\right\}-h(t)\right)\right|\notag\\
    \leq & \frac{Ct}{n^2}\left(\sum_{s,p}|b_{svpl}|h(t)+\sum_{s,q}|b_{svql}|h(t)\right) + \left|\frac{Ct}{n^5}\sum_{\substack{p\neq q\\ p,q \notin \{k,l\}}}\sum_{\substack{u\neq v\\ u,v \notin \{i,j\}}}\sum_{t\in\{i,j,u,v\}}(b_{tvkl}+b_{tvll})h(t)\right| \notag\\
    &+ \frac{Ct}{n}\sum_{r,s}|b_{svrl}| \cdot \max_{\substack{i\neq j\neq u\neq v\neq s\\p\neq q\neq k\neq l\neq r}}\left|\mathbb{E}\left\{\Psi_{t}(T_n)\middle|\substack{\sigma(i)=p\\\sigma(j) = q\\\sigma(u)=k\\\sigma(v)=l\\\sigma(s)=r}\right\}-h(t)\right| \notag\\
    \leq & \frac{Ct\delta}{n}h(t) + Ct\delta\max_{\substack{i\neq j\neq u\neq v\neq s\\p\neq q\neq k\neq l\neq r}}\left|\mathbb{E}\left\{\Psi_{t}(T_n)\middle|\substack{\sigma(i)=p,\sigma(j) = q\\\sigma(u)=k,\sigma(v)=l,\sigma(s)=r}\right\}-h(t)\right|,\label{fourth-part-representative-term-bound}
\end{align}
From (\ref{Psi(Wn-ij)-h(t)-bound-decompose}), (\ref{decompose-Tn-Sij})-(\ref{fourth-part-representative-term-bound}), it follows that for any fixed index $i\neq j, k\neq l$,
\begin{align}
    &\left|\mathbb{E}\left\{\Psi_{t}(W_{n}^{(i,j)})\middle|\substack{\pi(i) = k\\\pi(j) = l}\right\}-h(t)\right| \leq Ct\delta (H+H_{1}+H_{2}+H_{3}) + C(\frac{1}{n}+t^2\delta^2)h(t).\label{inequ:bound-of-difference-PsiWnij-ht-v1}
\end{align}
Where
\begin{align}
    H = & \max_{i\neq j,p\neq q}\left|\mathbb{E}\left\{\Psi_{t}(T_n)\middle|\substack{\sigma(i)=p\\\sigma(j)=q}\right\}-h(t)\right|, \quad H_{1} =  \max_{\substack{i\neq j\neq s\\p\neq q\neq r}}\left|\mathbb{E}\left\{\Psi_{t}(T_n)\middle|\substack{\sigma(i)=p\\\sigma(j)=q\\\pi(s) = r}\right\}-h(t)\right|,\notag\\
    H_{2} = & \max_{\substack{i\neq j\neq u\neq s\\p\neq q\neq k\neq r}}\left|\mathbb{E}\left\{\Psi_{t}(T_n)\middle|\substack{\sigma(i)=p\\\sigma(j) = q\\\sigma(u)=k\\\sigma(s)=r}\right\}-h(t)\right|, \quad H_{3} = \max_{\substack{i\neq j\neq u\neq v\neq s\\p\neq q\neq k\neq l\neq r}}\left|\mathbb{E}\left\{\Psi_{t}(T_n)\middle|\substack{\sigma(i)=p\\\sigma(j) = q\\\sigma(u)=k\\\sigma(v)=l\\\sigma(s)=r}\right\}-h(t)\right|.\notag
\end{align}
To bound the terms $H, H_1, H_2, H_3$ and simplify the right-hand side of (\ref{inequ:bound-of-difference-PsiWnij-ht-v1}), we now invoke the following key lemma:

\begin{lem}\label{bound-H1-H4-general}
Let $\pi$ be a random permutation chosen uniformly from $S_n$ (symmetric group of degree $n$), $W_n$ is defined in (\ref{normalized-DIPS-a-b-form}) and satisfies (\ref{equ-boundedness-condition-DIPS}). Suppose $k < n$, for any fixed index $i_{1},\dots,i_{k}\in[n]$ which are all distinct, and $l_{1},\dots,l_{k}\in[n]$ are also distinct, we have for $0<t<1/\delta$,
\begin{align}
    \left|\mathbb{E}\left\{\Psi_{t}(W_{n})\middle|\substack{\pi(i_{1}) = l_{1}\\...\\\pi(i_{k}) = l_{k}}\right\} - h(t)\right|\leq Ck^2e^{k}t\delta h(t).
\end{align}
\end{lem}
The proof of Lemma \ref{bound-H1-H4-general} is in the last part of Section \ref{proof-of-other-results}. Then, using Lemma \ref{bound-H1-H4-general}, for any fixed index $i\neq j, k\neq l \in [n]$, it follows that
\begin{align}
    \left|\mathbb{E}\left\{\Psi_{t}(W_{n}^{(i,j)})\middle|\substack{\pi(i) = k\\\pi(j) = l}\right\}-h(t)\right| \leq C(\frac{1}{n} + t^2\delta^2)h(t).\label{inequ:bound-of-difference-PsiWnij-ht-final}
\end{align}
By the same argument, we can also obtain for any fixed index $i_{1},\dots,i_{k}\in [n]$ are all distinct, $l_{1},\dots,l_{k}\in [n]$ are all distinct, $k<n$,
\begin{align}
    \left|\mathbb{E}\left\{\Psi_{t}(W^{(i_{1},\dots,i_{k})}_{n})\middle|\substack{\sigma(i_{1})=l_{1}\\\dots\\\sigma(i_{k})=l_{k}}\right\}-h(t)\right| \leq Ck^2(\frac{1}{n}+t^2\delta^2)h(t).\label{general-final-bound-PsiWn-i1-ik-ht}
\end{align}
Therefore we bound the second term of (\ref{J11-decompose}) by
\begin{align}
    \frac{\delta^4}{n(n-1)}\sum_{k\neq l}\sum_{i\neq j}\left|\mathbb{E}\left\{\Psi_{t}(W_{n}^{(i,j)})\middle| \substack{\pi(i) = k\\ \pi(j) = l}\right\}-h(t)\right|\leq C(n\delta^4+n^2\delta^6t^2)h(t).\notag
\end{align}
Consequently, we obtain the bound of $J_{11}$ as
\begin{align}
    J_{11} \leq & \frac{1}{n(n-1)}\left(\sum_{i,j}a^{2}_{ij}\right)^2h(t) + C\left(n\delta^4+n^2\delta^6t^2\right)h(t),\label{J11-bound-final}
\end{align}
Then for $J_{12}$, together with (\ref{def:Wn-ij-and-V}) and (\ref{V-bound}), we deduce that
\begin{align}
    J_{12} =& \frac{t}{n(n-1)}\sum_{i\neq j}\sum_{k\neq l}a^{2}_{ik}a^{2}_{jl}\mathbb{E}\left\{V_{ij}\Psi_{t}(W_{n}^{(i,j)})\middle| \substack{\pi(i) = k\\ \pi(j) = l}\right\}\notag\\
    =& \frac{t}{n(n-1)}\sum_{i\neq j}\sum_{k\neq l}a^{2}_{ik}a^{2}_{jl}\mathbb{E}\left\{\left(a_{i\pi(i)} + a_{j\pi(j)}\right)\Psi_{t}(W_{n}^{(i,j)})\middle| \substack{\pi(i) = k\\ \pi(j) = l}\right\}\notag\\
    & + \frac{t}{n(n-1)}\sum_{i\neq j}\sum_{k\neq l}a^{2}_{ik}a^{2}_{jl}\mathbb{E}\left\{\left(\sum_{\substack{p = 1\\ p\neq i}}^{n}(b_{ip\pi(i)\pi(p)}+b_{pi\pi(p)\pi(i)})\right)\Psi_{t}(W_{n}^{(i,j)})\middle| \substack{\pi(i) = k\\ \pi(j) = l}\right\},\notag\\
    & + \frac{t}{n(n-1)}\sum_{i\neq j}\sum_{k\neq l}a^{2}_{ik}a^{2}_{jl}\mathbb{E}\left\{\left(\sum_{\substack{p = 1\\ p\notin \{i,j\}}}^{n}(b_{jp\pi(j)\pi(p)}+b_{pj\pi(p)\pi(j)})\right)\Psi_{t}(W_{n}^{(i,j)})\middle| \substack{\pi(i) = k\\ \pi(j) = l}\right\},\notag\\
    \leq & \frac{2t}{n(n-1)}\sum_{i\neq j}\sum_{k\neq l}a^{3}_{ik}a^{2}_{jl}\mathbb{E}\left\{\Psi_{t}(W_{n}^{(i,j)})\middle|\substack{\pi(i) = k\\\pi(j) = l}\right\}\notag\\
    & +  \frac{Ct}{n^3}\sum_{i\neq j}\sum_{k\neq l}a^{2}_{ik}a^{2}_{jl}\left(\sum_{\substack{p = 1\\ p\notin \{i,j\}}}^{n}\sum_{\substack{t = 1\\ t\notin \{k,l\}}}^{n}(b_{ipkt}+b_{pitk})\mathbb{E}\left\{\Psi_{t}(W_{n}^{(i,j)})\middle|\substack{\pi(i) = k\\\pi(j) = l\\\pi(p) = t}\right\}\right)\notag\\
    \leq & \frac{2t}{n^2}\sum_{i\neq j}\sum_{k\neq l}a^{3}_{ik}a^{2}_{jl}h(t) + Cn^2\delta^5t\left|\mathbb{E}\left\{\Psi_{t}(W_{n}^{(i,j)})\middle|\substack{\pi(i) = k\\\pi(j) = l}\right\}-h(t)\right| + n\delta^4h(t)\notag\\
    & + Cn^2\delta^5t\max_{\substack{i\neq j\neq p\\k\neq l\neq t}}\left|\mathbb{E}\left\{\Psi_{t}(W_{n}^{(i,j,k)})\middle|\substack{\pi(i) = k\\\pi(j) = l\\\pi(p) = t}\right\}-h(t)\right|\notag\\
    & +  Cn^2\delta^5t^2\max_{\substack{i\neq j\neq p\\k\neq l\neq t}}\left|\mathbb{E}\left\{(W^{(i,j)}_{n} - W^{(i,j,k)}_{n})\Psi_{t}(W_{n}^{(i,j,k)})\middle|\substack{\pi(i) = k\\\pi(j) = l\\\pi(p) = t}\right\}\right|\notag\\
    &+  Cn^2\delta^5t^3\max_{\substack{i\neq j\neq p\\k\neq l\neq t}}\left|\mathbb{E}\left\{(W^{(i,j)}_{n} - W^{(i,j,k)}_{n})^2\Psi_{t}(W_{n}^{(i,j,k)} + U(W^{(i,j)}_{n} - W^{(i,j,k)}_{n}) )(1-U)\middle|\substack{\pi(i) = k\\\pi(j) = l\\\pi(p) = t}\right\}\right|.\label{J12-decompose}
\end{align}
Where we use (\ref{equ-b-condition-origin}), (\ref{equ-boundedness-condition-DIPS}) and the range $0<t<1/\delta$ in the last inequality. By condition (\ref{equ-boundedness-condition-DIPS}), we have $|W^{(i,j)}_{n} - W^{(i,j,p)}_{n}|\leq C\delta$. Then in view of Lemma \ref{bound-H1-H4-general} and (\ref{general-same-distribution-of-permutation-k})-(\ref{lem-bound-H1-H4-general-bound-1}), it follows that
\begin{align}
    J_{12} \leq \frac{2t}{n^2}\sum_{i\neq j}\sum_{k\neq l}a^{3}_{ik}a^{2}_{jl}h(t) + C(n\delta^4 + n^2\delta^6t^2)h(t).\label{J12-bound-final}
\end{align}
We have already established the bound for $J_{12}$. We next consider $J_{13}$, together with (\ref{inequ:bound-of-difference-PsiWnij-ht-final}) we obtain
\begin{align}
    J_{13} = & \frac{t^2}{n(n-1)}\sum_{k\neq l}\sum_{i\neq j}a^{2}_{ik}a^{2}_{jl}\mathbb{E}\left\{V^2\Psi_{t}(W_{n}^{(i,j)}+UV)(1-U)\middle| \substack{\pi(i) = k\\ \pi(j) = l}\right\} \notag\\
    \leq & \frac{t^2\delta^2}{n(n-1)}\sum_{k\neq l}\sum_{i\neq j}a^{2}_{ik}a^{2}_{jl}e^{t\delta}h(t) + \frac{t^2\delta^2}{n(n-1)}\sum_{k\neq l}\sum_{i\neq j}a^{2}_{ik}a^{2}_{jl}e^{t\delta}\left|\mathbb{E}\left\{\Psi_{t}(W_{n}^{(i,j)})\middle| \substack{\pi(i) = k\\ \pi(j) = l}\right\}-h(t)\right|\notag\\
    \leq & Cn^2\delta^6t^2 h(t).\label{J13-bound-final}
\end{align}
Combine (\ref{J11-bound-final}), (\ref{J12-bound-final}) and (\ref{J13-bound-final}), we have the following bound for $J_{1}$
\begin{align}
    J_{1} \leq & \frac{1}{n(n-1)}\left(\sum_{i,j}a^{2}_{ij}\right)^2h(t) + \frac{2t}{n^2}\sum_{i\neq j}\sum_{k\neq l}a^{3}_{ik}a^{2}_{jl}h(t) + C\left(n\delta^4+n^2\delta^6t\right)h(t).\label{J1-bound-final}
\end{align}
Next we consider $J_{2}$, for a fix index $k\in[n]$, we define $W_n^{(k)}$ which is close to $W_{n}$ as follows
\begin{align}
    W_{n}^{(k)} = \sum_{\substack{i = 1\\ i\neq k}}^{n}a_{i\pi(i)} + \sum_{\substack{i\neq j\\ i,j\neq k}}b_{ij\pi(i)\pi(j)}.\notag
\end{align}
Together with (\ref{equ-boundedness-condition-DIPS}), we easily obtain $|V^{\prime}| = |W_n - W_{n}^{(k)}|\leq 3\delta$. Then we do the Taylor expansion of $\Psi_{t}(x)$ at the point $W_{n}^{(k)}$, it follows that
\begin{align}
    J_{2} = & \frac{2}{n^2}\sum_{i,j}\sum_{k,l}a^{2}_{ij}a^{2}_{kl}\mathbb{E}\{\Psi_{t}(W_n)\mid\pi(k) = l\}\notag\\
    = & \frac{2}{n^2}\sum_{i,j}\sum_{k,l}a^{2}_{ij}a^{2}_{kl}\mathbb{E}\{\Psi_{t}(W_{n}^{(k)})\mid\pi(k) = l\}\notag\\
    &+ \frac{2t}{n^2}\sum_{i,j}\sum_{k,l}a^{2}_{ij}a^{2}_{kl}\mathbb{E}\{V^{\prime}\Psi_{t}(W_{n}^{(k)})\mid\pi(k) = l\}\notag\\
    &+ \frac{2t^2}{n^2}\sum_{i,j}\sum_{k,l}a^{2}_{ij}a^{2}_{kl}\mathbb{E}\{V^{\prime 2}\Psi_{t}(W_{n}^{(k)}+UV^{\prime})(1-U)\mid\pi(k) = l\}\notag\\
    := & J_{21} + J_{22} + J_{23}.\label{J2-decompose}
\end{align}
We have divided $J_{2}$ into three parts, and next, we will consider these three parts respectively. We first consider $J_{21}$, applying add and subtract technique, we have
\begin{align}
    J_{21} \geq & \frac{2}{n^2}\left(\sum_{i,j}a^{2}_{ij}\right)^2h(t) -  \frac{2}{n^2}\sum_{i,j}\sum_{k,l}a^{2}_{ij}a^{2}_{kl}\left|\mathbb{E}\{\Psi_{t}(W_{n}^{(k)})\mid\pi(k) = l\} - h(t)\right|. \notag
\end{align}
Then we define $\sigma^{k}_{l}$ and $S^{(k)}_{\sigma^{k}_{l}}$ as 
\begin{align}
    \sigma^{k}_{l} = \begin{cases}
        \sigma, & \sigma(k) = l,\\
        \sigma\circ\tau_{k,\pi^{-1}(l)}, & \sigma(k) \neq l,
    \end{cases}\notag
\end{align}
\begin{align}
    S^{(k)}_{\sigma^{k}_{l}} = \sum_{\substack{i = 1\\ i\neq k}}^{n}a_{i\sigma^{k}_{l}(i)} + \sum_{\substack{i\neq j\\ i,j\neq k}}b_{ij\sigma^{k}_{l}(i)\sigma^{k}_{l}(j)}.\notag
\end{align}
Under Lemma \ref{lem-same-distribution-of-permutation}, we obtain $\mathcal{L}(S^{(k)}_{\sigma^{k}_{l}}) \overset{d}{=}\mathcal{L}(W_n^{(k)}\mid\pi(k) = l)$. Under condition (\ref{equ-boundedness-condition-DIPS}), we have $|W_n-S^{(k)}_{\sigma^{k}_{l}}| \leq C\delta$. Applying (\ref{general-final-bound-PsiWn-i1-ik-ht}) we have for any $k,l\in[n]$,
\begin{align}
    \left|\mathbb{E}\{\Psi_{t}(W_{n}^{(k)})\mid\pi(k) = l\} - h(t)\right| \leq C(\frac{1}{n} + t^2\delta^2)h(t),\label{inequ:bound-of-difference-PsiWnk-ht-final}
\end{align}
Therefore, we obtain the lower bound of $J_{21}$ as
\begin{align}
    J_{21} \geq \frac{2}{n^2}\left(\sum_{i,j}a^{2}_{ij}\right)^2h(t) -  C(n\delta^4 + n^2\delta^6t^2)h(t),\label{J21-bound-final}
\end{align}
Next we consider $J_{22}$, since $V^{\prime} = a_{k\pi(k)} + \sum_{\substack{s = 1\\s \neq k}}^{n}b_{sk\pi(s)\pi(k)}+b_{ks\pi(k)\pi(s)}$, we expand $J_{22}$ by definition of $V^{\prime}$ as
\begin{align}
    J_{22} = & \frac{2t}{n^2}\sum_{i,j}\sum_{k,l}a^{2}_{ij}a^{3}_{kl}\mathbb{E}\left\{\Psi_{t}(W_{n}^{(k)})\middle|\pi(k) = l\right\}\notag\\
    &+\frac{2t}{n^2}\sum_{i,j}\sum_{k,l}a^{2}_{ij}a^{2}_{kl}\left(\frac{1}{n-1}\sum_{\substack{s = 1\\s\neq k}}^{n}\sum_{\substack{t =1\\t\neq l}}^{n}(b_{sktl}+b_{kslt})\mathbb{E}\left\{\Psi_{t}(W_{n}^{(k)})\middle|\substack{\pi(k) = l\\\pi(s) = t}\right\}\right),\notag\\
    \geq & \frac{2t}{n^2}\sum_{i\neq j}\sum_{k\neq l}a^{3}_{ik}a^{2}_{jl}h(t) - n^2\delta^4\max_{k,l}\left|\mathbb{E}\left\{\Psi_{t}(W_{n}^{(k)})\middle|\pi(k) = l\right\}-h(t)\right|\notag\\
    &-2n\delta^4\left(h(t)+\max_{k\neq s,l\neq t}\left|\mathbb{E}\left\{\Psi_{t}(W_{n}^{(k)})\middle|\substack{\pi(k) = l\\\pi(s) = t}\right\}-h(t)\right|\right),\notag
\end{align}
together with (\ref{inequ:bound-of-difference-PsiWnij-ht-final}) and (\ref{inequ:bound-of-difference-PsiWnk-ht-final}), it follows that
\begin{align}
    J_{22} \geq \frac{2t}{n^2}\sum_{i\neq j}\sum_{k\neq l}a^{3}_{ik}a^{2}_{jl}h(t) - C(n\delta^4 + n^2\delta^6t^2)h(t).\label{J22-bound-final}
\end{align}
Then we consider the lower bound of $J_{13}$, by (\ref{inequ:bound-of-difference-PsiWnk-ht-final}) we bound $J_{23}$ as
\begin{align}
    J_{23} = & \frac{2t^2}{n^2}\sum_{i,j}\sum_{k,l}a^{2}_{ij}a^{2}_{kl}\mathbb{E}\{V^{\prime 2}\Psi_{t}(W_{n}^{(k)}+UV^{\prime})(1-U)\mid\pi(k) = l\} \notag\\
    \geq & -\frac{2t^2}{n^2}\sum_{i,j}\sum_{k,l}a^{2}_{ij}a^{2}_{kl}\delta^2e^{t\delta}\left(h(t) + \left|\mathbb{E}\left\{\Psi_{t}(W_{n}^{(k)})\middle|\pi(k) = l\right\}-h(t)\right| \right) \notag\\
    \geq &  -Cn^2\delta^6t^2 h(t).\label{J23-bound-final}
\end{align}
Combining (\ref{J2-decompose}) and (\ref{J21-bound-final})-(\ref{J23-bound-final}), we deduce that
\begin{align}
    J_{2} \geq & \frac{2}{n^2}\left(\sum_{i,j}a^{2}_{ij}\right)^2h(t) + \frac{2t}{n^2}\sum_{i\neq j}\sum_{k\neq l}a^{3}_{ik}a^{2}_{jl}h(t) - Cn^2\delta^6t^2 h(t).\label{J2-bound-final}
\end{align}
Together with (\ref{equ-a-square-form-decompose}), (\ref{J1-bound-final}) and (\ref{J2-bound-final}), we complete the proof (\ref{inequ:final-bound-a-square-form}). By a same argument, we can also proof (\ref{inequ:final-bound-aij-square-form}).
Then we consider (\ref{inequ:final-bound-b-square-form}), The proof approaches of (\ref{inequ:final-bound-a-square-form}) and (\ref{inequ:final-bound-b-square-form}) are similar, but the property of $\{b(i,j,k,l)\}_{i,j,k,l\in[n]}$ is repeatedly utilized. We first decompose the left hand side of (\ref{inequ:final-bound-b-square-form}) as
\begin{align}
    & \mathbb{E}\left\{\left(\sum_{i\neq j}b^{2}_{ij\pi(i)\pi(j)}-\mathbb{E}\left\{\sum_{i\neq j}b^{2}_{ij\pi(i)\pi(j)}\right\}\right)^2e^{tW_n}\right\} \notag\\
    =&\frac{1}{n^2(n-1)^2}\left(\sum_{i\neq j}\sum_{k\neq l}b^{2}_{ijkl}\right)^2h(t) + \mathbb{E}\left\{\sum_{i\neq j}\sum_{p\neq q}b^{2}_{ij\pi(i)\pi(j)}b^{2}_{pq\pi(p)\pi(q)}\Psi_{t}(W_n)\right\}\notag\\
    &  - \frac{2}{n(n-1)}\sum_{i\neq j}\sum_{k\neq l}b^{2}_{ijkl}\mathbb{E}\left\{\sum_{p\neq q}b^{2}_{pq\pi(p)\pi(q)}\Psi_{t}(W_n)\right\}\notag\\
    :=& \frac{1}{n^2(n-1)^2}\left(\sum_{i\neq j}\sum_{k\neq l}b^{2}_{ijkl}\right)^2h(t) + Q_1  - Q_2.\label{equ:b-square-form-decompose}
\end{align}
We estimate $Q_1$ and $Q_2$ respectively. We first consider $Q_1$ and decompose it into three parts
\begin{align}
    Q_1 = & \mathbb{E}\left\{\sum_{i\neq j\neq p\neq q}b^{2}_{ij\pi(i)\pi(j)}b^{2}_{pq\pi(p)\pi(q)}\Psi_{t}(W_n)\right\} \notag\\
    &+ \mathbb{E}\left\{\sum_{i\neq j}\sum_{\substack{p = 1\\p\notin\{i,j\}}}^{n}\left(b^{2}_{ij\pi(i)\pi(j)}b^{2}_{pi\pi(p)\pi(i)}+b^{2}_{ij\pi(i)\pi(j)}b^{2}_{pj\pi(p)\pi(j)}\right)\Psi_{t}(W_n)\right\}\notag\\
    & + \mathbb{E}\left\{\sum_{i\neq j}\left(\sum_{\substack{q = 1\\q\neq i}}^{n}b^{2}_{ij\pi(i)\pi(j)}b^{2}_{iq\pi(i)\pi(q)}+\sum_{\substack{q = 1\\q\neq j}}^{n}b^{2}_{ij\pi(i)\pi(j)}b^{2}_{jq\pi(j)\pi(q)}\right)\Psi_{t}(W_n)\right\}\notag\\
    :=& Q_{11} + Q_{12} + Q_{13}.\label{Q1-decompose}
\end{align}
For $Q_{11}$, by doing the Taylor expansion, we decompose it into three parts
\begin{align}
    Q_{11} = & \frac{(n-4)!}{n!}\sum_{i\neq j\neq p\neq q}\sum_{k\neq l\neq u\neq v}b^{2}_{ijkl}b^{2}_{pquv}\mathbb{E}\left\{\Psi_{t}(W_n)\middle|\substack{\pi(i)=k,\pi(j)=l\\\pi(p)=u,\pi(q)=v}\right\}\notag\\
    = & \frac{(n-4)!}{n!}\sum_{i\neq j\neq p\neq q}\sum_{k\neq l\neq u\neq v}b^{2}_{ijkl}b^{2}_{pquv}\mathbb{E}\left\{\Psi_{t}(W_{n}^{(i,j,p,q)})\middle|\substack{\pi(i)=k,\pi(j)=l\\\pi(p)=u,\pi(q)=v}\right\}\notag\\
    & + \frac{(n-4)!t}{n!}\sum_{i\neq j\neq p\neq q}\sum_{k\neq l\neq u\neq v}b^{2}_{ijkl}b^{2}_{pquv}\mathbb{E}\left\{V_{ijpq}\Psi_{t}(W_{n}^{(i,j,p,q)})\middle|\substack{\pi(i)=k,\pi(j)=l\\\pi(p)=u,\pi(q)=v}\right\}\notag\\
    & + \frac{(n-4)!t^2}{n!}\sum_{i\neq j\neq p\neq q}\sum_{k\neq l\neq u\neq v}b^{2}_{ijkl}b^{2}_{pquv}\mathbb{E}\left\{V_{ilpq}^{2}\Psi_{t}(W_{n}^{(i,j,p,q)}+UV_{ijpq})(1-U)\middle|\substack{\pi(i)=k,\pi(j)=l\\\pi(p)=u,\pi(q)=v}\right\}\notag\\
    := & Q_{111} + Q_{112} + Q_{113}.\label{Q11-decompose}
\end{align}
We first consider $Q_{111}$. Applying (\ref{general-final-bound-PsiWn-i1-ik-ht}), for any fixed index $i,j,p,q\in[n]$ are all distinct, and $k,l,u,v\in[n]$ are all distinct, it follows that
\begin{align}
    \left|\mathbb{E}\left\{\Psi_{t}(W_{n}^{(i,j,p,q)})\middle|\substack{\pi(i)=k,\pi(j)=l\\\pi(p)=u,\pi(q)=v}\right\} - h(t)\right| \leq C\left(\frac{1}{n} +\delta^2t^2\right)h(t).\label{inequ:bound-of-difference-PsiWnijpq-ht-final}
\end{align}
Together with condition (\ref{equ-boundedness-condition-DIPS}), we deduce that
\begin{align}
    Q_{111} \leq & \frac{1}{n^2(n-1)^2}\left(\sum_{i\neq j}\sum_{k\neq l}b^{2}_{ijkl}\right)^2h(t) + \frac{C}{n^4}\left(\sum_{i, j}\sum_{k, l}b^{2}_{ijkl}\right)^2\left(\frac{1}{n} +\delta^2t^2\right)h(t)\notag\\
    \leq & \frac{1}{n^2(n-1)^2}\left(\sum_{i\neq j}\sum_{k\neq l}b^{2}_{ijkl}\right)^2h(t) + C\left(n\delta^4+n^2\delta^6t^2\right)h(t).\label{Q111-bound-final}
\end{align}
As for $Q_{112}$, since
\begin{align}
    V_{ijpq} = & a_{i\pi(i)} + a_{j\pi(j)} + a_{p\pi(p)} + a_{q\pi(q)} + \sum_{\substack{s = 1\\s\neq i}}^{n}(b_{si\pi(s)\pi(i)}+b_{is\pi(i)\pi(s)})\notag\\
    & +\sum_{\substack{s = 1\\s\notin \{i,j\}}}^{n}(b_{sj\pi(s)\pi(j)}+b_{js\pi(j)\pi(s)}) + \sum_{\substack{s = 1\\s\notin \{i,j,p\}}}^{n}(b_{sp\pi(s)\pi(p)}+b_{ps\pi(p)\pi(s)})\notag\\
    &+ \sum_{\substack{s = 1\\s\notin \{i,j,p,q\}}}^{n}(b_{sq\pi(s)\pi(q)}+b_{qs\pi(q)\pi(s)}),\notag
\end{align}
we expand $Q_{112}$ by its definition and using (\ref{equ-boundedness-condition-DIPS}), (\ref{general-final-bound-PsiWn-i1-ik-ht}). It then follows that
\begin{align}
    Q_{112} \leq & \frac{2t}{n^2(n-1)^2}\sum_{i\neq j\neq p\neq q}\sum_{k\neq l\neq u\neq v}b^{2}_{ijkl}b^{2}_{pquv}(a_{pu}+a_{qv})h(t) + \frac{C}{n^4}\left(\sum_{i,j}\sum_{k,l}b^{2}_{ijkl}\right)^2\left(\frac{1}{n}+\delta^2t^2\right)h(t)\notag\\
    \leq & \frac{2t}{n^2(n-1)^2}\sum_{i\neq j\neq p\neq q}\sum_{k\neq l\neq u\neq v}b^{2}_{ijkl}b^{2}_{pquv}(a_{pu}+a_{qv})h(t) + C\left(n\delta^4+n^2\delta^6t^2\right)h(t).\label{Q112-bound-final} 
\end{align}
Next we consider $Q_{113}$, by (\ref{equ-boundedness-condition-DIPS}) and (\ref{general-final-bound-PsiWn-i1-ik-ht}), we obtain
\begin{align}
    Q_{113} \leq & \frac{2t^2\delta^2}{n^2(n-1)^2}\left(\sum_{i,j}\sum_{k,l}b^{2}_{ijkl}\right)^2h(t) + \frac{C}{n^4}\left(\sum_{i,j}\sum_{k,l}b^{2}_{ijkl}\right)^2\left(\frac{1}{n}+\delta^2t^2\right)h(t)\notag\\
    \leq & Cn^2\delta^6t^2 h(t).\label{Q113-bound-final}
\end{align}
Combining (\ref{Q11-decompose}), (\ref{Q111-bound-final})-(\ref{Q113-bound-final}), we deduce that
\begin{align}
    Q_{11} \leq & \frac{1}{n^2(n-1)^2}\left(\sum_{i\neq j}\sum_{k\neq l}b^{2}_{ijkl}\right)^2h(t) + \frac{2t}{n^2(n-1)^2}\sum_{i\neq j\neq p\neq q}\sum_{k\neq l\neq u\neq v}b^{2}_{ijkl}b^{2}_{pquv}(a_{pu}+a_{qv})h(t)\notag\\
    & +C\left(n\delta^4+n^2\delta^6t^2\right)h(t).\label{Q11-bound-final}
\end{align}
By the same argument of $Q_{11}$, we have
\begin{align}
    Q_{12} \leq C\left(n\delta^4+n^2\delta^6t^2\right)h(t),\quad Q_{13} \leq C\left(n\delta^4+n^2\delta^6t^2\right)h(t).\label{Q12-Q13-bound-final}
\end{align}
Together with (\ref{Q1-decompose}), (\ref{Q11-bound-final}) and (\ref{Q12-Q13-bound-final}), it follows that
\begin{align}
    Q_{1} \leq & \frac{1}{n^2(n-1)^2}\left(\sum_{i\neq j}\sum_{k\neq l}b^{2}_{ijkl}\right)^2h(t) + \frac{2t}{n^2(n-1)^2}\sum_{i\neq j\neq p\neq q}\sum_{k\neq l\neq u\neq v}b^{2}_{ijkl}b^{2}_{pquv}(a_{pu}+a_{qv})h(t)\notag\\
    & +C\left(n\delta^4+n^2\delta^6t^2\right)h(t).\label{Q1-bound-final}
\end{align}
For $Q_2$, by Taylor expansion we decompose it into three parts
\begin{align}
    Q_2 = &\frac{2}{n^2(n-1)^2}\sum_{i\neq j}\sum_{p\neq q}\sum_{k\neq l}\sum_{u\neq v}b^{2}_{ijkl}b^{2}_{pquv}\mathbb{E}\left\{\Psi_{t}(W_{n}^{(p,q)})\middle|\substack{\pi(p) = u\\\pi(q) = v}\right\}\notag\\
    &+ \frac{2t}{n^2(n-1)^2}\sum_{i\neq j}\sum_{p\neq q}\sum_{k\neq l}\sum_{u\neq v}b^{2}_{ijkl}b^{2}_{pquv}\mathbb{E}\left\{V_{pq}\Psi_{t}(W_{n}^{(p,q)})\middle|\substack{\pi(p) = u\\\pi(q) = v}\right\}\notag\\
    &+ \frac{2t^2}{n^2(n-1)^2}\sum_{i\neq j}\sum_{p\neq q}\sum_{k\neq l}\sum_{u\neq v}b^{2}_{ijkl}b^{2}_{pquv}\mathbb{E}\left\{V^{2}_{pq}\Psi_{t}(W_{n}^{(p,q)}+UV_{pq}(1-U))\middle|\substack{\pi(p) = u\\\pi(q) = v}\right\}\notag\\
    := & Q_{21} + Q_{22} + Q_{23}.\label{Q2-decompose}
\end{align} 
where $U$ is a uniform random variable on $[0,1]$ and independent of any other random variables and
\begin{align}
    V_{pq} = a_{p\pi(p)}+a_{q\pi(q)}+ \sum_{\substack{i = 1\\i\neq p}}^{n}(b_{pi\pi(p)\pi(i)}+b_{ip\pi(i)\pi(p)})+\sum_{\substack{i = 1\\i\notin \{p,q\}}}^{n}(b_{iq\pi(i)\pi(q)}+b_{qi\pi(q)\pi(i)}).\notag
\end{align}
Using add and subtract technique and applying (\ref{general-final-bound-PsiWn-i1-ik-ht}), (\ref{equ-boundedness-condition-DIPS}), we obtain the lower bound of $Q_{21}$ as
\begin{align}
    Q_{21} \geq & \frac{2}{n^2(n-1)^2}\left(\sum_{i\neq j}\sum_{k\neq l}b^{2}_{ijkl}\right)^2h(t) - \frac{C}{n^4}\left(\sum_{i, j}\sum_{k, l}b^{2}_{ijkl}\right)^2\left(\frac{1}{n} +\delta^2t^2\right)h(t)\notag\\
    \geq & \frac{2}{n^2(n-1)^2}\left(\sum_{i\neq j}\sum_{k\neq l}b^{2}_{ijkl}\right)^2h(t) - C(n\delta^4+n^2\delta^6t^2)h(t).\label{Q21-bound-final}
\end{align}
Next for $Q_{22}$, we expand it by the definiton of $V_{pq}$ as
\begin{align}
    Q_{22} = & \frac{2t}{n^2(n-1)^2}\sum_{i\neq j}\sum_{p\neq q}\sum_{k\neq l}\sum_{u\neq v}b^{2}_{ijkl}b^{2}_{pquv}(a_{pu}+a_{qv})\mathbb{E}\left\{\Psi_{t}(W_{n}^{(p,q)})\middle|\substack{\pi(p) = u\\\pi(q = v)}\right\}\notag\\
    &+ \frac{2t}{n^2(n-1)^2}\sum_{i\neq j}\sum_{p\neq q}\sum_{k\neq l}\sum_{u\neq v}b^{2}_{ijkl}b^{2}_{pquv}\mathbb{E}\left\{\left(\sum_{\substack{s = 1\\s\neq p}}^{n}(b_{ps\pi(p)\pi(s)}+b_{sp\pi(s)\pi(p)})\right)\Psi_{t}(W_{n}^{(p,q)})\middle|\substack{\pi(p) = u\\\pi(q = v)}\right\},\notag\\
    &+ \frac{2t}{n^2(n-1)^2}\sum_{i\neq j}\sum_{p\neq q}\sum_{k\neq l}\sum_{u\neq v}b^{2}_{ijkl}b^{2}_{pquv}\mathbb{E}\left\{\left(\sum_{\substack{s = 1\\s\notin \{p,q\}}}^{n}(b_{sq\pi(s)\pi(q)}+b_{qs\pi(q)\pi(s)})\right)\Psi_{t}(W_{n}^{(p,q)})\middle|\substack{\pi(p) = u\\\pi(q = v)}\right\}.\notag
\end{align}
Since
\begin{align}
    &\sum_{i\neq j}\sum_{p\neq q}\sum_{k\neq l}\sum_{u\neq v}b^{2}_{ijkl}b^{2}_{pquv}(a_{pu}+a_{qv})\notag\\
    = & \sum_{i\neq j\neq p\neq q}\sum_{k\neq l\neq u\neq v}b^{2}_{ijkl}b^{2}_{pquv}(a_{pu}+a_{qv})\notag\\
    &+\sum_{i\neq j}\sum_{p=1}^{n}\sum_{k\neq l}\sum_{u\neq v}\left(b^{2}_{ijkl}b^{2}_{piuv}(a_{pu}+a_{iv})+b^{2}_{ijkl}b^{2}_{pjuv}(a_{pu}+a_{jv})\right)\notag\\
    &+\sum_{i\neq j}\sum_{\substack{q = 1\\q\notin \{i,j\}}}^{n}\sum_{k\neq l}\sum_{u\neq v}\left(b^{2}_{ijkl}b^{2}_{iquv}(a_{iu}+a_{qv})+ b^{2}_{ijkl}b^{2}_{jquv}(a_{ju}+a_{qv})\right)\notag\\
    &+\sum_{i\neq j\neq p\neq q}\sum_{k\neq l}\sum_{u = 1}^{n}\left(b^{2}_{ijkl}b^{2}_{pquk}(a_{pu}+a_{qk})+b^{2}_{ijkl}b^{2}_{pqul}(a_{pu}+a_{ql})\right)\notag\\
    &+\sum_{i\neq j\neq p\neq q}\sum_{k\neq l}\sum_{\substack{v = 1\\v\notin \{k,l\}}}^{n}\left(b^{2}_{ijkl}b^{2}_{pqkv}(a_{pk}+a_{qv})+b^{2}_{ijkl}b^{2}_{pqlv}(a_{pl}+a_{qv})\right),\notag
\end{align}
together with (\ref{equ-boundedness-condition-DIPS}) and (\ref{general-final-bound-PsiWn-i1-ik-ht}), it follows that
\begin{align}
    Q_{22} \geq & \frac{2t}{n^2(n-1)^2}\sum_{i\neq j\neq p\neq q}\sum_{k\neq l\neq u\neq v}b^{2}_{ijkl}b^{2}_{pquv}(a_{pu}+a_{qv})h(t) - n\delta^4h(t)\notag\\
    & - \frac{C}{n^4}\left(\sum_{i,j}\sum_{k,l}b^{2}_{ijkl}\right)^2\left(\frac{1}{n}+\delta^2t^2\right)h(t)\notag\\
    \geq & \frac{2t}{n^2(n-1)^2}\sum_{i\neq j\neq p\neq q}\sum_{k\neq l\neq u\neq v}b^{2}_{ijkl}b^{2}_{pquv}(a_{pu}+a_{qv})h(t) - C(n\delta^4+n^2\delta^6t^2)h(t).\label{Q22-bound-final}
\end{align}
Then, we consider $Q_{23}$, also by (\ref{equ-boundedness-condition-DIPS}) and (\ref{general-final-bound-PsiWn-i1-ik-ht}), we obtain
\begin{align}
    Q_{23} \geq & -\frac{C}{n^4}\left(\sum_{i,j}\sum_{k,l}b^{2}_{ijkl}\right)^2\delta^2t^2 h(t)\geq -Cn^2\delta^6t^2 h(t).\label{Q23-bound-final}
\end{align}
From (\ref{Q2-decompose})-(\ref{Q23-bound-final}), we have a lower bound of $Q_2$ as
\begin{align}
    Q_2 \geq & \frac{2}{n^2(n-1)^2}\left(\sum_{i\neq j}\sum_{k\neq l}b^{2}_{ijkl}\right)^2h(t) + \frac{2t}{n^2(n-1)^2}\sum_{i\neq j\neq p\neq q}\sum_{k\neq l\neq u\neq v}b^{2}_{ijkl}b^{2}_{pquv}(a_{pu}+a_{qv})h(t) \notag\\
    &  -C(n\delta^4+n^2\delta^6t^2)h(t).\label{Q2-bound-final} 
\end{align}
Combining (\ref{equ:b-square-form-decompose}), (\ref{Q1-bound-final}) and (\ref{Q2-bound-final}), we complete the proof of (\ref{inequ:final-bound-b-square-form}). By a same argument, we can also proof (\ref{inequ:final-bound-bjiij-square-form}).
\end{proof}

\begin{proof}[Proof of Lemma \ref{lem-same-distribution-of-permutation}]


To get a Berry-Esseen bound for Combinatorial Central Limit Theorems, a transformation
was constructed by \cite{Goldstein_2005}, and further applied by \cite{chen2015onthe/13-BEJ569} and \cite{Liu2023cramer/23-AAP1931} to prove Berry-Esseen bound and Cram\'{e}r type moderate deviation results for combinatorial central limit theorems. Our transformation (\ref{def-transformation-sigma-prime}) is a bit different from theirs, and we use a similar train of thought from Proof of Lemma 4.5. in \cite{chen2010normal} to prove Lemma \ref{lem-same-distribution-of-permutation}. We only prove (\ref{same-distribution-of-permutation-double-transformation}), since (\ref{same-distribution-of-permutation-single-transformation}) can be proved similarly. 

Let $A_1,A_2,A_3,A_4$ denote the four cased of $(\ref{def-transformation-sigma-prime})$ in their respective order. Let $p_{m}, m\notin\{i,j\}$ be distinct and satisfy $p_{m}\notin \{k,l\}$. Under $A_1$ we have $\sigma(j) \neq k$ and $i\neq \sigma^{-1}(k)$. Hence $\sigma^{-1}(k)\notin\{i,j\}$, then we have
\begin{align}
    &P(\sigma^{ij}_{kl}(m) = p_{m}, m\notin\{i,j\}, A_{1})\notag\\
    =& P(\sigma^{ij}_{kl}(m) = p_{m}, m\notin\{i,j,\sigma^{-1}(k)\}, \sigma(i) = l,\sigma(j)\neq k,\sigma^{ij}_{kl}(\sigma^{-1}(k)) = p_{\sigma^{-1}(k)})\notag\\
    = & P(\sigma^{ij}_{kl}(m) = p_{m}, m\notin\{i,j,\sigma^{-1}(k)\}, \sigma(i) = l,\sigma(j) =  p_{\sigma^{-1}(k)})\notag\\
    = & \sum_{q\notin\{i,j\}}P(\sigma^{ij}_{kl}(m) = p_{m}, m\notin\{i,j,q\}, \sigma(i) = l,\sigma(j) =  p_{q}, \sigma(q) = k)\notag\\
    = & \sum_{q\notin\{i,j\}}P(\sigma(m) = p_{m}, m\notin\{i,j,q\}, \sigma(i) = l,\sigma(j) =  p_{q}, \sigma(q) = k)\notag\\
    = & \frac{n-2}{n!}.\notag
\end{align}
Case $A_{2}$ can be calculated similarly upon interchanging the roles of $i$ and $j$, and $k$ and $l$. So we obtain
\begin{align}
    P(\sigma^{ij}_{kl}(m) = p_{m}, m\notin\{i,j\}, A_{1}\cup A_{2}) = \frac{2(n-2)}{n!}.\notag
\end{align}
Under $A_3$, we have $\sigma(i) = l$ and $\sigma(j) = k$, therefore 
\begin{align}
    &P(\sigma^{ij}_{kl}(m) = p_{m}, m\notin\{i,j\}, A_{3})\notag\\
    = &P(\sigma^{ij}_{kl}(m) = p_{m}, m\notin\{i,j\}, \sigma(i) = l, \sigma(j) = k)\notag\\
    = & P(\sigma(m) = p_{m}, m\notin\{i,j\}, \sigma(i) = l, \sigma(j) = k)\notag\\
    = & \frac{1}{n!}.\notag    
\end{align}
Finally, under $A_4$, we divide $A_{4}$ into subcases depending on  $R = \left|\{\sigma(i),\sigma(j)\}\cap\{k,l\}\right|$, and let $A_{4r} = A_{4}\cap \{R = r\}$ for $r = 0,1,2$. If $R = 0$, applying (4.131) in \cite{chen2010normal}, we have
\begin{align}
    P(\sigma^{ij}_{kl}(m) = p_{m}, m\notin\{i,j\}, A_{40}) = \frac{(n-2)(n-3)}{n!}.\notag
\end{align}
Considering $R = 1$, by (4.132) in \cite{chen2010normal},
\begin{align}
    P(\sigma^{ij}_{kl}(m) = p_{m}, m\notin\{i,j\}, A_{41}) = \frac{2(n-2)}{n!}.\notag
\end{align}
Finally, if $R = 2$, we have $A_{42} = \{\pi(i) = k, \pi(j) = l\}$. So by interchanging $i$ and $j$ in the calculation of $A_3$, it follows that
\begin{align}
    P(\sigma^{ij}_{kl}(m) = p_{m}, m\notin\{i,j\}, A_{42}) = \frac{1}{n!}.\notag
\end{align}
Summing over all the cases, we obtain
\begin{align}
    P(\sigma^{ij}_{kl}(m) = p_{m}, m\notin\{i,j\}) = \frac{(n-2)(n-3)}{n!} + \frac{4(n-2)}{n!} + \frac{2}{n!} = \frac{1}{(n-2)!}.\label{sigma-ij-kl-all-case-probability}
\end{align}
This shows that $\sigma^{ij}_{kl}$ is uniformly distributed over the set of permutations $\tau$ such that $\tau(i) = k$ and $\tau(j) = l$. Hence
\begin{align}
    \mathcal{L}(\sigma^{ij}_{kl}) \overset{d}{=}\mathcal{L}(\pi\mid\substack{\pi(i) = k\\ \pi(j) = l}).\notag
\end{align}
Thus we complete the proof of (\ref{same-distribution-of-permutation-double-transformation}) and by a same argument we can proof (\ref{same-distribution-of-permutation-single-transformation}). Then, we prove $\mathcal{L}\left(\mathcal{P}^{pq}_{uv}\sigma^{ij}_{kl}\right)\overset{d}{=}\mathcal{L}\left(\pi\middle|\substack{\pi(i) = k\\\pi(j) = l\\\pi(p) = u\\\pi(q) = v}\right)$ in (\ref{composite-transformation}) based on the result of (\ref{same-distribution-of-permutation-double-transformation}). In the following tables, we show the values of permutation $\sigma^{ij}_{kl}$ and $\mathcal{P}^{pq}_{uv}\sigma^{ij}_{kl}$ on index $i,j,\sigma^{-1}(k),\sigma^{-1}(l)$ in several cases.
\begin{table}[H]
    \centering
    \begin{tabular}{c|cc|cc|cc|ccc|ccc}
        \toprule
        case & \multicolumn{2}{c}{for all} & \multicolumn{2}{c}{$C_{1}$} & \multicolumn{2}{c}{$C_{2}$} & \multicolumn{3}{c}{$C_{3}$} & \multicolumn{3}{c}{$C_{4}$}  \\
        \cmidrule{1-13}
        index & $i$ & $j$ & $p$ & $q$ & $p$ & $q$ & $p$ & $q$ & $\sigma^{-1}(v)$ & $p$ & $q$ & $\sigma^{-1}(u)$ \\
        \cmidrule{1-13}
        $\sigma^{ij}_{kl}(\text{index})$ & $k$ & $l$ & $u$ & $v$ & $v$ & $u$ & $u$ & $\sigma(q)$ & $v$ & $\sigma(p)$ & $v$ & $u$ \\
        \cmidrule{1-13}
        $\mathcal{P}^{pq}_{uv}\sigma^{ij}_{kl}(\text{index})$ & $k$ & $l$ & $u$ & $v$ & $u$ & $v$ & $u$ & $v$ & $\sigma(q)$ & $u$ & $v$ & $\sigma(p)$ \\
        \bottomrule
    \end{tabular}
\end{table}
\begin{table}[H]
    \centering
    \begin{tabular}{c|ccc|ccc|cccc}
        \toprule
        case & \multicolumn{3}{c}{$C_{5}$} & \multicolumn{3}{c}{$C_{6}$} & \multicolumn{4}{c}{$C_{7}$} \\
        \cmidrule{1-11}
        index & $p$ & $q$ & $\sigma^{-1}(v)$ & $p$ & $q$ & $\sigma^{-1}(u)$ & $p$ & $q$ & $\sigma^{-1}(u)$ & $\sigma^{-1}(v)$ \\
        \cmidrule{1-11}
        $\sigma^{ij}_{kl}(\text{index})$ & $\sigma(p)$ & $u$ & $v$ & $v$ & $\sigma(q)$ & $u$ & $\sigma(p)$ & $\sigma(q)$ & $u$ & $v$ \\
        \cmidrule{1-11}
        $\mathcal{P}^{pq}_{uv}\sigma^{ij}_{kl}(\text{index})$ & $u$ & $v$ & $\sigma(p)$ & $u$ & $v$ & $\sigma(q)$ & $u$ & $v$ & $\sigma(p)$ & $\sigma(q)$ \\
        \bottomrule
    \end{tabular}
\end{table}
Since $\Omega  = \cup_{i = 1}^{7}C_{i}$ and $C_{i}, i = 1,\dots,7$ are disjoint events, therefore we have
\begin{align}
    P\left(\mathcal{P}^{pq}_{uv}\sigma^{ij}_{kl}(m) = t_{m}, m\notin\{i,j,p,q\}\right) = \sum_{i = 1}^{7}P\left(\mathcal{P}^{pq}_{uv}\sigma^{ij}_{kl}(m) = t_{m}, m\notin\{i,j,p,q\}, C_{i}\right),\label{C1-C7-decomposition}
\end{align}
where $t_{m}, m\notin\{i,j,p,q\}$ are distinct and satisfy $t_{m}\notin\{k,l,u,v\}$. Then we calculate (\ref{C1-C7-decomposition}) case by case. Recalling (\ref{sigma-ij-kl-all-case-probability}), under $C_{1}$ it follows that
\begin{align}
    P\left(\mathcal{P}^{pq}_{uv}\sigma^{ij}_{kl}(m) = t_{m}, m\notin\{i,j,p,q\}, C_{1}\right) = & P\left(\sigma^{ij}_{kl}(m) = t_{m}, m\notin\{i,j,p,q\}, \substack{\sigma^{ij}_{kl}(p) = u\\ \sigma^{ij}_{kl}(q) = v} \right) = \frac{1}{(n-2)!}.\label{probability-under-C1}
\end{align}
By interchanging $p$ and $q$, we have under $C_{2}$,
\begin{align}
    P\left(\mathcal{P}^{pq}_{uv}\sigma^{ij}_{kl}(m) = t_{m}, m\notin\{i,j,p,q\}, C_{2}\right) = \frac{1}{(n-2)!}. \label{probability-under-C2}
\end{align}
Then we calculate the probability under $C_{3}$ and $C_{5}$ as
\begin{align}
    &P\left(\mathcal{P}^{pq}_{uv}\sigma^{ij}_{kl}(m) = t_{m}, m\notin\{i,j,p,q\}, C_{3}\right)\notag\\
    = & \sum_{r\notin\{i,j,p,q\}}P\left(\sigma^{ij}_{kl}(m) = t_{m},m\notin\{i,j,p,q,r\}, \substack{\sigma^{ij}_{kl}(p) = u\\ \sigma^{ij}_{kl}(q) = t_{r}\\ \sigma^{ij}_{kl}(r) = u}  \right) =  \frac{n-4}{(n-2)!},\label{probability-under-C3}
\end{align}
and
\begin{align}
    &P\left(\mathcal{P}^{pq}_{uv}\sigma^{ij}_{kl}(m) = t_{m}, m\notin\{i,j,p,q\}, C_{5}\right)\notag\\
    = & \sum_{r\notin\{i,j,p,q\}}P\left(\sigma^{ij}_{kl}(m) = t_{m},m\notin\{i,j,p,q,r\}, \substack{\sigma^{ij}_{kl}(q) = u,\\ \sigma^{ij}_{kl}(p) =t_{r}\\ \sigma^{ij}_{kl}(r) = v} \right) = \frac{n-4}{(n-2)!}.\label{probability-under-C5}
\end{align}
Case $C_{4}$ and $C_{6}$ can be calculated similarly upon interchanging the roles of $p$ and $q$, and $u$ and $v$. So we obtain
\begin{align}
    P\left(\mathcal{P}^{pq}_{uv}\sigma^{ij}_{kl}(m) = t_{m}, m\notin\{i,j,p,q\}, C_{4}\right) = \frac{n-4}{(n-2)!},\label{probability-under-C4}\\
    P\left(\mathcal{P}^{pq}_{uv}\sigma^{ij}_{kl}(m) = t_{m}, m\notin\{i,j,p,q\}, C_{6}\right) = \frac{n-4}{(n-2)!}.\label{probability-under-C6}
\end{align}
Then we calculate the probability under $C_{7}$ as follows
\begin{align}
    &P\left(\mathcal{P}^{pq}_{uv}\sigma^{ij}_{kl}(m) = t_{m}, m\notin\{i,j,p,q\}, C_{7}\right)\notag\\
    = & \sum_{\substack{r,s\notin\{i,j,p,q\}\\ r\neq s}}P\left(\sigma^{ij}_{pq}(m) = t_{m},m\notin\{i,j,p,q,r,s\}, \substack{\sigma^{ij}_{kl}(p) = t_{r}, \sigma^{ij}_{kl}(q) = t_{s}\\\sigma^{ij}_{kl}(r) = u, \sigma^{ij}_{kl}(s) = v} \right) = \frac{(n-4)(n-5)}{(n-2)!}.\label{probability-under-C7}
\end{align}
Combining (\ref{C1-C7-decomposition})-(\ref{probability-under-C7}), we deduce that
\begin{align}
    P\left(\mathcal{P}^{pq}_{uv}\sigma^{ij}_{kl}(m) = t_{m}, m\notin\{i,j,p,q\}\right)
    = & \frac{2}{(n-2)!} + \frac{4(n-4)}{(n-2)!} + \frac{(n-4)(n-5)}{(n-2)!} = \frac{1}{(n-4)!}.\notag
\end{align}
It implies that
\begin{align}
    \mathcal{L}\left(\mathcal{P}^{pq}_{uv}\sigma^{ij}_{kl}\right) \overset{d}{=}\mathcal{L}\left(\pi\middle|\substack{\pi(i) = k\\ \pi(j) = l\\pi(p) = u\\\pi(q) =v}\right).\notag
\end{align}

Next we prove $\mathcal{L}\left(\mathcal{P}^{p}_{q}\sigma^{ij}_{kl}\right) \overset{d}{=}\mathcal{L}\left(\pi\middle|\substack{\pi(i) = k\\\pi(j) = l\\\pi(p) = q}\right)$ in (\ref{composite-transformation}). We define $q_m$, $m\notin\{i,j,p\}$ are distinct and satisfy $q_m\notin\{k,l,q\}$. Then by (\ref{sigma-ij-kl-all-case-probability}), we have
\begin{align}
    P\left(\mathcal{P}^{p}_{q}\sigma^{ij}_{kl}(m) = q_{m}, m\notin\{i,j,p\}, \sigma^{ij}_{kl}(p) = q\right) = P\left(\sigma^{ij}_{kl}(m) = p_{m}, m\notin\{i,j,p\}, \sigma^{ij}_{kl}(p) = q\right) = \frac{1}{(n-2)!},\notag
\end{align}
and
\begin{align}
    &P\left(\mathcal{P}^{p}_{q}\sigma^{ij}_{kl}(m) = q_{m}, m\notin\{i,j,p\}, \sigma^{ij}_{kl}(p) \neq q\right)\notag\\
    = &\sum_{r\notin\{i,j,p\}}P\left(\sigma^{ij}_{kl}(m) = p_{m}, m\notin\{i,j,p,r\}, \sigma^{ij}_{kl}(r) = q, \sigma^{ij}_{kl}(p) = q_{r}\right) = \frac{n-3}{(n-2)!}.\notag 
\end{align}
Therefore,
\begin{align}
    P\left(\mathcal{P}^{p}_{q}\sigma^{ij}_{kl}(m) = q_{m}, m\notin\{i,j,p\}\right) =\frac{1}{(n-2)!} + \frac{n-3}{(n-2)!} = \frac{1}{(n-3)!}.\notag
\end{align}
It implies that
\begin{align}
    \mathcal{L}\left(\mathcal{P}^{p}_{q}\sigma^{ij}_{kl}\right) \overset{d}{=}\mathcal{L}\left(\pi\middle|\substack{\pi(i) = k\\ \pi(j) = l\\ \pi(p) = q}\right).\notag
\end{align}
Thus, we complete the proof of Lemma \ref{lem-same-distribution-of-permutation}.
\end{proof}

\begin{proof}[Proof of Lemma \ref{bound-H1-H4-general}]
Define $\sigma^{i_{1},\dots,i_{k}}_{l_{1},\dots,l_{k}}$ and $S_{\sigma^{i_{1},\dots,i_{k}}_{l_{1},\dots,l_{k}}}$ as
\begin{align}
    \sigma^{i_{1},\dots,i_{k}}_{l_{1},\dots,l_{k}} = \begin{cases}
        \mathcal{P}^{i_{k},i_{k-1}}_{l_{k},l_{k-1}}\circ \cdots\circ\mathcal{P}^{i_{1},i_{2}}_{l_{1},l_{2}}\circ\sigma, & k \text{ is even},\\
        \mathcal{P}^{i_{k}}_{l_{k}}\circ\mathcal{P}^{i_{k-1},i_{k-2}}_{l_{k-1},l_{k-2}}\circ \cdots\circ\mathcal{P}^{i_{1},i_{2}}_{l_{1},l_{2}}\circ\sigma, & k \text{ is odd},
    \end{cases}\notag
\end{align}
and
\begin{align}
    S_{\sigma^{i_{1},\dots,i_{k}}_{l_{1},\dots,l_{k}}} = & \sum_{p  =1}^{n}a_{p\sigma^{i_{1},\dots,i_{k}}_{l_{1},\dots,l_{k}}(p)} + \sum_{p\neq q}b_{pq\sigma^{i_{1},\dots,i_{k}}_{l_{1},\dots,l_{k}}(p)\sigma^{i_{1},\dots,i_{k}}_{l_{1},\dots,l_{k}}(q)},\notag
\end{align}
where $\sigma$ is a random permutation chosen uniformly from $S_{n}$ and independent of $\pi$, $\mathcal{P}_{k}^{i}$ and $\mathcal{P}_{kl}^{ij}$ are defined as (\ref{def-transformation-single}) and (\ref{def-transformation-sigma-prime}). Since $i_{1},\dots,i_{k}\in[n]$ are all distinct, and $l_{1},\dots,l_{k}\in[n]$ are also all distinct, when $k$ is even, by applying (\ref{same-distribution-of-permutation-double-transformation}) in Lemma \ref{lem-same-distribution-of-permutation} $k/2$ times, we have
\begin{align}
    \mathcal{L}\left(\sigma^{i_{1},\dots,i_{k}}_{l_{1},\dots,l_{k}}\right) \overset{d}{=} \mathcal{L}\left(\pi\middle|\substack{\pi(i_{1}) = l_{1}\\...\\\pi(i_{k}) = l_{k}}\right),\label{general-same-distribution-of-permutation-k}
\end{align}
when $k$ is odd, by applying (\ref{same-distribution-of-permutation-single-transformation}) once and (\ref{same-distribution-of-permutation-double-transformation}) $(k-1)/2$ times, we also have 
(\ref{general-same-distribution-of-permutation-k}). Then it follows by definition that
\begin{align}
    \mathcal{L}\left(S_{\sigma^{i_{1},\dots,i_{k}}_{l_{1},\dots,l_{k}}}\right) \overset{d}{=} \mathcal{L}\left(W_n\middle|\substack{\pi(i_{1}) = l_{1}\\...\\\pi(i_{k}) = l_{k}}\right).\label{general-same-distribution-of-estimator-k}
\end{align}
Denote $ \left|\mathbb{E}\left\{\Psi_{t}(W_{n})\middle|\substack{\pi(i_{1}) = l_{1}\\...\\\pi(i_{k}) = l_{k}}\right\} - h(t)\right|$ as $H_{\substack{i_{1},\dots,i_{k}\\ l_{1},\dots,l_{k} }}$, by (\ref{general-same-distribution-of-estimator-k}) it follows that
\begin{align}
    H_{\substack{i_{1},\dots,i_{k}\\ l_{1},\dots,l_{k} }} = & \left|\mathbb{E}\left\{\Psi_{t}\left(S_{\sigma^{i_{1},\dots,i_{k}}_{l_{1},\dots,l_{k}}}\right)\right\} - h(t)\right|.\notag
\end{align}
Then we do the Taylor expansion of $e^{tx}$ at $x = T_{n} = \sum_{i = 1}^{n}a_{i\sigma(i)} + \sum_{i\neq j}b_{ij\sigma(i)\sigma(j)}$, we obtain
\begin{align}
    \Psi_{t}(S_{\sigma^{i_{1},\dots,i_{k}}_{l_{1},\dots,l_{k}}}) = \Psi_{t}(T_{n}) + (S_{\sigma^{i_{1},\dots,i_{k}}_{l_{1},\dots,l_{k}}} - T_{n})\mathbb{E}\left\{\Psi_{t}^{\prime}\left(T_{n}+U\left(S_{\sigma^{i_{1},\dots,i_{k}}_{l_{1},\dots,l_{k}}} - T_{n}\right)\right)\right\}, \notag
\end{align}
where $U$ is a $U[0,1]$ random variable independent of any other random variables, therefore
\begin{align}
    H_{\substack{i_{1},\dots,i_{k}\\ l_{1},\dots,l_{k} }} = & \left|\mathbb{E}\left\{\Psi_{t}(T_{n})\right\} - h(t)\right| + t\left|\mathbb{E}\left\{V\Psi_{t}(T_{n}+UV)\right\}\right| = t\left|\mathbb{E}\left\{V\Psi_{t}(T_{n}+UV)\right\}\right|,\notag
\end{align} 
where $V = S_{\sigma^{i_{1},\dots,i_{k}}_{l_{1},\dots,l_{k}}} - T_{n}$. By the definition of $S_{\sigma^{i_{1},\dots,i_{k}}_{l_{1},\dots,l_{k}}}^{(i_{1},\dots,i_{k})}$ and $T_{n}$ and the condition (\ref{equ-boundedness-condition-DIPS}), we have
\begin{align}
    |V| \leq  & \sum_{m = 1}^{k}\left|a_{i_{m}\sigma(i_{m})}\right| + \sum_{m = 1}^{k}\left|a_{\sigma^{-1}(l_{m})l_{m}}\right| + \sum_{m =1}^{k}\left|a_{i_{m}\sigma^{(i_{1},\dots,i_{k})}_{l_{1},\dots,l_{k}}(i_{m})}\right|+ \sum_{m = 1}^{k}\left|a_{\left(\sigma^{(i_{1},\dots,i_{k})}_{l_{1},\dots,l_{k}}\right)^{-1}(l_{m})l_{m}}\right|\notag\\
    & + \sum_{m = 1}^{k}\sum_{s\notin\{i_{1},\dots,i_{m}\}}\left|b_{is\sigma(i)\sigma(s)}+b_{si\sigma(s)\sigma(i)}\right|\notag\\
    & + \sum_{m = 1}^{k}\sum_{s\notin\{i_{1},\dots,i_{m}\}}\left|b_{is\sigma^{(i_{1},\dots,i_{k})}_{l_{1},\dots,l_{k}}(i)\sigma^{(i_{1},\dots,i_{k})}_{l_{1},\dots,l_{k}}(s)}+b_{si\sigma^{(i_{1},\dots,i_{k})}_{l_{1},\dots,l_{k}}(s)\sigma^{(i_{1},\dots,i_{k})}_{l_{1},\dots,l_{k}}(i)}\right| \notag\\ 
    & + \sum_{m = 1}^{k}\sum_{s\notin\{i_{1},\dots,i_{k},\sigma^{-1}(l_{1}),\dots,\sigma^{-1}(l_{m})\}}\left|b_{s\sigma^{-1}(l_{m})\sigma(s)l_{m}}+b_{\sigma^{-1}(l_{m})sl_{m}\sigma(s)}\right|\notag\\
    & + \sum_{m = 1}^{k}\sum_{s\notin\{i_{1},\dots,i_{k},\sigma^{-1}(l_{1}),\dots,\sigma^{-1}(l_{m})\}}\left|b_{s\sigma^{-1}(l_{m})\sigma^{(i_{1},\dots,i_{k})}_{l_{1},\dots,l_{k}}(s)\sigma^{(i_{1},\dots,i_{k})}_{l_{1},\dots,l_{k}}\left(\sigma^{-1}(l_{m})\right)}\right|\notag\\
    & + \sum_{m = 1}^{k}\sum_{s\notin\{i_{1},\dots,i_{k},\sigma^{-1}(l_{1}),\dots,\sigma^{-1}(l_{m})\}}\left|b_{\sigma^{-1}(l_{m})s\sigma^{(i_{1},\dots,i_{k})}_{l_{1},\dots,l_{k}}\left(\sigma^{-1}(l_{m})\right)\sigma^{(i_{1},\dots,i_{k})}_{l_{1},\dots,l_{k}}(s)}\right|\notag\\
    \leq & Ck\delta.\notag
\end{align}
Recalling $0<t<1/\delta$, it follows that
\begin{align}
    H_{\substack{i_{1},\dots,i_{k}\\ l_{1},\dots,l_{k} }} \leq & Cke^{t\delta k}t\delta h(t) \leq Cke^{k}t\delta h(t).\label{lem-bound-H1-H4-general-bound-1}
\end{align}
This completes the proof of Lemma \ref{bound-H1-H4-general}.

\end{proof}

\begin{acks}[Acknowledgments]
    Liu S.H. was partially supported by the Fundamental Research Funds for the Central Universities DUT25RC(3)133.
\end{acks}

\bibliographystyle{apalike}
\bibliography{refs.bib}

\end{document}